\documentclass[11pt,reqno]{amsart}
\usepackage[utf8]{inputenc}
\usepackage{graphicx}
\usepackage{amsfonts}
\usepackage{amsmath}
\usepackage{mathrsfs}
\usepackage{nccmath}
\usepackage{amssymb}
\usepackage{fancyhdr}
\usepackage{amsthm}
\usepackage{array}
\usepackage{systeme}
\usepackage{xcolor}
\usepackage{mathtools}

\usepackage{afterpage}
\usepackage{calligra}
\usepackage[T1]{fontenc}
\usepackage{esint}
\let\oldemptyset\emptyset
\let\emptyset\varnothing
\usepackage{enumerate}
\usepackage[top = 1in,
bottom = 1in, left = 1in, right = 1in]{geometry}
\usepackage[unicode, pagebackref,colorlinks=true,linkcolor=blue,citecolor=blue]{hyperref}
\usepackage{subfig}
\usepackage[]{mdframed}
\usepackage{lipsum}
\usepackage{enumitem}
\usepackage{booktabs}
\usepackage{longtable}
\usepackage{tikz}

\renewcommand{\d}{\mathrm{d}}
\newcommand{\cA}{\mathcal{A}}
\newcommand{\cB}{\mathcal{B}}

\newcommand{\cD}{\mathcal{D}}
\newcommand{\cE}{\mathcal{E}}
\newcommand{\cF}{\mathcal{F}}

\newcommand{\cH}{\mathcal{H}}
\newcommand{\cI}{\mathcal{I}}

\newcommand{\cK}{\mathcal{K}}
\newcommand{\cL}{\mathcal{L}}
\newcommand{\cM}{\mathcal{M}}
\newcommand{\cN}{\mathcal{N}}

\newcommand{\cP}{\mathcal{P}}

\newcommand{\cS}{\mathcal{S}}
\newcommand{\cT}{\mathcal{T}}

\newcommand{\cV}{\mathcal{V}}
\newcommand{\cW}{\mathcal{W}}
\newcommand{\cX}{\mathcal{X}}
\newcommand{\cY}{\mathcal{Y}}
\newcommand{\cZ}{\mathcal{Z}}

\newcommand{\bF}{\mathbb{F}}

\newcommand{\bK}{\mathbb{K}}
\newcommand{\bM}{\mathbb{M}}
\newcommand{\bN}{\mathbb{N}}

\renewcommand{\P}{\mathbb{P}}
\newcommand{\PP}{\mathbf{P}}
\newcommand{\Q}{\mathbb{Q}}

\newcommand{\bS}{\mathbb{S}}

\newcommand{\bV}{\mathbb{V}}
\newcommand{\bX}{\mathbb{X}}

\newcommand{\hmu}{\hat\mu}
\newcommand{\hnu}{\hat\nu}

\newcommand{\hQ}{\hat{\Q}}

\newcommand{\hQn}{\hat{\Q}_n}

\newcommand{\om}{\omega}

\newcommand{\deq}{\coloneqq}

\newcommand{\wto}{\overset{w}{\to}}

\newcommand{\1}{{\rm 1}\kern-0.24em{\rm I}}

\newcommand{\R}{\mathbb{R}}
\newcommand{\p}{\mathbb{P}}
\newcommand{\E}{\mathbb{E}}

\newcommand{\pn}{\p_{\kern-0.25em n}}
\newcommand{\pnm}{\p_{\kern-0.25em n,m}}
\newcommand{\psubm}{\p_{\kern-0.25em m}}
\newcommand{\psubp}{\p_{\kern-0.25em p}}
\newcommand{\cfi}{\cF_{\kern-0.25em \infty}}

\newcommand{\Lip}{\mathrm{{Lip}}}
\newcommand{\Var}{\mathrm{Var}}

\newtheorem{theorem}{Theorem}[section]

\newtheorem{remark}[theorem]{Remark}

\newtheorem{proposition}[theorem]{Proposition}
\newtheorem{corollary}[theorem]{Corollary}
\newtheorem{lemma}[theorem]{Lemma}

\newtheorem{definition}[theorem]{Definition}

\renewcommand{\tilde}{\widetilde}

\newlength{\minipagewidth}
\usepackage{chngcntr}
\counterwithin{figure}{section}

\DeclareMathOperator{\divergence}{div}

\DeclareMathOperator{\KL}{\mathsf{KL}}
\newcommand{\dd}{d}

\newcommand{\gradW}{\nabla\mkern-10mu\nabla}

\newcommand{\id}{\mathrm{id}}

\renewcommand{\emptyset}{\varnothing}

\renewcommand{\Delta}{\varDelta}
\renewcommand{\Phi}{\varPhi}
\renewcommand{\Sigma}{\varSigma}
\renewcommand{\Theta}{\varTheta}
\renewcommand{\Upsilon}{\varUpsilon}

\newcommand{\LW}{L^2\big(\Omega; (\mathcal{P}_{2}(\R^d), W_2)\big)}
\newcommand{\WL}{L^2_{W}(\R^d)}
\newcommand{\WWL}{L^2_{\mathrm{WoW}}(\R^d)}
\newcommand{\dw}{d_{\mathrm{WoW}}}
\newcommand{\wlac}{L^2_{W,ac}(\R^d)}

\newcommand{\wlsac}{L^2_{W,ac}(\bS^{d-1})}
\newcommand{\grad}{\nabla}
\newcommand{\Pac}{\cP_{2,ac}(\R^d)}
\newcommand{\gradWL}{\gradW_{L^2}}
\newcommand{\mF}{\mathfrak{F}}
\newcommand{\mE}{\mathfrak{E}}

\newcommand{\nv}{\nabla\varphi}
\newcommand{\Pc}{\cP_2(\R^d)}
\newcommand{\Pcpc}{\cP_2\big(\cP_2(\R^d)\big)}
\newcommand{\Pkpk}{\cP_2(\cP_2(\cK))}
\newcommand{\ninf}{n\to\infty}

\newcommand{\hmun}{\hmu_n}
\newcommand{\hmuno}{\hmu_{\om,n}}
\newcommand{\hnun}{\hnu_n}
\newcommand{\hnuno}{\hnu_{\om,n}}
\newcommand{\Prior}{\Pi}
\newcommand{\Postn}{\Pi_n}

\newcommand{\disp}{\mathrm{disp}}

\def\l{\left}
\def\r{\right}

\renewcommand*{\backrefalt}[4]{%
    \ifcase #1 %
        (Not cited.)%
    \or
        (Cited on p. #2)%
    \else
        (Cited on p. #2)%
    \fi
}
\newcommand{\proj}{{\mathbf{P}}^\perp}

\numberwithin{equation}{section}

\title{Wasserstein random variables: Statistical analysis for optimal transport}
\author{Riccardo Passeggeri}
\author{Rohan M. Shenoy}
\author{Pengcheng Ye}

\thanks{Imperial College London, Department of Mathematics.\\ \indent Emails: riccardo.passeggeri@imperial.ac.uk, rohan.shenoy22@imperial.ac.uk, david.ye23@imperial.ac.uk}

\date{\today}

\begin{document}

\maketitle

\begin{abstract}
    Optimal transport provides an inherently geometric and highly structured framework for studying spaces of probability measures, supplying a rich theoretical toolkit for contemporary statistics, machine learning, and generative modelling. In applications, however, the measures of interest are almost never known precisely, calling for a theory of optimal transport that accounts for statistical uncertainty. We construct such a framework, lifting the classical theory to the setting of \textit{Wasserstein random variables}. We introduce the $L^2$ over Wasserstein space of random measures, establishing that it inherits the formal Riemannian structure of the Wasserstein space by characterising distances and geodesic geometry. The structure induces random flows with Wasserstein gradient flow sample paths, making it the natural extension for random gradient flow dynamics. We ensemble statistical consistency results for the optimal transport machinery using the empirical measure within the framework. We produce a Bayesian analysis of optimal transport, deducing posterior consistency of the same tools in both the parametric and non-parametric settings with explicit posterior contraction rates. We demonstrate that the representations for various generative models can be embedded into the framework, allowing for a statistical analysis of these models. The results provide a unifying statistical treatment of optimal transport and its consequences for principled inference under the statistical uncertainty of random sampling.\newline

\noindent \textbf{Keywords:} Optimal transport, random measures, Wasserstein gradient flows, Wasserstein random variables, empirical measures, Bayesian non-parametrics, generative modelling.
\end{abstract}
\maketitle

\tableofcontents

\addtocontents{toc}{\protect\setcounter{tocdepth}{1}}
\section{Introduction}
Optimal transport has been enormously successful in applications, as its geometric framework for studying spaces of probability measures provides a rich theoretical toolkit for contemporary statistics, machine learning, and generative modelling. In practice, however, the measures of interest are almost never known precisely. This calls for a theory of optimal transport which accounts for \textit{statistical uncertainty}. The aim of this paper is to develop such a theory.

In particular, this work introduces Wasserstein random variables, lifting the classical theory of optimal transport to random probability measures, and enabling measure-valued statistics. We demonstrate the utility of the space to discuss an empirical and Bayesian analysis of optimal transport. These perspectives allow us to determine statistical consistency results for various generative models which make use of the available optimal transport machinery. We consolidate both new and existing results under one unified framework for optimal transport under uncertainty.

In what follows, we present the main contributions of the paper and discuss how they relate to the existing literature. To this end, we follow the structure of the paper.

\subsection{Deterministic optimal transport}

In Section \ref{sec:2} we provide an overview of the Optimal Transport focusing on the results and concepts needed for the paper. The origin of optimal transport due to Monge \cite{monge1781memoire} finds a volume-preserving mapping to transport mass from a starting distribution $\mu$ to a terminal distribution $\nu$ whilst minimising the transport cost. The contemporary formulation due to Kantorovich \cite{Kantorovich1942} replaces the search for a deterministic transport map $T:\cX\to\cY$ with the more flexible notion of a coupling, a joint distribution over $\cX\times\cY$ with marginals $\mu$ and $\nu$. The resulting Wasserstein distance $W_2(\mu,\nu)$, obtained by optimising the expected squared displacement over all such couplings, endows the space of probability measures $\Pc$ with a rich geometry. The study of Wasserstein over Wasserstein spaces due to \cite{nguyen2016borrowing} has been of interest more recently in \cite{pinzi2025nestedsuperpositionprinciplerandom}, \cite{pinzi}, \cite{beiglboeck}, iterating layers of Wasserstein spaces of Wasserstein spaces.

Brenier’s fundamental theorem \cite{brenier1991polar} delivers a unique optimal transport map realised as the gradient of a convex potential in the absolutely continuous case. The Benamou–Brenier theorem from \cite{benamou2000computational} identifies the alternative dynamic formulation of the Wasserstein distance as the minimised kinetic energy over all curves of measures subject to the continuity equation, inducing a geodesic geometry characterised by displacement interpolations. This underpins a formal Riemannian structure on Wasserstein space introduced by Otto \cite{otto2001geometry}, enabling the study of gradient flows \cite{ambrosioGradientFlowsMetric2008}. The gradient flow theory superposed to the Wasserstein over Wasserstein level is explored by Pinzi and Savaré \cite{pinzi2025nestedsuperpositionprinciplerandom}.

\subsection{Wasserstein random variables}

The central contribution of this paper is to provide a unifying statistical framework for optimal transport and its applications. Towards this goal, we develop a theory for $L^2$ Wasserstein-valued random variables, allowing one to discuss optimal transport statistics using these variables and, in doing so, their applications for various generative models.

In the language of Kallenberg \cite{kallenberg2017random}, a random measure $\xi$ is a measure-valued random variable. We adapt this notion to the optimal transport paradigm, carefully defining the $L^2$ over Wasserstein space of Wasserstein random variables $\WL$ as the $L^2$ space of probability measures endowed with the Wasserstein topology, naturally equipping the space with the $L^2$ metric $d$,
\begin{equation}\label{eq:metric_d}
    \WL \deq L^2\l(\Omega\,;\, (\Pc, W_2)\r),\qquad d(\xi,\eta)\deq \E_\om\l[W_2^2\l(\xi_\om,\eta_\om\r)\r]^{\frac{1}{2}}.
\end{equation}
The structure induced by this metric topology is rich. Within the framework, we are able to lift the geodesic geometry of $\Pc$ to $\WL$. Constant-speed geodesics are obtained with a Benamou-Brenier displacement interpolation on the sample paths, characterising $d^2(M_0, M_1)$ as the infimum of expected kinetic energy over random sample paths subject to the continuity equation. 

The framework also gives rise to a formal Riemannian structure over the space of Wasserstein random variables. The tangent space is identified with the space of $L^2$ random vectors whose samples are tangent vectors to the Wasserstein space; these are endowed with the Wasserstein Dirichlet form inner product from \cite{vonrenesse2009}, the kinetic energy of the random tangent vector over Wasserstein space. In the context of Wasserstein gradient flows over $\Pc$, this definition induces a random gradient flow with Wasserstein gradient flow sample paths, thus the space we define can be considered the natural extension of the Wasserstein space which allows for random gradient flow dynamics.

We additionally define the superposed space of $L^2$ Wasserstein over Wasserstein random variables, providing the structure to discuss uncertainty for non-parametric models over Wasserstein space.

\subsection{Statistical analysis with empirical measures}

Section \ref{sec:4} demonstrates the utility of the new definitions by consolidating a statistical theory for optimal transport between empirical measures. Given sample draws from a population measure $\mu$, we view the empirical measure $\hmun$ as a Wasserstein random variable in $\WL$, conditioned on the random data sequence $X_{1:n}\sim\mu$. 

We characterise a Wasserstein Law of Large Numbers, a classical result embedded into the language from our framework, with convergence rates provided by a theorem from Fournier and Guillin \cite{fournier2015rate}. Concerning the dynamical theory, we pass the Law of Large Numbers to both the convergence of empirical Wasserstein geodesics, and also the empirical Wasserstein gradient flow. For the geodesics, the convergence of the endpoints determines the convergence of the interpolants. Respectively for the gradient flow, we provide a propagation of chaos result for the empirical flow, demonstrating uniform convergence of flow paths initialised at the empirical measure to the gradient flow initialised at the population measure $\mu$. Given the characterisation of $L^2$ Wasserstein over Wasserstein random variables, we also discuss a hierarchical empirical scheme, with latent, unobserved $\mu_1,\dots,\mu_J\sim\Q\in\Pcpc$, and observed sampled points in $\R^d$, $X_{j,1}\dots, X_{j,n_j}\sim\mu_j\in\Pc$.
We provide statistical consistency results for the empirical measure on Wasserstein over Wasserstein space to approximate the hierarchical population measure $\Q\in \Pcpc$.

The identification of the space of Wasserstein random variables allows us to unify both new and existing empirical process results within a single, canonical framework to study frequentist optimal transport statistics.

\subsection{Bayesian analysis}

In Section \ref{sec:5} we provide a unifying formulation for a Bayesian study of optimal transport. Many examples of existing research combining variational inference with optimal transport focus on Bayesian robustness \cite{micheli2025wasserstein_dro} and approximate Bayesian computation as in \cite{bernton2019approximate_abc}, \cite{deshpande2019approximate_abc}, \cite{li2026optimaltransportbasedgenerativemodel_abc}, \cite{mallasto2020bayesianinferenceoptimaltransport_abc} among other works. However, many of these works \textit{fix the sample data}, treating the posterior as a deterministic measure. To target uncertainty quantification, one assumes the posterior is itself a random measure, where the data are not known and are themselves modelled as random processes, similar to the empirical measure. This is the natural Bayesian view.

More concretely, the data-generating distribution $\pi\in\cP(\R^d)$ is unknown. We encode our prior belief about the population measure $\pi$ with a prior measure $\Prior$, and update the belief after new observations $X_{1:n}\sim\pi$ to form a posterior $\Postn(\cdot)\deq\Prior(\cdot|X_{1:n})$ so that the posterior is a random probability measure, conditional on the observed data. 
Working at the parametric level, $\Postn\in\WL$, given that the posterior is a randomly chosen measure over a parameter space $\Theta\subset\R^d$. In the non-parametric setting, $\Postn\in\WWL$, given that the posterior is a randomly chosen measure then over a parameter space $\Theta\subset\cP(\R^d)$.

We embed the notion of Bayesian consistency, the phenomenon of the posterior mass concentrating about the truth $\pi$, into our statistical framework for optimal transport - we target the phenomenon of \textit{Wasserstein} consistency. We detail how existing parametric results, such as the classical Bernstein-von Mises phenomenon, can be adapted into the language and structure of our framework using integrability assumptions. These allow us to provide a Bayesian analysis for Wasserstein geodesics and gradient flows which propagate uncertainty estimates.

At the non-parametric level, we provide sufficient conditions for Wasserstein consistency in the language of our Section \ref{sec:3}. These take the form of the conditions of Schwartz's theorem for the weak topology, and a uniform integrability criterion, commensurate with the fact that the Wasserstein distance metrizes weak convergence over compact spaces. This allows us to adapt existing Bayesian consistency literature into the optimal transport paradigm. Under our framework, we are able to provide \textit{explicit non-parametric posterior contraction rates}, which generalises the convergence theory of Fournier and Guillin \cite{fournier2015rate}. We are able to provide a new set of sufficient assumptions, which allow for a more general and interpretable formulation of posterior contraction, with strictly \textit{sharper contraction rates} compared to Camerlenghi et al. \cite{Camerlenghietal2026} for certain model classes including Dirichlet process priors. We also show that the formulation naturally generalises to the mixture model setting.

Allowing for uncertainty about some oracle non-parametric model $\P\in\Pcpc$, we work at the further superposed level of $
\Pcpc$-valued random variables in $\WWL$.
Under appropriate conditions, we show consistency in the hierarchical case, demonstrating that as the number of $\R^d$-valued observation groups $X_{j,i}$ increases, the hierarchical posterior $\cH(\cdot|X_{1:n})\in\cP_2\l(\Pcpc\r)$ concentrates its mass around the non-parametric truth $\P$.

Under the framework of Section \ref{sec:3}, we unify both new and existing results for a Bayesian analysis of optimal transport where the Bayesian uncertainty is structurally encoded in our identification of the $L^2$ space of Wasserstein random variables.

\subsection{Applications}

Section \ref{sec:3} provides a framework of measure-valued random variables, inviting Sections \ref{sec:4} and \ref{sec:5} to provide statistical consistency results for the optimal transport machinery. In Section \ref{sec:6}, we provide some examples of statistical consistency for generative models making use of this machinery.

Bonet et al. \cite{bonet2025flowing} present a Wasserstein over Wasserstein representation for labelled datasets, using a class-conditional distribution of features, modelling the dataset itself as a mixture distribution supported on those classes. The authors make use of the differential structure to design gradient flows for the target applications. When the classes are themselves unknown, this is exactly the hierarchical regime detailed in Sections \ref{subsec:4:hierarchical} and \ref{subsec:5:consistency}; we demonstrate statistical consistency of their hierarchical representation which can then confidently be passed to the gradient flow theory the authors explore for domain adaptation, transfer learning and domain distillation.

Another natural application of our framework is Conditional Optimal Transport (COT), explored by Chemseddine et al. in \cite{chem}, who consider transport between two data conditional probability measures. We apply the formulation of the random tangent space of Section \ref{subsec:3:RGFs} to the COT flow matching problem, showing that the solution is an orthogonal projection onto our random tangent space, hence permitting access to our statistical framework for the source and target measures considered.

The topic of transformer self-attention flows has recently exploded in interest. The interplay between its non-local interaction and spherical geometry produces rich collective dynamics such as clustering and synchronisation studied in great depth by Rigollet \cite{rigollet2026meanfielddynamicstransformers}, and collaborators Geshkovski et al. \cite{koubbi2026homogenizedtransformers},
and Bruno et al. \cite{bruno2026a}. Classical gradient flow theory applies given a fixed token initialisation. We  consider a random token initialisation, using either the empirical measure or, more generally, a parametric posterior law. The random gradient flow this defines provides the general setting to explore statistical aspects of the paths, embedding Agazzi et al.'s \cite{agazzi2026} propagation of chaos result into our framework.

\subsection*{Related works and literature}

For the classical study of (deterministic) optimal transport, the reader is directed to Villani's pair of monographs \cite{villani_topics_2003}, \cite{villani2009optimal} which explore the metric topology of general Wasserstein spaces in great detail. Ambrosio et al. \cite{ambrosioGradientFlowsMetric2008} provide a comprehensive review of dynamic optimal transport theory with the formal Riemannian structure to define gradient flows in Wasserstein space.

Our Section \ref{sec:3} explores the $L^2$ space of Wasserstein random variables, combining literature from random measures and $L^2$ metric spaces with optimal transport. A comprehensive review of random measures is presented by \cite{kallenberg2017random}. 

We have mentioned the closely related recent work of Pinzi and Savaré's on Wasserstein over Wasserstein (WoW) distances, \cite{pinzi2025nestedsuperpositionprinciplerandom}, \cite{pinzi}. The key difference in approach is that \textit{our framework operates directly on random measures themselves}, rather than on their laws. The WoW distance between random measure laws can be expressed using our definitions via optimal $\WL$ couplings. The authors characterise nested superposition principles, first superposing the dynamics over $\R^d$ to dynamics of measures on $\Pc$, then further superposing to dynamics of random measure laws on $\Pcpc$, detailing a theory for WoW gradient flows. As analogy, our framework defines a single (stochastic) superposition from $L^2(\R^d)$ to $\WL$. Beiglböck et al. \cite{beiglboeck} additionally provide results for general nested Wasserstein spaces on $\cP_2(\cP_2(\dots(\R^d)\dots))$, focusing on Brenier-type results.

Concerning the random dynamic theory, the inner product imposed on the tangent space to $\WL$ is equivalent to the \textit{Wasserstein Dirichlet form} from von Renesse and Sturm in \cite{vonrenesse2009}. The authors establish the existence of a unique, continuous, symmetric Markov process on $\cP(\bS^{d-1})$, applying standard Dirichlet form theory from \cite{Fukushima2010}. Subsequent contributions \cite{Schiavo2018ART} and \cite{delloschiavo2024massive}, similarly construct diffusion processes over Wasserstein spaces. Our resulting differential structure, which arises naturally from our definitions and the resulting dynamic equivalence of the $L^2$ distance, coincides with these works - this solidifies our structure within the general setting of random Wasserstein dynamics. While the aforementioned works focus on randomness introduced by Wasserstein diffusion processes, primarily motivated by the statistical analysis, we are interested instead in random gradient flows driven by a prescribed functional over Wasserstein space.

Adjacent to our structure are adapted Wasserstein distances which address optimal transport between stochastic processes under bi-causal constraints. This application of causal transport \cite{lassalle2018causal} is explored by Acciaio et al. in \cite{acciaio2025absolutelycontinuouscurvesstochastic}. Their work explores a slightly different structural concern. Adapted transport between two process sample paths is still a deterministic transport, though with a restricted set of available couplings to respect the pair of filtrations. In contrast, we consider a random transport between two random processes fixed in time, the space of accessible couplings unrestricted.

Fournier and Guillin \cite{fournier2015rate} provide consistency rates for the empirical distribution to the population, a result translated into our definitions, and a result that we generalise to the Bayesian posterior in Section \ref{subsec:5:pcr}. Central limit type results are provided by del Barrio and Loubes \cite{Barrio2017CentralLT}. Goldfeld et al. \cite{golfeldGreenewaldPolyanskiyNilesWeed} detail a mollified, absolutely continuous empirical measure to circumvent the issue of passing discrete empirical measures through a gradient flow. The more general concept of Wasserstein statistics via random probability measures is explored in Panaretos and Zemel's book \cite{Panaretos2020AnIT}, in particular Section 2.4 where the authors focus on the measurability of optimal transport maps. This is primarily in service of discussing Wasserstein Fréchet means; in contrast, we look to discuss not only random distances, but random dynamics such as geodesics and gradient flows which leverage the specified $L^2$ structure.

We have mentioned the viewpoint of the posterior law as a random measure, aligning with the general Bayesian paradigm explored in the books of Ghosal and van der Vaart \cite{Ghosal_van_der_Vaart_2017}, and Gelman et al. \cite{gelman2013bayesian}. Random probability measures have long played a central role in Bayesian analysis. Their connection to exchangeability goes back to de Finetti\cite{deFinetti} and Savage \cite{savage}. Ferguson \cite{Ferguson1973} introduced the Dirichlet process as a prior on an unknown probability distribution, Regazzini, Lijoi, and Prünster\cite{Regazzini2003} developed a distributional theory for normalised random measures, MacEachern\cite{MacEachern} developed the dependent Dirichlet processes, hierarchical mixture models that share components across groups were studied in Teh et al.~\cite{TehEtAl2006}, and Camerlenghi et al.~\cite{CamerlenghiEtAl2019} characterised the partially exchangeable partitions induced by hierarchical normalised random measures. Of course, this list is far from complete. Use of OT in Bayesian analysis has been recently investigated, we just cite a few works. Catalano and Lavenant \cite{catalano2023merging} develop a natural non-parametric framework to investigate posterior merging rates using Dirichlet process priors, according to their dual Wasserstein over bounded Lipschitz distance, treating posterior updates as a stochastic process. The work makes use of similar structure to Pinzi and Savaré's, operating at the level of random measure laws, as opposed to random probability measures directly. 

Dolera et al.~\cite{DoleraFavaroMainini2024} and Camerlenghi et al.~\cite{Camerlenghietal2026} target non-parametric posterior contraction rates for the $L^1$ of the $p$-Wasserstein distance, providing sufficient conditions for the convergence. Our framework considers the $L^p$ which bounds the former above. We provide alternative sufficient conditions to deduce non-parametric posterior contraction rates, which is also sharper for certain class of models. We further demonstrate that the framework can be adapted to mixture models, which were first introduced by Ferguson \cite{FERGUSON1983287} and Lo \cite{YLo}. We provide an alternative, generalised formulation to Nguyen \cite{nguyen2013}, who first studied mixture models using the Wasserstein distance in the well-specified setting.

In Section \ref{sec:6} we apply our framework to generative models providing statistical analysis for the works of Bonet et al.\cite{bonet2025flowing}, Chemseddine et al.\cite{chem}, and Rigollet \cite{rigollet2026meanfielddynamicstransformers}.

\subsection*{Organisation of the paper}

In Section \ref{sec:2} we recall the classical optimal transport theory, broadly following Sections 1 and 5 from \cite{chewi2025statistical} to discuss the Wasserstein space and \cite{pinzi2025nestedsuperpositionprinciplerandom} to discuss Wasserstein over Wasserstein, in a manner facilitating the developments in later sections. In Section \ref{sec:3} we develop the structure for the $L^2$ over Wasserstein space $\WL$, lifting the theory from Section 2 to the $L^2$ space of Wasserstein random variables. In Section \ref{sec:4} we consider the induced statistical analysis of optimal transport between empirical measures, providing frequentist consistency results for the optimal transport machinery. In Section \ref{sec:5} we proceed similarly, now in the Bayesian paradigm; we consider a Bayesian framework for random optimal transport, demonstrating posterior consistency in the parametric setting and non-parametric settings. In Section \ref{sec:6} we provide applications of the statistical results for various generative modelling setups. In the \hyperlink{app}{Appendix} we collect various technical results and explorations deferred from the main sections.

\addtocontents{toc}{\protect\setcounter{tocdepth}{2}}
\section{Deterministic optimal transport}\label{sec:2}
In this section we recall the classical optimal transport theory which will form the basis for our random theory in Section \ref{sec:3}. We characterise the metric topology of the Wasserstein space, Otto's formal Riemannian geometry for the space, the dynamical equivalence of the Benamou-Brenier theorem, and the identification of Wasserstein gradient flows. For this, we predominantly follow Sections 1 and 5 from \cite{chewi2025statistical}. We also describe a superposition of this theory to the Wasserstein over Wasserstein space, predominantly following \cite{pinzi2025nestedsuperpositionprinciplerandom}.

\subsection{Wasserstein spaces}\label{subsec:2:Wasserstein}

The utility of optimal transport lies in its endowment of the space of probability measures with a geometric structure which canonically respects that of the underlying space. The most common optimal transport distance is the Wasserstein distance which lifts the metric on the ground space $\cX$ to the space of measures on $\cX$, $\cP(\cX)$.

\begin{definition}[Wasserstein distance]
    Consider a Polish metric space $(\cX,d_\cX)$. Let $\mu,\nu\in\cP_p(\cX)$ be two probability measures over $\cX$ with finite $p^{\text{th}}$ moment. Then the $p$-Wasserstein distance between $\mu$ and $\nu$ is given by
\begin{equation}\label{eq:wassersteindef}
    W_p(\mu,\nu) = \inf_{\gamma\in\Gamma_{\mu,\nu}}\l(\int d_\cX(x,y)^p\,\gamma(\d x,\d y)\r)^{\frac{1}{p}},
\end{equation}
where $\Gamma_{\mu,\nu}$ is the set of couplings between $\mu$ and $\nu$, the set of joint distributions on $\cX\times\cX$ with marginals $\mu$ and $\nu$ respectively.
\end{definition} 
The map $x \mapsto \delta_x$ isometrically embeds $(\cX,d_\cX)$ into $(\cP_p(\cX), W_p)$ i.e. $W_p(\delta_x,\delta_y) = d_\cX(x,y)$; it is in this manner that the Wasserstein distance provides its naturally geometric distance between probability measures. The metric topology induced on $\cP_p(\cX)$  is explored extensively in Villani's monographs \cite{villani_topics_2003} and \cite{villani2009optimal} where the author remarks that the Wasserstein space $(\cP_p(\cX), W_p)$ inherits several topological properties of the base space $\cX$ as a consequence of the isometry, a first example being the theorem below.
\begin{theorem}[\cite{villani2009optimal} Theorem 6.18]\label{thm:Polish}
    Let $\cX$ be a Polish metric space, and $p\geq1$. Then the Wasserstein space $(\cP_p(\cX),W_p)$ is also a Polish metric space. Moreover, if $\cX$ is compact, then $(\cP_p(\cX),W_p)$ is also compact.
\end{theorem}

The topology induced by the Wasserstein metric is characterised more explicitly by the following equivalence, shown using Prokhorov's theorem.

\begin{theorem}[\cite{villani_topics_2003} Theorem 7.12]\label{thm:weak_wass}
    Let $\cX$ be a Polish metric space. A sequence $(\mu_n)_{n\geq1}\subset \cP_p(\cX)$ satisfies $W_p(\mu_n,\mu) \to 0$ if and only if it converges weakly to $\mu$, (denoted $\mu_n \wto \mu$), and the $p$-th moment, denoted $\cM_p(\cdot)$, converges: $\int_\cX d_\cX(x,x_0)^p  \mu_n(\d x) \to \int_\cX d_\cX(x,x_0)^p  \mu(\d x)$ for some reference $x_0\in\cX$.
\end{theorem}
Further equivalences are documented in Theorem 7.12 of \cite{villani_topics_2003}. The equivalence of Theorem \ref{thm:weak_wass} proves powerful in establishing convergence results by proving weak convergence and moment convergence, rather than computing $W_p$ directly, which is intractable in general. An immediate consequence of the theorem is that on compact spaces, the Wasserstein topology coincides with the weak topology, as the moment condition follows from Portmanteau's equivalence.
\begin{corollary}\label{cor:compact_equivalence}
    Let $\cK\subset \cX$ be a compact subset of a Polish space $\cX$. Then the Wasserstein distance metrizes the topology of weak convergence on $\cP(\cK)$ i.e. for $(\mu_n)\subset \cP_p(\cK)$, $\mu\in \cP_p(\cK)$,
        $W_p(\mu_n,\mu)\to0 \iff \mu_n\wto \mu.$
\end{corollary}

In a majority of applications, the available data are real-valued, as in Section \ref{sec:4} for the empirical measure, \ref{sec:5} for the Bayesian posterior and in Section \ref{sec:6} for generative modelling examples - in practical settings, it is sufficient to consider Euclidean base spaces. In these cases, for $p=2$, Brenier's theorem provides the existence of a deterministic transport map when the source measure is absolutely continuous with respect to Lebesgue measure. 
\begin{theorem}[\cite{brenier1991polar}]\label{thm:brenier}
Let $\mu, \nu \in \cP_2(\R^\dd)$ be two probability measures such that $\mu$ has a density and let $X \sim \mu$. If $\bar \gamma$ is an optimal coupling for \eqref{eq:wassersteindef} with squared Euclidean cost, then there exists a unique convex function $\varphi:\R^\dd \to \R$ such that $(X, \nabla \varphi(X))\sim \bar \gamma \in \Gamma_{\mu, \nu}$ i.e. $\bar\gamma$ is a deterministic optimal coupling. $\nv$ is referred to as the Brenier map.
\end{theorem}
This provides the necessary step for the formal Riemannian structure for the Wasserstein space in the next sections.

\subsection{Curves of probability measures}\label{subsec:2:curves}

\label{subsec:curves}

Given the isometric embedding $x\mapsto \delta_x$ from $\l(\R^d,\|\cdot\|\r)$ into $\l(\Pc, W_2\r)$, one can understand dynamics on $\Pc$ by superposing the corresponding dynamics of particles on $\R^d$ satisfying an appropriate evolution equation. The standard prescription of dynamics on $\R^d$ using differential calculus makes use of ordinary differential equations (ODEs). 

Given a time-dependent family of vector fields ${(v_t)}_{t\geq 0}$, consider the ODE $\dot X_t = v_t(X_t).$ Suppose now that $X_0\sim\mu_0 \in \Pc$, and similarly let $\mu_t$ denote the law of $X_t$ for each $t\geq 0$.
Informally, the curve of measures ${(\mu_t)}_{t\geq 0}$ describes the evolution of a \emph{collection} of particles with (potentially infinitesimal) mass prescribed by the measure at each time $t$.
The dynamics of ${(\mu_t)}_{t\geq 0}$ are subject to the continuity equation below.

\begin{proposition}[Continuity equation \cite{villani_topics_2003}]
    Suppose that $X_0 \sim \mu_0$, and that ${(X_t)}_{t\geq 0}$ evolves according to the dynamics $\dot X_t = v_t(X_t)$, assumed to be well-posed.
    Let $\mu_t$ denote the law of $X_t$ for all $t\geq 0$.
    Then, ${(\mu_t)}_{t\geq 0}$ satisfies the following equation in the weak sense,
    \begin{align}\label{eq:continuity}
        \partial_t \mu_t + \divergence(\mu_t v_t) = 0,\qquad \text{i.e.}\qquad \partial_t \int \varphi  \,\d \mu_t = \int \langle \nabla \varphi, v_t\rangle  \,\d \mu_t.
    \end{align}
    for all compactly supported smooth test functions $\varphi\in C^\infty_c(\R^d)$, $\varphi : \R^d\to\R$.
\end{proposition}
The equation expresses a conservation of (probability) mass: the change in density at any point is equal to the divergence of the probability mass escaping at that point. 
When $\mu_t$ admits a smooth density, an integration by parts yields the strong form of \eqref{eq:continuity}, identifying $\mu_t$ with its density.

Consider a single particle of unit mass evolving according to the ODE $\dot X_t = v_t(X_t)$ with trajectory $t\mapsto X_t$. The instantaneous kinetic energy (KE) is given by $\|\dot X_t\|^2 = \|v_t(X_t)\|^2$ and the total KE of the trajectory over the time interval $[0,1]$ is given by the integral
$\int_0^1 \|v_t(X_t)\|^2\d t$. The energy minimising trajectory over curves with fixed endpoints $X_0, X_1$ is shown to be the straight line interpolating $X_0$ to $X_1$, $t\mapsto (1-t)X_0 + tX_1$.

The picture is superposed to $\Pc$. Viewing all particles in aggregate, weighted by the probability mass assigned by the measure $\mu_t$ at time $t$, one measures the instantaneous and total KE over the time interval $[0,1]$ of the system with
\begin{align}\label{eq:kinetic_energy}
    \|v_t\|_{\mu_t}^2 \deq \int \|v_t\|^2  \d \mu_t,\qquad \int_0^1 \|v_t\|_{\mu_t}^2 \d t = \int_0^1 \int_{\R^d} \|v_t(x)\|^2\mu_t(\d x)
    \d t.
\end{align}
As Chewi et al. remark in Section 5.1 of \cite{chewi2025statistical}, the velocity field generating the path $\mu_t$ can be highly non-unique. Ideally, one would like to choose the velocity field minimising the KE in \eqref{eq:kinetic_energy} while describing the dynamics of the curve and be able to conclude that this choice is unique.

The action of the velocity field $v_t(\cdot)$ on a single particle $x_t$ in $\R^d$ describes a tangent vector to its trajectory $(x_t)_{t\geq0}$ at time $t$. To fully prescribe the dynamics of a curve in Wasserstein space, we need to make the identification of suitable tangent vectors. Otto's \cite{otto2001geometry} formal Riemannian structure uses Brenier's theorem to show that optimal velocity vector fields lie in the $L^2$ closure of the space of gradients of functions.

\begin{definition}[Tangent space to $\Pac$]\label{def:w2_tangent}
    Let $\mu \in \cP_{2,\rm ac}(\R^d)$.
    The tangent space to $\cP_{2,\rm ac}(\R^d)$ at $\mu$ is defined to be
    \begin{align*}
        \cT_\mu \cP_{2,\rm ac}(\R^d)
        &\deq \overline{\bigg\{\nabla \psi \mid \psi : \R^d\to\R~\text{compactly supported, smooth}\bigg\}}^{L^2(\mu;\R^d)}
    \end{align*}
    where $\overline{\{\cdot\}}^{L^2(\mu;\R^d)}$ denotes the $L^2(\mu)$ closure.
    The space $\cT_\mu \cP_{2,\rm ac}(\R^d)$ is endowed with the $L^2(\mu)$ inner product (inducing the KE norm),
    \begin{align*}
        \langle \nabla \psi_1,\nabla \psi_2\rangle_\mu
        &\deq \int \langle \nabla \psi_1,\nabla \psi_2\rangle\, \d \mu.
    \end{align*}
\end{definition}
The Benamou-Brenier Theorem \ref{thm:benamou_brenier} below from \cite{benamou2000computational} computes the constant-speed geodesics in Wasserstein space, which are the paths between $\mu_0,\mu_1\in\Pac$ minimising the KE \eqref{eq:kinetic_energy}. It asserts that the metric induced on $\Pac$ by the Riemannian structure of Definition \ref{def:w2_tangent} is indeed the Wasserstein distance, thereby providing a dynamic mass transportation scheme.

\begin{theorem}[\cite{benamou2000computational}]\label{thm:benamou_brenier}\index{Benamou--Brenier formula}
    Let $\mu_0,\mu_1 \in \cP_{2,\rm ac}(\R^d)$.
    Then,
    \begin{align}
        W_2^2(\mu_0,\mu_1)
        &= \min\l\{\int_0^1 \|v_t\|_{\mu_t}^2  d t \Bigm\vert {(\mu_t, v_t)}_{t\in [0,1]}~\text{solves}~\eqref{eq:continuity}\r\}.\label{eq:benamou_brenier}
    \end{align}
    The optimal curve ${(\mu_t)}_{t\in [0,1]}$ is unique and is described by $X_t\sim\mu_t$, where $X_t = (1-t)X_0 +tX_1$ and $(X_0,X_1)\sim\bar\gamma \in \Gamma_{\mu_0,\mu_1}$ for an optimal coupling $\bar\gamma$.
\end{theorem}
One formulates the following definition as a result, introduced by McCann in \cite{mccann1997convexity}.

\begin{definition}[Displacement interpolation]\label{def:w2_geod}
    Let $\mu_0,\mu_1 \in \cP_{2,\rm ac}(\R^d)$ and let $\nv$ denote the Brenier map from $\mu_0$ to $\mu_1$.
    The displacement interpolation joining $\mu_0$ to $\mu_1$ is the curve ${(\mu_t)}_{t\in [0,1]}$ where
    \begin{align}\label{eq:displ_interp}
        \mu_t = \big[(1-t)\,{\id} + t\nv\big]_\# \mu_0.
    \end{align}
\end{definition}

Definition \ref{def:w2_geod} characterises a constant-speed geodesic \cite{burago2001course} between elements of $\Pac$. The Wasserstein geometry superposes the geodesic structure of $\R^d$ - the points are probability measures and the geodesics are displacement interpolations of their masses. Ambrosio et al. \cite{ambrosioGradientFlowsMetric2008} Section 7.2 extend the notion of displacement interpolation in Definition \ref{def:w2_geod} to an interpolation of any optimal coupling (restricting to the absolutely continuous case, the coupling is unique and deterministic, recovering the interpolation \eqref{eq:displ_interp} above).

\begin{proposition}[\cite{ambrosioGradientFlowsMetric2008}]\label{prop:Wasserstein_is_geodesic}
    Let $\mu_0, \mu_1\in(\Pc,W_2)$ and let $\gamma\in\Gamma_{\mu_0,\mu_1}$ be any optimal coupling. Then for the displacement interpolation on the coupling $\disp_t(x,y) \deq (1-t)x + ty$ where $t\in[0,1],$ the path $\mu_t$ given by ${\disp_t}_\#\gamma$ is a constant-speed geodesic in $(\Pc,W_2)$ connecting $\mu_0$ to $\mu_1$ i.e. for $s,t\in[0,1]$, 
    $W_2(\mu_s, \mu_t) = |s-t|W_2(\mu_0,\mu_1)$. In particular, when $\mu_0,\mu_1\in \Pac$ are absolutely continuous, the coupling $\disp_t\#\gamma$ is deterministic, given by the displacement interpolation \eqref{eq:displ_interp}.
\end{proposition}

\subsection{Wasserstein gradient flows}\label{subsec:2:WGF}

With Otto's \cite{otto2001geometry} formal Riemannian geometry defined in Section \ref{subsec:2:curves}, one can make the identification of the Wasserstein gradient of a functional over the space of probability measures, in turn permitting the definition of a Wasserstein gradient flow. We begin by defining the first variation.

\begin{definition}[First variation]\label{def:first_variation}\index{first variation}
    Let $\cF : \cP_{2,\rm ac}(\R^d)\to\R$ be a functional.
    The {first variation} of $\cF$ at $\mu$, denoted $\frac{\delta\cF}{\delta\mu}(\mu) : \R^d\to\R$, is given by
    \begin{align*}
        \lim_{\varepsilon\searrow 0}\frac{\cF(\mu+\varepsilon\chi) - \cF(\mu)}{\varepsilon} = \int \frac{\delta\cF}{\delta\mu}(\mu)\,\d\chi,
    \end{align*}
    for all signed measures $\chi$ such that $\mu+\varepsilon\chi \in \Pac$ for all sufficiently small $\varepsilon$.
\end{definition}

Given a curve of measures ${(\mu_t)}_{t\geq 0}$ with corresponding tangent vectors ${(v_t)}_{t\geq 0}$, the gradient of a functional $\cF : \cP_{2,\rm ac}(\R^d) \to\R$ is the element of the tangent space $\cT_{\mu_t}\cP_{2,\rm ac}(\R^d)$ such that $\partial_t \cF(\mu_t) = \langle \gradW \cF(\mu_t),v_t\rangle_{\mu_t}$.
This definition can be written directly in terms of the first variation, a formal calculation first explored by Otto in \cite{otto2001geometry}.
\begin{proposition}[\cite{otto2001geometry}]\label{prop:w2_grad}\index{Wasserstein gradient}
    Let $\cF : \cP_{2,\rm ac}(\R^d) \to\R$ be a functional with first variation $\frac{\delta}{\delta\mu} \cF$.
    Then, the Wasserstein gradient of $\cF$ is the vector field $ \gradW \cF(\mu): \R^d \to \R^d$ defined by
    \begin{align*}
        \gradW \cF(\mu)(\cdot) = \nabla \frac{\delta\cF}{\delta\mu}(\mu)(\cdot),
    \end{align*}
    where $\nabla$ on the right-hand side denotes the usual Euclidean gradient.
\end{proposition}
This follows after applying the continuity equation in combination with an integration by parts, concluding via the identification described above. Recalling that the tangent vectors govern the evolution of ${(\mu_t)}_{t\geq 0}$ via the continuity equation \eqref{eq:continuity}, one arrives at the following definition.

\begin{definition}[Wasserstein gradient flow]\label{def:w2_grad_flow}\index{Wasserstein gradient flow}
    Let $\cF : \cP_{2,\rm ac}(\R^d)\to\R$ be a functional. ${(\mu_t)}_{t\geq 0}$ is the {Wasserstein gradient flow} of $\cF$ if it solves the PDE $\partial_t \mu_t = \divergence\bigl(\mu_t \gradW \cF(\mu_t)\bigr)$ in the weak sense.
\end{definition}
The Wasserstein gradient flow allows one to consider optimisation problems at the level of measures, using the gradient flow to define a path towards the measure minimising a functional under suitable conditions. Many of these applications, as well as their computational schemes, are presented in chapter 6 of \cite{chewi2025statistical}. We explore the applications of Wasserstein gradient flows for flow-matching and transformer self-attention flows in Sections \ref{subsec:6:COT} and \ref{subsec:6:TSA} respectively, generalising to the random flows setting we define in Section \ref{subsec:3:RGFs}.

\subsection{Wasserstein over Wasserstein space}\label{subsec:2:wow}

The fact that the Wasserstein space inherits much of the structure of the base space also allows for the definition of Wasserstein "over" Wasserstein spaces which hierarchically inherit the structure of the base space $\cX$. First introduced by Nguyen in \cite{nguyen2016borrowing}, these have been explored in great detail by Pinzi and Savaré in \cite{pinzi2025nestedsuperpositionprinciplerandom}, \cite{pinzi} and Beiglböck et al. in \cite{beiglboeck}. These are of particular interest for hierarchical, non-parametric setups where we are concerned with distributions over $\cP_p{(\R^d)}$.

The Wasserstein over Wasserstein (WoW) space over a Polish space $\cX$ is the Wasserstein space over $\cP_p{(\R^d)}$ with base Wasserstein cost i.e. $\l(\cP_p(\cP_p(\cX)), W_p\r)$.

\begin{definition}[Wasserstein over Wasserstein distance]
    Let $\P,\Q\in\cP_p(\cP_p(\cX))$. The $p$-Wasserstein over Wasserstein (WoW) distance between the two is given by
\begin{equation*}
    W_{W_p}(\P,\Q) \deq \inf_{\gamma\in\Gamma_{\P,\Q}} \l(\int_{\cP_p(\cX)\times \cP_p(\cX)} W_p^p(\mu,\nu)\,\gamma(\d\mu,\d \nu)\r)^{\frac{1}{p}}
\end{equation*}
where $\Gamma_{\P,\Q}$ is the space of couplings between $\P$ and $\Q$ i.e. joint probability measures over $\cP_{p}(\cX)\times\cP_{p}(\cX)$ with marginals $\P,\Q$ respectively.
\end{definition}
In the same manner that the standard Wasserstein distance embeds $\cX$ isometrically into $\cP_p(\cP_p(\cX))$, the Wasserstein space $\cP_p(\cX)$ embeds into the WoW space $\cP_p(\cP_p(\cX))$ isometrically via $\mu\mapsto\delta_\mu$, so that much of the structure of the Wasserstein space carries over to the WoW space. But the Wasserstein space itself enjoys much of the structure of the base space $\cX$, so the WoW space inherits the structure of $\cX$ hierarchically via the isometric embedding $x\mapsto\delta_{\delta_x}$. This phenomenon characterises a \textit{nested superposition principle}, explored in great detail by Pinzi and Savaré \cite{pinzi2025nestedsuperpositionprinciplerandom}, who demonstrate that the dynamics of particles on $\cX$ can first be superposed from $\cX$ to $\cP_p(\cX)$, then further superposed to $\cP_p(\cP_p(\cX))$. Restricting to $\cX = \R^d$ and $p=2$, we recover much of the theory for $\Pc$ at the superposed level $\Pcpc$.
\begin{theorem}[\cite{pinzi2025nestedsuperpositionprinciplerandom}]\label{thm:WoW_topology}
    The following topological results hold for the Wasserstein over Wasserstein space over $\R^d$, $\Pcpc$:
    \begin{itemize}
        \item The space is Polish, and restricting to a compact set $K\subset \R^d$, the WoW space $\Pkpk$ is also compact
        \item The convergence of $(\P_n)_{n\geq1}\subset\Pcpc$ to $ \P\in\Pcpc$ in WoW is equivalent to weak convergence over $\Pcpc$ plus convergence of the second moment, $$\int_{\Pc}W_2^2(\mu,\delta_0)\,\P_n(\d\mu)\to\int_{\Pc}W_2^2(\mu,\delta_0)\,\P(\d\mu).$$ 
        \item The WoW distance metrizes the weak topology on $\Pkpk$ for compact sets $\cK\subset\R^d$.
    \end{itemize}
\end{theorem}
These results are particularly useful for the non-parametric statistical analysis, which necessitates a consideration of the superposed level $\cP(\cP(\R^d))$. The dynamic theory from Sections \ref{subsec:2:curves} and \ref{subsec:2:WGF} can also be superposed to the WoW level, as explored extensively in \cite{pinzi} and \cite{pinzi2025nestedsuperpositionprinciplerandom} - we review this concisely in Appendix \ref{App:WoW_dynamics}. Various applications of the differential structure for WoW, primarily minimising WoW functionals via WoW gradient flows, are covered by Bonet et al. in \cite{bonet2025flowing}; in Section \ref{subsec:6:flowing_datasets} we demonstrate that their WoW representation of labelled datasets can be embedded into our random framework, and is thus open to statistical convergence results we develop in Sections \ref{subsec:4:hierarchical} and \ref{subsec:5:consistency} before the representation is passed through the gradient flow.

\section{Wasserstein random variables}\label{sec:3}
In this section we develop the $L^2$ over Wasserstein space of random probability measures, the $L^2$ space of Wasserstein-valued measures. We characterise the structure of this space, including metric geometry, constant-speed geodesics, and formal Riemannian geometry. The structure is shown to be canonical, in the sense that it induces a random gradient flow where the sample paths are each deterministic Wasserstein gradient flows. In the last subsection, we detail the same theory, superposed to the $L^2$ Wasserstein over Wasserstein space, the $L^2$ space of Wasserstein over Wasserstein valued measures.

\subsection{Random probability measures}\label{subsec:3:RPM}

We will see that the $L^2$ over Wasserstein space is a subset of the space of random measures. We first recall some basic notions of random measures, broadly following Kallenberg's Section 2 from \cite{kallenberg2017random}.

A random measure is, in effect, a measure-valued random variable: a randomly chosen measure $\xi$ on a measurable space $(S,\cS)$. In this manner, one regards $\xi$ simply as a measure depending on an extra parameter $\om$ belonging to some abstract probability space $(\Omega,\cA, \PP)$. Henceforth, unless otherwise stated, we assume the base probability space is given by the triplet $(\Omega,\cA, \PP)$. Moreover, one ensures $\xi$ is a kernel from $\Omega$ to $S$, that is, the integral $\xi f \deq \int f \d \xi$ is a (real-valued) random variable for every measurable function $f\geq0$ on $S$.
\begin{definition}[Random measure \cite{kallenberg2017random}]\label{def:random_measure}
    Let $\cX$ be a Polish space and  ${\cB(\cX)}$ its Borel $\sigma$-algebra.
    A random measure $\xi$ is an almost sure locally finite transition kernel from a probability space 
    $(\Omega ,{\cA},\PP)$ to ${ (\cX,{\cB(\cX)}})$ i.e.
    \begin{itemize}
        \item For every fixed $B\in\cB(\cX)$, the map $\om\mapsto\xi(\om, B)$ is measurable from $(\Omega,\cA)\to(\R, \cB(\R))$.
        \item For every fixed $\om\in\Omega$, the map $B\mapsto\xi(\om, B)$, $B\in \cB(\cX)$ is a measure on $(\cX,\cB(\cX))$ where each measure $\xi(\om, \cdot)$ satisfies $\xi(\om,\tilde B)<\infty$ for all bounded measurable sets $\tilde B\in \cB(\cX)$ and $\om\in\Omega$ up to some $\PP$-null set.
    \end{itemize}
\end{definition}
    One formalises the analogy of treating a random measure as a measure-valued random variable by furnishing the space of measures with an appropriate $\sigma$-algebra. Denote $$\tilde \bM\deq\big\{\mu\,|\,\mu \text{ is a measure on }(\cX,\cB(\cX))\big\},\qquad \bM\deq\l\{\mu\in\tilde \bM\,|\,\mu(\tilde B)<\infty\text{ for all bounded }\tilde B \in \cB(\cX)\r\}.$$
    For all bounded measurable ${\Tilde {B}}$, define the maps $I_{\Tilde B}:\Tilde{\bM}\to\R$, $\mu\mapsto\mu(\Tilde B)$. Let $\Tilde{\bM}$ be the $\sigma$-algebra induced by the mappings $I_{\Tilde B}$ on $\Tilde{\bM}$, and $\bM$ the $\sigma$-algebra induced by the mappings $I_{\Tilde B}$ on $\bM$ (one finds $\Tilde{\bM}|_\bM = \bM$). A random measure is then a measurable map from $(\Omega, \cA, \PP)$ to $(\Tilde \bM, \Tilde{\bM})$.\newline

\subsection{Wasserstein random variables}\label{subsec:3:WRM}

The classical optimal transport framework requires measures of the same total mass. We specialise to random \textit{probability} measures on $\R^d$ where for each $\om\in\Omega$, $\xi(\om, \cdot)=\xi_\om(\cdot)$ is a probability measure on $\R^d$ with volume 1. Having detailed optimal transport between probability measures, we are placed to lift the existing theory to the space of random probability measures via optimal transport on their realisations.

Let $\xi,\eta$ be two random probability measures. Under suitable assumptions (e.g. finite moments), the machinery from Section \ref{sec:2} immediately applies to consider the optimal transport problem between $\xi_\om(\cdot)$ and $\eta_\om(\cdot)$. We consider the induced optimal transport problems over the full event space $\Omega$, where the probability space $(\Omega,\cA,\PP)$ describes statistical uncertainty, first about the source and target measures, and later about flow paths on $\Pc$.

Just as $L^2$ random variables are used to describe statistical uncertainty at the level of $\R^d$, we explore an $L^2$ space of random probability measures to describe statistical uncertainty at the level of the Wasserstein space, $\l(\Pc,W_2\r)$. We make the following definition.

\begin{definition}[$L^2$ over Wasserstein space]\label{def:rm_space}
    Consider the space of random measures on the probability space $(\Omega,\cA,\PP)$. Let $\WL\deq\LW$ denote the subset of random probability measures which
    \begin{itemize}
        \item are measurable from $(\Omega,\cA)\to \l(\Pc, \cB_{W_2}\r)$ where $\cB_{W_2}$ is the Borel $\sigma$-algebra induced by the $W_2$ topology on $\Pc$,
        \item are square-integrable i.e. for $\xi\in\WL$,
    $$\cE_2(\xi) \deq \E_\om[\cM_2(\xi_\om)] = \E_{\om}\l[\int_{\R^d}\|x\|^2\xi_\om(\d x)\r] <\infty.$$
    \end{itemize}
    We refer to this as the $L^2$ over Wasserstein space of random probability measures, endowed with the $L^2$ metric
    $d(\xi,\eta) \deq \E_{\om}\l[W_2^2\big(\xi_\om,\eta_\om\big)\r]^{\frac{1}{2}}.$
    
\end{definition}
As is common for $L^2$ spaces of random variables, we introduce the equivalence classes $$[\xi] = \l\{\eta\in \WL: \eta\sim\xi\r\}, \qquad \xi\sim\eta\iff \xi = \eta \quad \PP\text{-a.s.}$$
and identify the equivalence class $[\xi]$ by a representative $\xi$ throughout, ensuring that the $L^2$ distance is definite. Taking this random variable perspective lets us reframe convergence according to Vitali's equivalent conditions for $L^2$ convergence. 

\begin{theorem}[Vitali equivalence]\label{thm:Vitali_consistency}
    $d\l(\xi_n, \xi\r)\to0$ if and only if the following two conditions hold:
    \begin{enumerate}[label=(\alph*)]
    \item Consistency in Probability: $\PP\l( W_2^2(\xi_n,\xi) > \varepsilon\r)\rightarrow 0 \quad \forall \varepsilon > 0$
    \item Uniform Integrability: $\l\{\cM_2(\xi_n)\r\}_{n\geq1}$ is uniformly integrable so that \begin{equation*}\label{eq:Uniform_integrability_dw}
        \lim_{R \to \infty} \sup_{n\geq1} \, \E_\om \l[\int_{\|x\|^2>R}\|x\|^2\xi_{\om,n}(\d x)\r] = 0.
    \end{equation*}
    \end{enumerate}
\end{theorem}

\begin{remark}\label{remark:sigma_algebras}
    The Borel $\sigma$-algebra induced by the Wasserstein topology, as used for Definition \ref{def:rm_space}, is actually equal to the Borel $\sigma$-algebra induced by the weak topology. Considering Theorem \ref{thm:weak_wass}, it suffices to verify that the moment map $\cM_2(\cdot)$ is weakly measurable, as shown in Appendix \ref{App:remark:sigma_algebras}.
\end{remark}

The Wasserstein distance is realised as the infimum of transport costs over all possible transport plans. In general, the minimising plan is not unique, which can pose a measurable selection issue if not treated. The following theorem provides the existence of the necessary selection.

\begin{theorem}[Measurability of the optimal coupling, \cite{villani2009optimal} Corollary 5.22]\label{thm:Measurability}
Given a probability space $(\Omega, \cA, \PP)$ and two random measures $\xi, \eta \in L^2\l(\cP_2(\cX)\r)$ over a Polish base space $\cX$, there exists a measurable selector $\om \mapsto \gamma_\om$, where $\gamma_\om$ is an optimal transport plan between $\xi_\om$ and $\eta_\om$.
\end{theorem}
This follows from the fact that the set of all couplings is uniformly tight and weakly closed, hence weakly compact, providing access to the Arsenin-Kunugui Selection Theorem (\cite{kechris1995classical}, Theorem 18.18) to conclude.

Recall that while $L^p$ spaces act at the level of random variables directly, the Wasserstein space acts at the level of laws of those random variables, and relies on couplings between them. This distinguishes the $L^2$ over Wasserstein space from the Wasserstein over Wasserstein (WoW) space presented by Pinzi and Savaré in \cite{pinzi2025nestedsuperpositionprinciplerandom}, \cite{pinzi} and other works.  The 2-WoW distance between random measure laws $\P,\Q\in\cP_2(\Pc)$ can be formulated using random probability measures in $\WL$:
    \begin{equation*}
        W_{W_2}(\P,\Q)\deq \inf_{\xi\sim\P,\eta\sim\Q}\E_\om\l[W_2^2(\xi_\om,\eta_\om)\r]^{\frac{1}{2}} = \inf_{\xi\sim\P,\eta\sim\Q}d(\xi,\eta)
    \end{equation*}
    where $\xi,\eta\in\WL$ are defined over a new common probability space, and the infimum determines an optimal coupling. Alternatively, at the direct level of random measures, rather than their laws, $\xi,\eta\in\WL$ are already defined on the same probability space $(\Omega,\cA,\PP)$, hence already coupled. The relationships between the different spaces are shown in Figure \ref{fig:superposition}. 
    
    \begin{remark}
        To facilitate the statistical analysis in Sections \ref{sec:4} and \ref{sec:5}, it is necessary to work at the level of the random measures directly, rather than their laws. The framework presented in this section is discussed in a manner in anticipation of this analysis.
    \end{remark}

\begin{figure}
    \centering
    \begin{center}
\begin{tikzpicture}[x=3.5cm, y=3cm, >=stealth, every node/.style={font=\large}]
  
  \node (R)    at (0, 0) {\Large $\R^d$};
  \node (L2R)  at (1, 1) {\Large$L^2\l(\R^d\r)$};
  \node (P2R)  at (2, 0) {\Large$\Pc$};
  \node (L2P2) at (3, 1) {\Large$L^2\l(\Pc\r)$};
  \node (P2P2) at (4, 0) {\Large $\cP_2\l(\Pc\r)$};

  \draw[->, thick] (R) -- (L2R) 
    node[midway, sloped, above=2pt, text=darkgray] {\small Stochasticity}
    node[midway, sloped, below=2pt, text=gray] {\small $\om\mapsto X_\om$};
    
  \draw[->, thick] (R) -- (P2R) 
    node[midway, above=2pt, text=darkgray] {\small Superposition}
    node[midway, sloped, below=2pt, text=gray] {\small $x\mapsto \delta_x$};
    
  \draw[->, thick] (L2R) -- (P2R) 
    node[midway, sloped, above=2pt, text=darkgray] {\small Laws}
    node[midway, sloped, below=2pt, text=gray]{\small $X\mapsto \P_X$};
  
  \draw[->, thick] (L2R) -- (L2P2) 
    node[midway, above=2pt, text=darkgray] {\small Superposition}
    node[midway, sloped, below=2pt, text=gray]{\small $X_\om\mapsto \delta_{X_\om}$};
  
  \draw[->, thick] (P2R) -- (L2P2) 
    node[midway, sloped, above=2pt, text=darkgray] {\small Stochasticity}
    node[midway, sloped, below=2pt, text=gray]{\small $\om\mapsto \xi_\om$};
    
  \draw[->, thick] (P2R) -- (P2P2) 
    node[midway, above=2pt, text=darkgray] {\small Superposition}
    node[midway, sloped, below=2pt, text=gray]{\small $\mu\mapsto \delta_\mu$};
    
  \draw[->, thick] (L2P2) -- (P2P2) 
    node[midway, sloped, above=2pt, text=darkgray] {\small Laws}
    node[midway, sloped, below=2pt, text=gray]{\small $\xi\mapsto \P_\xi$};

  \draw[->, thick] (R) to[bend right=20] 
    node[midway, above=6pt, text=darkgray] {\small Nested Superposition} node[midway, sloped, below=2pt, text=gray]{\small $x\mapsto \delta_{\delta_x}$} (P2P2);

\end{tikzpicture}
\end{center}
    \caption{Interaction between the different spaces. \textit{Stochasticity} denotes moving from a space to $L^2$ functions (square integrable random variables) on the space. \textit{Laws} denote moving from an $L^2$ random variable to its probability distribution. \textit{Superposition} denotes the embedding of one space into a higher space via a Dirac distribution (recall from Section \ref{subsec:2:curves} the superposition of dynamics on $\R^d$ to dynamics on $\Pc$). The consequences of (nested) superpositions are explored in detail in \cite{pinzi2025nestedsuperpositionprinciplerandom}.}
    \label{fig:superposition}
\end{figure}

\subsection{Geodesic geometry for Wasserstein random variables}\label{subsec:3:Geodesics}

Having outlined the $L^2$ over Wasserstein space, we simultaneously consider all optimal transport problems induced between the realisations $\xi_\om$ and $\eta_\om$. On the paths, the theory of the Wasserstein space in Section \ref{sec:2} applies, thus many of the classical deterministic results will continue to hold at the $L^2$ level $\WL$ too, given suitable arguments.

Just as the 2-Wasserstein distance over $\Pac$ enjoys a rich Riemannian structure, we can lift much of this structure to the space $\WL$ when we restrict to random probability measures with absolutely continuous realisations, which we denote $\wlac$. The first instance of theory we sample are the constant-speed geodesics from the Benamou-Brenier Theorem \ref{thm:benamou_brenier}.

\begin{proposition}[Constant-speed geodesics]\label{prop:lifted_geodesics}
    Let $\xi_0,\xi_1\in \wlac$. The following quantity $\xi_t$ is a constant-speed geodesic in $\wlac$ joining $\xi_0$ to $\xi_1$ in $\wlac$,
    \begin{equation}\label{eq:lifted_constant_speed}
        \xi_{\om,t} \deq \big[(1-t)\,\id + t\nabla \varphi_\om\big]_{\#} \xi_{\om,0} \qquad t\in[0,1],
    \end{equation}
    where $\varphi_\om$ is the Brenier map between the realisations $\xi_{\om,0}$ and $\xi_{\om,1}$.
\end{proposition}
A simple calculation verifying the result is given in Appendix \ref{App:prop:lifted_geodesics}.
\begin{remark}
    One verifies that these constant-speed geodesics in the absolutely continuous case $\wlac$ generalise to $\WL$ by using a pathwise coupling interpolation (as in Proposition \ref{prop:Wasserstein_is_geodesic}) rather than displacement interpolations. The selection of the optimal coupling used for the interpolation is measurable due to Theorem \ref{thm:Measurability}.
\end{remark}

Recall that in the context of the Benamou-Brenier Theorem \ref{thm:benamou_brenier}, the constant-speed geodesics \eqref{eq:displ_interp} define an optimal path minimising the KE between probability measures in the Wasserstein space.
Similarly, the constant-speed geodesic \eqref{eq:lifted_constant_speed} defines an optimal trajectory to minimise now the \textit{expected} KE of the trajectory between the random probability measures, subject to paths which satisfy the continuity equation \eqref{eq:continuity}. This is a simple pathwise application of the Benamou-Brenier Theorem \ref{thm:benamou_brenier}.
\begin{theorem}[Random Benamou-Brenier]\label{thm:lifted_benamou-brenier}
    Let $\xi_0, \xi_1\in \wlac$. Then 
    \begin{equation}\label{eq:lifted_benamou_brenier}
        d^2(\xi_0,\xi_1) = 
        \min\l\{
            \E_\om\l[\int_0^1
                \big\|V_{\om,t}(\cdot)\big\|^2_{\xi_{\om,t}}
            \d t \r]\,\bigg|\,(\xi_{\om,t}, V_{\om,t})_{t\in[0,1]}\text{ solves \eqref{eq:continuity} $\PP$-a.s.}
        \r\}
    \end{equation}
\end{theorem}
A proof, which makes pathwise use of the standard Benamou-Brenier Theorem, is given in Appendix \ref{App:thm:lifted_benamou-brenier}. This result provides the natural intuition for defining a formal Riemannian structure for $\WL$ in the next section.

\subsection{Random gradient flow paths}\label{subsec:3:RGFs}

In Definition \ref{def:w2_tangent} we defined the tangent space to $\Pac$ for which the metric induced by the given Riemannian structure coincides with that of the Wasserstein distance. Given the formulation of $\WL$, we can similarly define a lifted tangent space to $\wlac$ for which the metric induced by the lifted Riemannian structure will coincide with our $L^2$ over Wasserstein metric. We will show that the resulting structure gives rise to a canonical random Wasserstein gradient flow, in the sense that it induces a random gradient flow where the sample paths are each deterministic Wasserstein gradient flows.

In light of the (random) dynamic formulation of the $\WL$ distance $d$ in Theorem \ref{thm:lifted_benamou-brenier}, the definition of a formal Riemannian structure on $\WL$ is now apparent. The tangent space we define consists of $L^2$ random vectors whose realisations are tangent vectors to $\Pac$ i.e. the $L^2$ closure of gradients of functions in $C_c^\infty(\R^d)$. The (squared) norm for those tangent vectors to $\Pac$ was given by the corresponding KE $\int_0^1 \|v_t\|^2\d t$. At the level of $\Pac$, the norm we place on the tangent space below in \eqref{eq:lifted_ke_norm} corresponds to the $\PP$-expectation of the KE, as discussed in Theorem \ref{thm:lifted_benamou-brenier}.

\begin{definition}[Tangent space to $\wlac$]\label{def:Lifted_tangent}
    Let $\xi\in\wlac$. We define the tangent space to $\wlac$ at $\xi$ to be the $L^2$ space of measurable functions from $\Omega\times\R^d\to\R^d$ with samples in $\cT_{\xi_\om}\Pac$ $\PP$-a.s., recalling
    \begin{align*}
        \cT_{\xi_\om} \Pac
        &\deq \overline{\bigg\{\nabla \psi \,\Big|\, \psi : \R^d\to\R ~\text{compactly supported, smooth}\bigg\}}^{L^2(\xi_\om)}.
    \end{align*}
    More explicitly, let $\PP_\xi\in\cP(\Omega\times\R^d)$ be given by $\PP_\xi(\d\om,\d x)\deq \PP(\d\om)\xi_{\om}(\d x)$. Then the tangent space is given by 
    \begin{equation*}
        \cT_\xi\wlac = {\bigg\{
            V\in L^2\l(\PP_\xi;\R^d\r)\,\Big|\,V_{\om}(\cdot)\in \cT_{\xi_\om}\Pac, \,\,||V||_\xi<\infty
        \bigg\}},
    \end{equation*}
    where we have endowed the tangent space $\cT_\xi\wlac$ with the inner product
    \begin{equation}\label{eq:lifted_ke_norm}
        \big\langle V, W \big\rangle_\xi = \E_\om\bigg[\big\langle V_{\om}, W_{\om}\big\rangle_{\xi_\om}\bigg] = \int_{\Omega\times\R^d} \big\langle V_{\om}( x), W_{\om}( x)\big\rangle\, \PP_\xi(\d\om, \d x).
    \end{equation}
\end{definition}
\begin{remark}
    The inner product \eqref{eq:lifted_ke_norm} is known as the Wasserstein Dirichlet form, first introduced by von Renesse and Sturm in \cite{vonrenesse2009}, who use standard Dirichlet form theory (as in \cite{Fukushima2010}) to provide the existence of a unique, continuous, symmetric Markov process on $\cP(\bS^{d-1})$ after proving closure of the Dirichlet form \eqref{eq:lifted_ke_norm}. Similar works such as \cite{Schiavo2018ART} and \cite{delloschiavo2024massive} have since shown the existence of other diffusion processes on various Wasserstein spaces via a similar route, making use of the Wasserstein Dirichlet form. These works use random measure laws to describe random dynamics on Wasserstein spaces and indeed, the quantity \eqref{eq:lifted_ke_norm} uses only the law of the random measure $\xi$. 
    
    The fact that our structure, emerging naturally from the definition of $\WL$ and the dynamic equivalence established in Theorem \ref{thm:lifted_benamou-brenier}, aligns with the structure explored in these other works situates our approach within the broader framework of random Wasserstein dynamics. However, we focus on a different type of random dynamics. While these works investigate randomness via diffusion processes on Wasserstein space, aiming at statistical modelling, our interest lies instead in random gradient flows, instantiated by a suitable, given functional.
\end{remark}

The formal Riemannian structure enables us to make the identification of random Wasserstein gradient flows for random probability measures, highlighting that $\WL$ is the appropriate space in which to study random gradient flows. We first determine the appropriate space of functionals. Given our definition of $L^2$ over Wasserstein as a space for random optimal transport, we consider functionals of random probability measures as random elements too.

\begin{definition}[Random functional]\label{def:random_functional}
    Let $\{\cF_\om\}_{\om\in\Omega}$ be a collection of measurable functionals over $\Pac$ i.e.  for each $\om\in\Omega$, $\cF_{\om}:\Pac\to\R$ . We say that $\cF:\Omega\times\wlac\to\R$ is a random functional when given by $\cF(\om,\xi) = \cF_{\om}(\xi_\om).$ The corresponding lifted functional is the map $\mathbf{F}:\wlac\to\R$, $$\mathbf{F}(\xi) = \E[\cF(\xi)] = \E_{\om}[\cF_\om(\xi_\om)]$$
\end{definition}

We will primarily consider $\cF_\om = \cF$ a.s. for some deterministic functional $\cF:\Pac\to\R$; however, the theory is more general given that the $\om$-sample paths are decoupled. Under suitable conditions on the $\cF_\om$ (e.g. measurability w.r.t $B_{W_2}$, the $W_2$ Borel $\sigma$-algebra), this defines a real-valued random variable $\mathbf{F}:\Omega\to\R$ when $\xi$ is fixed. \newline

In our lifted setting, given a curve of random probability measures ${(\xi_t)}_{t\ge 0}$ with corresponding tangent vectors ${(V_t)}_{t\ge 0}$, the gradient of a lifted functional $\mathbf{F} :\Omega\times \wlac \to\R$ is the element of the tangent space $\cT_{\xi_t}\wlac$ such that $\partial_t \mathbf{F}(\xi_t) = \langle \gradWL \mathbf{F}(\xi_t),V_t\rangle_{\xi_t}$ (where the inner product is defined by \eqref{eq:lifted_ke_norm}). We can now identify the $L^2$ over Wasserstein gradient as a random gradient with Wasserstein gradient sample paths, making the $\WL$ structure the natural extension of the Wasserstein space to allow for random dynamics.
\begin{theorem}[$\WL$ gradient]\label{thm:LiftedgradW}
    The function $\gradW_{L^2}\mathbf{F}(\xi):\Omega\times\R^d\to\R^d$ given by
    \begin{equation}\label{eq:gradW}
        \gradW_{L^2}\mathbf{F}(\xi)\,(\om) = \gradW \cF(\xi_\om) = \nabla \frac{\delta\cF}{\delta\mu} (\xi_\om)
    \end{equation}
    is the $L^2$ over Wasserstein gradient of $\mathbf{F}$ at $\xi\in\wlac$ under suitable integrability conditions. Phrased differently, the lifted $\WL$ gradient is a random gradient with Wasserstein gradient sample paths.
\end{theorem}
The proof, detailed in Appendix \ref{App:thm:LiftedgradW}, identifies the tangent vector \eqref{eq:gradW} according to Definition \ref{def:Lifted_tangent}. 

\begin{remark}
    Concerning measurability of the random gradient, we require that the map $(\mu,x)\mapsto \grad \frac{\delta\cF}{\delta\mu}(\mu)(x)$ be Borel measurable on $\Pac\times\R^d$ with respect to the Wasserstein and Euclidean topologies respectively. A common sufficient condition for this could be joint continuity of the functional for example.
\end{remark}

We arrive at the following definition for the random gradient flow
\begin{definition}[Random Wasserstein gradient flow]\label{def:lifted_w2_grad_flow}\index{Lifted Wasserstein gradient flow}

    Let $\{\cF_\om\}_{\om\in\Omega}$ be a collection of functionals on $\Pac$. ${(\xi_t)}_{t\ge 0}$ is the {Wasserstein gradient flow} of the lifted functional $\mathbf{F}$ if for each $\om\in\Omega$ it solves the following PDE in the weak sense.
    \begin{equation}\label{eq:WGF_lifted}
        \partial_t \xi_{\om,t} = \divergence\bigg(\xi_{\om,t}\, \gradWL \mathbf{F}(\xi_t)(\om)\bigg) = \divergence\bigg(\xi_{\om,t} \gradW \,\cF_\om\big(\xi_{\om,t}\big)\bigg),\qquad \text{$\PP$-a.s.}
    \end{equation}
      Phrased differently $\xi_t$ has sample paths which are Wasserstein gradient flows of $\cF_\om$, as in Definition \ref{def:w2_grad_flow}.
\end{definition}
One such condition for the gradient flow to be well-posed is provided below.
\begin{definition}[Wasserstein $C^1$ functional]\label{def:Wasserstein_C1}
    We say that $\cF:\Pc\to \R$ is Wasserstein $C^1$ if there exists a jointly continuous $C^1$ map
    \begin{equation}\label{eq:Wasserstein_C1}
        \frac{\delta \cF}{\delta\mu}:\Pc\times\R^d\to\R,\qquad (\mu,x)\mapsto\frac{\delta \cF}{\delta \mu}(\mu,x),
    \end{equation}
    where continuity is in $W_2$ for $\mu$ and Euclidean for $x$.
\end{definition}

Under this condition, we provide the following useful Grönwall estimate, a classical Dobrushin \cite{dobrushin1979gaussian} type result.
\begin{proposition}[Propagation of chaos]\label{prop:WGF_lifted_stability}
    Let $\cF$ be Wasserstein $C^1$ as above in Definition \ref{def:Wasserstein_C1} with $\gradW \cF$ globally Lipschitz in $W_2\times\R^d$. Let $\xi_{\om,t}, \eta_{\om,t}$ denote the Wasserstein gradient flow paths of $\cF$ as in \eqref{eq:WGF_lifted} initialised pathwise at $\xi_{\om,0}, \eta_{\om,0}\in\Pac$ respectively for $\xi_0,\eta_0\in\wlac$. \newline
    
    Then, for each fixed $T>0$, there is a Lipschitz constant $L>0$ such that,
    \begin{equation*}
        a)\, \sup_{t\in[0,T]}W_2^2(\xi_{\om,t},\eta_{\om,t})\leq e^{LT}W_2^2(\xi_{\om,0}, \eta_{\om,0})\,\,\,\text{$\PP$-a.s.}
        \qquad b)\, 
        \sup_{t\in[0,T]}d^2(\xi_t,\eta_t)\leq e^{LT}d^2(\xi_0,\eta_0).
    \end{equation*}
\end{proposition}
The proof in Appendix \ref{App:prop:WGF_lifted_stability} proceeds via a standard argument similar to Iacobelli's in \cite{Iacobelli2022}, coupling the initialisations, and then using the Lipschitz bound to access Grönwall's inequality on their propagations.

The stability result, in the language of random dynamics, provides a propagation of chaos statement. For example, when initialising via an empirical measure in Section \ref{sec:4} and via a Bayesian posterior in Section \ref{sec:5} to approximate a population measure $\mu$ under uncertainty, the stability provides a way to discuss how initial uncertainty around the true parameter is propagated as the gradient flow progresses. These, for example, enable an analysis of conditional flow matching as in Section \ref{subsec:6:COT} under a random conditional distribution, or self-attention dynamics for transformers as in Section \ref{subsec:6:TSA} under random token sampling.

\subsection{Wasserstein over Wasserstein random variables} \label{subsec:3:l2wow}

Targeting non-parametric setups, we also consider Wasserstein \textit{over} Wasserstein valued random measures, where the samples are probability distributions over probability distributions. An analogous definition can be made at the superposed level.
\begin{definition}[$L^2$ Wasserstein over Wasserstein space]\label{def:WoW2L2}
    Let $\WWL\deq L^2\l(\Omega;(\Pcpc,W_{W_2})\r)$ denote the subset of hierarchical random probability measures which
    \begin{itemize}
        \item are measurable from $(\Omega,\cA)\to \l(\Pcpc, \cB_{W_{W_2}}\r)$ where $\cB_{W_{W_2}}$ is the Borel $\sigma$-algebra induced by the $W_{W_2}$ topology on $\Pcpc$,
        \item are square-integrable i.e. 
    $$\cE_2^{(2)}(\Lambda) \deq \E_\om[\cE_2(\Lambda_\om)] = \E_{\om}\l[\int_{\Pc}\cM_2(\mu)\Lambda_\om(\d\mu)\r] <\infty.$$
    \end{itemize}
    We refer to this as the $L^2$ Wasserstein over Wasserstein space of random probability measures, endowed with the $L^2$ metric, $
        d_{{WoW}}(\Lambda_1,\Lambda_2) =\E_{\om}\l[W_{W_2}^2\big(\Lambda_1(\om),\Lambda_2(\om)\big)\r]^{\frac{1}{2}}.$
\end{definition}
We recover a similar Vitali-type equivalence from Theorem \ref{thm:Vitali_consistency}, now at the superposed level.
\begin{theorem}[WoW Vitali equivalence]\label{thm:Vitali_consistency_lifted}
    $\dw\l(\Lambda_n, \Lambda\r)\to0$ if and only if the following two conditions hold:
    \begin{enumerate}[label=(\alph*)]
    \item Consistency in Probability: $\PP\l( W_{W_2}^2(\Lambda_n,\Lambda) > \varepsilon\r)\rightarrow 0 \quad \forall \varepsilon > 0$
    \item Uniform Integrability: $\l\{\cE_2(\Lambda_n)\r\}_{n\geq1}$ is uniformly integrable so that \begin{equation*}\label{eq:Uniform_integrability}
        \lim_{R \to \infty} \sup_{n\geq1} \, \E \l[\int_{\cM_2(\mu)>R}\cM_2(\mu)\,\Lambda_n(\d \xi)\,\xi_n(\d\mu)\r] = 0.
    \end{equation*}
    \end{enumerate}
\end{theorem}
Once again, we need to ensure that the selection of any optimal coupling is indeed measurable. This follows from the statement of Theorem \ref{thm:Measurability}, given that the Wasserstein space is itself Polish (Theorem \ref{thm:Polish}).
\begin{proposition}[Measurability of the WoW coupling]
    Given a probability space $(\Omega, \cA, \PP)$ and two hierarchical random measures $\zeta, \Lambda \in \WWL$, there exists a measurable selector $\om \mapsto \gamma_\om$, where $\gamma_\om$ is a WoW optimal transport plan between $\zeta_\om$ and $\Lambda_\om$.
\end{proposition}

The random dynamics from Sections \ref{subsec:3:Geodesics} and \ref{subsec:3:RGFs} can also be studied similarly at the superposed level $\WWL$, the details of which can be found in Appendix \ref{App:L2WoW} for completeness. As mentioned, the level of $\WWL$ is appropriate for hierarchical statistical setups, which are explored in Sections \ref{subsec:4:hierarchical} and \ref{subsec:5:consistency}. We then use the results of those sections to perform a statistical analysis on the WoW representation of labelled datasets from \cite{bonet2025flowing} in Section \ref{subsec:6:flowing_datasets}.

\section{Statistical analysis with empirical measures}\label{sec:4}
In this section we collect various statistical consistency results of optimal transport between empirical measures. Let $\mu\in\Pc$ and let $X_1,\dots,X_n\overset{i.i.d.}{\sim}\mu$ be a sequence of draws from the population measure $\mu$. Define the sample empirical measure $
    \hmuno = \frac{1}{n} \sum_{i=1}^n \delta_{X_i(\om)}.$
According to Definition \ref{def:rm_space}, this is a random probability measure in $\WL$. Any realisation has bounded support hence finite second moment; the map $\om\mapsto \hmun(\om,B) = \frac{1}{n}\sum_{i=1}^n\1_{X_i(\om)\in B}$ is measurable as a sum of measurable indicators; and under the i.i.d. assumption on the sample, the second energy $\cE_2(\hmun) = \cM_2(\mu)<\infty$ is unbiased. For a general review of empirical measures, and empirical processes, the reader is directed to \cite{Dudley_1999}.

The developments in Section \ref{sec:3} allow us to consider statistics concerning the optimal transport machinery, both at the sample path level, and the lifted level of $\WL$. We establish a Wasserstein Law of Large Numbers, statistical consistency of geodesic interpolants, and a propagation of chaos result for gradient flow paths initialised with empirical measures. Moreover, we provide a Law of Large Numbers at the superposed Wasserstein over Wasserstein level. The section unifies both new and existing results under a single framework for frequentist optimal transport statistics.

\subsection{Statistical consistency for random empirical measures}\label{subsec:4:Convergence}

Given that $\hmun$ is an estimator for $\mu\in\Pc$, we expect large sample convergence as $\ninf$. In the i.i.d. case, these are established results.
\begin{theorem}[\cite{Panaretos2020AnIT} Proposition 2.2.6., Lemma 4.7.1.]\label{thm:empirical_convergence}
    Let $\mu\in\Pc$. If $\hmun$ is sampled using i.i.d. draws from $\mu$: a) $W_2(\hmuno,\mu)\to 0$ $\PP$-a.s., and b) $d(\hmu_n,\mu)\to 0$.
\end{theorem}
A proof follows after making use of the topological equivalence in Theorem \ref{thm:weak_wass} for a) to use Varadarajan's Law of Large Numbers for the weak topology \cite{varadarajan1958weak}, and then using uniform integrability of the convergent sequence to lift to b) in accordance with Theorem \ref{thm:Vitali_consistency}.

The above statements conclude appropriate convergence of the empirical measure when the measure $\mu$ is fixed. In fact, on any compact set $\cK\subset\R^d$ we are guaranteed \emph{uniform} $L^2$ convergence of empirical measures to their population measure $\mu\in \cP_2(\cK)$.

\begin{proposition}[Uniform convergence on compact spaces]\label{prop:L2W_Uniform_empirical}
    Let $\cK\subset\R^d$ be compact. Then the convergence of the empirical measure to its population measure in $L^2_W(\cK)$ is uniform, i.e.
    $b(n)\deq \sup_{\mu\in\cP_2(\cK)}d(\hmun,\mu)\to0.$
\end{proposition}
The proof given in Appendix \ref{App:prop:L2W_Uniform_empirical} makes use of the Kantorovich-Rubinstein duality from Theorem 5.10 of \cite{villani2009optimal} to access an Arzelá-Ascoli argument over the compact space $\cK$.
Under additional conditions, the convergence above admits rates. These are provided for general $p\geq 1$ by Fournier and Guillin in Theorem 1 of \cite{fournier2015rate} as below\footnote{The authors additionally present the results in the case $p<1$ but this does not give rise to a Wasserstein distance.}, which asks only that $\mu$ has a $q^{\text{th}}$ moment for some $q>p$.
\begin{theorem}[\cite{fournier2015rate} Theorem 1]\label{thm:FG}
    Let $p\geq1$ and let $\mu\in\cP_q(\R^d)$ for some $q>p$. Then there exists $C= C(p,d,q)$ such that for all $n\geq1$
    \begin{equation}\label{eq:FG_rate}
        \E\l[W_p^p(\hmu_n,\mu)\r] \leq C \cM_q^{p/q}(\mu)\, r_p(n,d,q),
    \end{equation}
    where
    \begin{equation*}
        r_p(n,d,q) \deq
        \l\{\begin{array}{ll}
        n^{-\frac{1}{2}} +n^{-(q-p)/q}& \!\!\!\hbox{if $p>\frac{d}{2}$ and $q\ne 2p$},  \\
        n^{-\frac{1}{2}} \log(n)+n^{-(q-p)/q} &\!\!\! \hbox{if $p=\frac{d}{2}$ and $q\ne 2p$}, \\
        n^{-p/d}+n^{-(q-p)/q} &\!\!\!\hbox{if $p\in (0,\frac{d}{2})$ and $q\ne d/(d-p)$}.
        \end{array}\r.
    \end{equation*}
    In the $\WL$ framework we work with $p=2$, and denote $r(n,d,q)\deq r_2(n,d,q)$, so that
    \begin{equation*}
        \E\l[W_2^2(\hmun,\mu)\r] \leq C \cM_q^{2/q}(\mu)\, r(n,d,q).
    \end{equation*}
\end{theorem}

The authors show that the rates are indeed sharp, providing examples where the bounds are achieved explicitly. The value $d=4$ provides a phase transition above which the high-dimensional rate slows. Using these rates, we can establish an elementary confidence ball for the empirical measure using Markov's inequality, recalling from Definition \ref{def:rm_space} that the $W_2$ distance between two random probability measures is measurable and non-negative, hence Markov's inequality is well defined.
\begin{proposition}[Markov concentration]
    Let $\mu$ satisfy the Fournier-Guillin conditions of Theorem \ref{thm:FG} above. Then from Markov's inequality, we have the following concentration bound
    \begin{equation*}\label{FG_Markov}
        \PP\big(W_2(\hmun, \mu) > \varepsilon\big) \leq \frac{d^2(\hmun,\mu)}{\varepsilon^2} \leq \frac{C(q,d)\big(1+\cM_q(\mu)\big)}{\varepsilon^2}r(n,d,q).
    \end{equation*}
\end{proposition}
However, under stronger sub-Gaussian conditions on $\mu$, we find a generally stronger McDiarmid concentration bound. \begin{definition}[Sub-Gaussian measure]\label{def:Sub-Gaussian}
    A probability measure $\mu\in\cP(\R^d)$ is $\kappa$-sub-Gaussian if
    \begin{equation*}\label{eq:K_Sub-Gaussian}
        \int e^{\langle \theta, x\rangle}\mu(\d x) \leq e^{\frac{\kappa^2\|\theta\|^2}{2}},\qquad \forall\theta\in\R^d.
    \end{equation*}
    The sub-Gaussian constant $\kappa(\mu)$ is the infimum of all such $\kappa>0$.
\end{definition}
Sub-Gaussian measures are explored in detail in Rigollet and Hütter's lecture notes on high-dimensional statistics \cite{rigollet2023highdimensionalstatistics}. One finds the following concentration bound.
\begin{proposition}[\cite{Bandeira2015} Theorem 4.7]  \label{prop:lipschitz_sub-gaussian}
    Let $f:(\R^d)^n\to\R$ be $L$-Lipschitz, and let $X=(X_1,\dots,X_n)$ consist of $n$ independent $\kappa$-sub-Gaussian vectors. Then $f(X)$ satisfies
    \begin{equation*}\label{eq:sub-Gaussian_concentration}
        \PP\bigg(f(X) - \E[f(X)] > t\bigg) \leq \exp\l({-\dfrac{t^2}{2\kappa^2L^2}}\r).
    \end{equation*}
\end{proposition}
This allows us to formulate the following result.
\begin{proposition}[Sub-Gaussian concentration]\label{prop:W2_concentration} 
    Let $\mu$ satisfy the Fournier-Guillin conditions of Theorem \ref{thm:FG} above, and be $\kappa$-sub-Gaussian. Then for $\delta\in(0,1)$, 
    \begin{equation*}\label{eq:sub-gaussian_prop_concentration}
        \PP\bigg(W_2(\hmun,\mu)>\varepsilon_n(\delta)\bigg) \leq \delta \qquad \text{where}\quad \varepsilon_n(\delta) \deq
        r(n,d,q)^{\frac{1}{2}} + \kappa\sqrt{\frac{2\log({1/\delta})}{n}}.
    \end{equation*}
\end{proposition}
    This follows from Proposition \ref{prop:lipschitz_sub-gaussian} upon verifying that $f:(\R^d)^n\to\R$, given by
        $f(x_1,\dots,x_n) = W_2\l(\frac{1}{n}\sum_{i=1}^n\delta_{x_i},\,\mu\r)$
    is Lipschitz which is shown in Appendix \ref{App:prop:W2_concentration}. The radius $\varepsilon_n(\delta)$ splits into a bias term, the Fournier-Guillin rate, and a concentration term, in the sub-Gaussian constant $\kappa$.\newline

In addition to these concentration bounds, del-Barrio and Loubes \cite{Barrio2017CentralLT} also demonstrate a Central Limit result, which governs asymptotic concentration bounds under slightly stronger conditions of absolute continuity to leverage the Kantorovich potentials in the dual formulation of the optimal transport problem.
\begin{theorem}[\cite{Barrio2017CentralLT} Theorem 4.1]\label{thm:dBL}
    Let $\mu,\nu\in\Pac$ such that $\cM_{4+\delta}(\mu),  \cM_{4+\delta}(\nu)<\infty$ for some $\delta>0$. Additionally, assume that the optimal Kantorovich potential $\varphi\in C^2(\R^d)$ is strictly convex with $\grad^2\varphi$ bounded away from zero. Let $\varphi^c$ be the convex-conjugate potential so that $(\varphi,\varphi^c)$ are optimal for the dual Kantorovich problem as in \cite{Beiglbock2009GeneralDT}. Denote
    \begin{equation*}\label{eq:Kantorovich_variance}
        \sigma^2(\mu,\nu) \deq \int_{\R^d}\l(\|x\|^2 - 2\varphi(x)\r)^2 \,\mu(\d x) - \l(\int_{\R^d} \l(\|x\|^2 - 2\varphi(x)\r)\,\mu(\d x)\r)^2.
    \end{equation*}
    Then assuming $n,m\to\infty$ such that $n/(n+m)\to\lambda\in(0,1)$, with $\hmu_n, \hnu_m\in\WL$ as above,
    \begin{equation*}\label{eq:var_convergence_clt}
        \frac{nm}{n+m}\Var \l(W_2^2(\hmu_n,\hnu_m)\r) \longrightarrow (1-\lambda)\sigma^2(\mu,\nu) + \lambda\sigma^2(\nu,\mu)
    \end{equation*}
    and
    \begin{equation*}\label{eq:clt_convergence}
        \sqrt{\frac{nm}{n+m}}\bigg(
            W_2^2\l(\hmu_n,\hnu_m\r) - d^2\l(\hmun,\hnun\r)\bigg)\, \overset{w}{\longrightarrow}\, \cN \bigg(0\,,(1-\lambda)\sigma^2(\mu,\nu) + \lambda\sigma^2(\nu,\mu)\bigg).
    \end{equation*}
\end{theorem}
In the case $n=m$, the limiting variance $\frac{1}{2}\sigma^2(\mu,\nu) + \frac{1}{2}\sigma^2(\nu,\mu)$ is shown in Proposition 2 of Goldfeld et al. \cite{GoldfeldKato} to match the semi-parametric efficiency bound for estimating $W_2^2(\mu,\nu)$ from i.i.d. data, making the empirical plug-in estimator $W_2^2(\hmu_n,\hnu_m)$ for the distance $W_2^2(\mu,\nu)$ asymptotically efficient.

\subsection{Empirical geodesics}\label{subsec:4:empirical_geodesics}

As empirical measures have discrete realisations, we make use of the more general pathwise coupling interpolation of Proposition \ref{prop:Wasserstein_is_geodesic}. For an optimal coupling $\gamma_{\om,n}$ between the empirical measures, we define the pathwise geodesic
\begin{equation*}
    \eta_{\om,t}^n \deq \l[(1-t)\pi^1 + t\pi^2\r]_\#\gamma_{\om,n}.
\end{equation*}
We find that the convergence of endpoints $\hmun$ and $\hnun$ to $\mu$ and $\nu$ respectively, as provided in the last section, delivers uniform convergence of interpolants between them, both at the pathwise level, and in $\WL$.
\begin{theorem}[Geodesic consistency]\label{thm:geodesic_empirical_convergence_uniform}
    Let $\mu,\nu \in \Pac$, and use $\eta_t$ to denote the displacement interpolation from Definition \ref{def:w2_geod} between the two. Then the empirical geodesic $\eta_t^n$ converges to $\eta_t$ uniformly in $t$, both a) pathwise:  $\sup_{t\in[0,1]}W_2(\eta_{\om,t}^n,\eta_t)\to 0$ $\PP$-a.s. , and b) at the $L^2$ level: $\sup_{t\in[0,1]}d(\eta_t^n,\eta_t)\to 0$.

\end{theorem}
A proof is given in Appendix \ref{App:thm:geodesic_empirical_convergence_uniform}, where we show that convergence of the endpoints in Wasserstein space provides convergence of the interpolants, letting us conclude by the LLN consistency results of the previous section after suitable arguments.

In the general case, it is not simple to produce rates of convergence for the geodesic interpolants, given the possibility that the optimal coupling is not unique in the discrete empirical case.
Under stronger assumptions, as in the theorem below from Li and Nochetto \cite{NochettoLi}, one can pass the Fournier-Guillin rates of convergence to the endpoints to the geodesic interpolants.
\begin{theorem}[\cite{NochettoLi} Corollary 3.9]\label{thm:Li_Nochetto}
    Let $\mu,\nu\in\Pac$ and let $\nv$ be the Brenier map from $\mu$ to $\nu$. Suppose that $\varphi$ is $\lambda$-regular for some $\lambda>0$ i.e. $\frac{\lambda}{2}\|x\|^2-\varphi(x)$ is convex. 

    Let $\gamma = (\id,\nv)_\# \mu$ be the optimal transport plan induced by $\nv$. Let $\mu_h, \nu_h \in P_2(\R^d)$ be approximations to $\mu,\nu$ and $\gamma_h \in \Gamma_{\mu_h, \nu_h}$ be a transport plan between $\mu_h$ and $\nu_h$. Then we have
	\begin{equation}\label{eq:Li_Nochetto_bound}
    	W_2(\gamma, \gamma_h) \le 2 \, \lambda^{\frac12} \,
    	\tilde{e}_h^{\frac{1}{2}} \Big( W_2(\mu,\nu) + \tilde{e}_h \Big)^{\frac{1}{2}} + e_h,
    \end{equation}
	where $e_h = W_2(\mu, \mu_h) + W_2(\nu, \nu_h)$ and $\tilde{e}_h$ is given by
    \begin{equation*}\label{eq:OT_stability_epsilon}
  \tilde{e}_h \deq e_h + \frac{\varepsilon_h}{2},
  \quad
  \varepsilon_h \deq \l( \int \l\|x' - y'\r\|^2 \d \gamma_h(x',y') \r)^{\frac{1}{2}} - W_2(\mu_h, \nu_h).
\end{equation*}
\end{theorem}
Note that in the case $\gamma_h\in\Gamma_{\nu_h\mu_h}$ is an optimal coupling for $(\mu_h,\nu_h)$, $\varepsilon_h = 0$ and $\tilde e_h = e_h$ is the sum of approximation errors. Under these conditions, the rate of uniform convergence of the empirical geodesics to the true geodesic in the metric $d$ is given by the Fournier-Guillin rate of Theorem \ref{thm:FG}.

\begin{proposition}[Geodesic consistency rate]\label{prop:LN_FG}
    Let $p\geq1$ and let $\mu,\nu\in\cP_q(\R^d)$ for some $q>p$. Under the condition of Theorem \ref{thm:Li_Nochetto}, the convergence of b) in Theorem \ref{thm:geodesic_empirical_convergence_uniform}, $\sup_{t\in[0,1]}d(\eta_t^n,\eta_t)\to0$ achieves the square root of the Fournier-Guillin rates of Theorem \ref{thm:FG}, 
        $\sup_{t\in[0,1]}d(\eta_t^n,\eta)\leq C r(n,d,q)^{\frac{1}{2}}$, 
    for some $C = C(p,d,q)$.
\end{proposition}
A calculation determining the rate is given in Appendix \ref{App:prop:LN_FG}.

\subsection{Gradient flow paths for empirical measures}\label{subsec:4:GFs}

We now turn to the dynamics of empirical measures in $\WL$. As empirical measures are defined to be discrete, they do not enjoy immediate access to the gradient flow dynamics of Section \ref{subsec:3:RGFs}. One ad-hoc resolution is to mollify the empirical measure via Gaussian convolution. A detailed exploration of this direction is provided by Goldfeld et al. in \cite{golfeldGreenewaldPolyanskiyNilesWeed}.

However, it turns out we can still apply the gradient flow theory for functionals which are Wasserstein $C^1$ as in Definition \ref{def:Wasserstein_C1}, without applying any mollification. In this case, the Wasserstein gradient admits a continuous extension from $\Pac$ to $\Pc$, and is then well defined for all $\mu\in\Pc, x\in\R^d$. 

We can now establish a propagation of chaos result for the random Wasserstein gradient flow, using convergence of the initialisations of the paths. One of the key features of the space $\WL$ is that it provides the necessary structure to discuss a canonical random gradient flow for the random empirical measure. Recall from Definition \ref{def:lifted_w2_grad_flow} that given a functional $\cF$ the random gradient $\gradW_{L^2}\E[\cF(\xi)]$ is the pathwise Wasserstein gradient $\gradW\cF(\xi_\om))$, hence the gradient flow initialised at the random sample $\hmuno$ satisfies
\begin{equation}\label{eq:empirical_gradient_flow}
    \frac{\partial}{\partial t}\hmu_{\om,t}^n = \divergence\bigg(\hmu_{\om,t}^n\,\gradW\cF\big(\hmu_{\om,t}^n\big)\bigg)= \divergence\bigg(\hmu_{\om,t}^n\,\grad\delta\cF\big(\hmu_{\om,t}^n\big)\bigg),\qquad \text{$\PP$-a.s.}
\end{equation}
which is well-posed in $\Pc$ under suitable conditions on $\gradW \cF$. In particular, we have the following Grönwall estimate.
\begin{proposition}[Empirical propagation of chaos]\label{prop:WGF_empirical}
    Let $\cF$ be Wasserstein $C^1$ as in Definition \ref{def:Wasserstein_C1} with $\gradW \cF$ globally Lipschitz in $W_2$. Let $\mu_t$ denote the Wasserstein gradient flow of $\cF$ initialised at $\mu_0\in\Pac$ and let $\hmu_{\om,t}^n$ denote the gradient flow in \eqref{eq:empirical_gradient_flow} above initialised at $\hmuno$. Then, for each fixed $T>0$:
    \begin{equation*}
        a)\,
        \sup_{t\in[0,T]}W_2(\hmuno,\mu_t)\to 0,\quad \text{$\PP$-a.s.}\qquad
        b)\,
        \sup_{t\in[0,T]}d(\hmu_t^{n},\mu_t)\to 0
    \end{equation*}

\end{proposition}

This follows immediately from the Grönwall estimate in Proposition \ref{prop:WGF_lifted_stability}, given the convergence of the initialisations in Theorem \ref{thm:empirical_convergence}. Under the additional moment assumptions of Theorem \ref{thm:FG}, the $L^2$ convergence in b) follows the Fournier-Guillin rate \eqref{eq:FG_rate} in $n$, though deteriorates exponentially in time $T$.

\subsection{Hierarchical empirical measures}\label{subsec:4:hierarchical}

We now turn to empirical convergence results on Wasserstein over Wasserstein space, where the random objects live in $\WWL$ with realisations in the WoW space. \newline

Varadarajan's Theorem \cite{varadarajan1958weak} is stated for general Polish spaces, hence we can immediately state a Law of Large Numbers at the level of $\Pcpc$, where the samples arrive in $\Pc$.

\begin{theorem}[Wasserstein over Wasserstein Law of Large Numbers]\label{thm:WoWLLn}
    Let $\Q\in\Pcpc$ be a random measure law, and let $\mu_1,\dots,\mu_n\sim\Q$ be an i.i.d. data sequence of probability measures in $\Pc$. Then the empirical measure $\hQn\in\WWL$, $\hQn(\om) \deq \frac{1}{n} \sum_{i=1}^n\delta_{\mu_i(\om)}$ converges to the population random measure law $\Q$ almost surely in $W_{W_2}$, and in $\WWL$.
\end{theorem}
The proof in Appendix \ref{App:thm:WoWLLn} makes use of the equivalence in Theorem \ref{thm:weak_wass} showing both weak convergence and convergence of the second moments at the WoW level. The application of the Law of Large Numbers in the former requires the samples themselves to be (deterministic) probability measures. However, for practical purposes, the data one observes are almost always $\R^d$-valued (even if we use a parametric model to immediately deliver a probability measure), and so we instead approximate the random measure law $\Q$ using an appropriate sample from $\R^d$. 

Let $\Q\in\Pcpc$ and consider a hierarchical sample where we sample probability measures $\mu_j$ in $\Pc$ from $\Q$, and then sample points $x_{j,i}$ in $\R^d$ from these probability measures, observing only the points $x_{j,i}$.
\begin{align*}
    \mu_1,\dots,\mu_J \overset{i.i.d.}{\sim} \Q\qquad &\text{Unobserved, latent}\\
    Z_j\deq X_{j,1},\dots, X_{j,n_j}\overset{i.i.d.}{\sim}\mu_j\qquad &\text{Observed}
\end{align*}
The latent WoW empirical measure is given by the empirical measure of the latent distributions $\tilde\Q$ below. For each group in the sample, indexed by $j$, the group empirical measure is $\hmu_j\deq \frac{1}{n_j}\sum_{i=1}^{n_j}\delta_{X_{j,i}}$, leading to the Wasserstein over Wasserstein empirical measure $\hQ$
\begin{equation*}
    \tilde\Q_J\deq\frac{1}{J}\sum_{j=1}^J \delta_{\mu_j}, \qquad \hat\Q_J\deq \frac{1}{J}\sum_{j=1}^J \delta_{\hmu_j}
\end{equation*}
We let $n_j\to\infty$ be an increasing sequence, determining the size of all grouped samples. Given the formulation, $\hQ_J\in \WWL$, with realisations in WoW.

\begin{remark}\label{remark:Not_intensity}
    It is important to note that the sampling scheme above is not the same as sampling from the intensity measure of the random measure law $\int \mu \Q(\d \mu)$. I.i.d. sampling from the intensity measure delivers i.i.d. samples $x_k$. However, in the hierarchical sample, dependence structure is embedded differently between each group, as the groups $Z_j$ are sampled from different $\mu_j$.
\end{remark}
The hierarchical sample assumes an inherent grouping of the samples - a natural case to consider is a dataset with natural class labels, and class-conditional distributions. This is explored in depth by Bonet et al. in \cite{bonet2025flowing} who use a class conditional representation of a data sample - we demonstrate the consistency of this representation in Section \ref{subsec:6:flowing_datasets}.

Clearly, a Law of Large Numbers applies to the intensity measure over $\R^d$ at the regular Wasserstein, deterministic measure level via Theorem \ref{thm:empirical_convergence}. However, given our hierarchical setup, we also find a Law of Large Numbers at the level of WoW random measure laws too, even where the data are $\R^d$-valued.

\begin{theorem}[Hierarchical WoW LLN]\label{thm:Hierachical_WoW_LLN}
    Let $\Q\in\Pcpc$ be a random measure law. With the hierarchical sampling scheme above, and $n_j\overset{j}{\to}\infty$, $\E\l[W_{W_2}^2\l(\hQ_J,\Q\r)\r]\to 0.$
\end{theorem}
The proof in Appendix \ref{App:thm:Hierachical_WoW_LLN} first approximates the WoW measure $\P$ with the latent empirical WoW measure $\tilde\Q$, approximating the latter with the empirical WoW measure $\hQ$, and then combines the approximations to arrive at the convergence statement.\newline

Given that the Fournier-Guillin rates are not summable, we cannot expect in general that the convergence above is almost sure. However, we can upgrade the above to almost sure convergence over compact spaces $\cK\subset\R^d$.

\begin{corollary}[Compact hierarchical convergence]\label{cor:compact_wow_upgrade}
    Let $\cK\subset\R^d$ be compact. Then the convergence in Theorem \ref{thm:Hierachical_WoW_LLN} is almost sure.
\end{corollary}
A proof is given in Appendix \ref{App:cor:compact_wow_upgrade}, which makes use of Kolmogorov's SLLN, as the variance of the distance between the sample empirical measures and the latents can be uniformly bounded on compact $\cK$. This result will prove useful in Section \ref{subsec:5:consistency} to show Wasserstein consistency of the Bayesian posterior on compact spaces when we use the hierarchical sampling scheme above.

\section{Bayesian analysis}\label{sec:5}
The Bayesian framework \cite{gelman2013bayesian}, offers a principled framework for reasoning under uncertainty: one begins with a prior belief about an unknown quantity, and updates the belief in light of observed data to form a posterior belief. In classical parametric statistics, the unknown quantity is a finite-dimensional parameter $\theta\in\R^d$, and the prior is a probability distribution over the space of possible parameters $\R^d$. In the non-parametric setting, the unknown quantity is itself a probability measure $\pi\in\cP(\R^d)$, and the prior is a distribution over the space of probability measures. Before observing any data, uncertainty about the data-generating distribution is encoded by a prior $\Prior$.
Given a random sequence $X_{1:n}$ drawn i.i.d. from $\pi$, Bayes’ theorem provides the posterior law $\Pi_{\om,n}(\cdot) \deq \Prior\big(\cdot \big|\, X_1(\om),\dots,X_n(\om)\big)$
Thus, $\Pi_{\om,n}$ is a measure-valued random variable.

In this section we target sufficient conditions for Wasserstein consistency, which asserts that after sufficiently many observations, we can conclude posterior convergence towards the population truth in Wasserstein distance. We first present results at the parametric level. Using the structure we define for $\WL$, we provide conditions for a parametric consistency rate and present a Wasserstein Bernstein-von Mises result. Then, we formulate Wasserstein consistency results at the \textit{non-parametric} level and, additionally, provide explicit non-parametric posterior contraction rates. Finally, we extend the posterior contraction rate results to Wasserstein mixture models.

\subsection{Parametric consistency}\label{subsec:5:parametrics}

We embed existing optimal transport based Bayesian parametric results  into our statistical framework. Let $\Theta\subset\R^d$ be a Borel parameter space, and consider a dominated, identifiable statistical model given by 
    $\cM_\Theta = \l\{P_\theta:\theta\in\Theta\r\},$ $ P_\theta(\d x) = p_\theta(x)\lambda(\d x)$
where $\lambda$ denotes a Lebesgue density.

Consider i.i.d. data, drawn $X_i\overset{i.i.d.}{\sim}P_{\theta}$ from some true distribution $P_{\theta_0}$ where $\theta\in\Theta$ is fixed. Let $\Prior$ be a prior probability measure over the parameter space $\Theta$. As the parameter space is a subset of $\R^d$, and the posterior is a random probability measure over that parameter space, the posterior $\Postn$ is an element of $\WL$. We  thus consider the $L^2$ distance to the true parameter $\theta_0$
\begin{align*}
    d^2\l(\Postn, \delta_{\theta_0}\r) &= \E\l[W_2^2\l(\Postn, \delta_{\theta_0}\r)\r]= \E\l[\int_\Theta\|\theta-\theta_0\|^2\,\Prior(\d \theta|X_{1:n})\r],
\end{align*}
the reformulation a result of the fact that the Dirac mass only places mass at the true parameter $\theta_0$; thus the $\WL$ distance is the frequentist expectation of the posterior mean squared parameter error. We would like to find sufficient conditions under which this distance converges to $0$.

\begin{remark}
    Dolera et al. \cite{DoleraFavaroMainini2024} provide some strong results concerning the Wasserstein posterior contraction rate (PCR), which the authors define through the expected posterior Wasserstein radius, making use of the KL expansion$$\varepsilon_n\deq \E_{P_{\theta_0}}\l[W_p\l(\Pi_n(\cdot\mid X_{1:n}), \delta_{\theta_0}\r)\r],\qquad \KL(P_{\theta_0}\| P_{\theta}) = \frac{1}{2}(\theta - \theta_0)^T\cI(\theta_0)(\theta - \theta_0) + o\l(\|\theta - \theta_0\|^2\r)$$
where $\cI(\theta_0)$ is the Fisher information. Sufficient assumptions from \cite{DoleraFavaroMainini2024} Proposition 1 include:
\begin{itemize}
    \item The prior $\Prior$ has a Lebesgue density $q$, continuous at $\theta_0$ with $q(\theta_0)>0$, and has finite second moments, $\int_{\Theta}\|\theta\|^2\Prior(\d \theta)<\infty$
    \item The map $\theta\mapsto \KL(P_{\theta_0}\|P_\theta)$ is $C^2$ and its Hessian at $\theta_0$ is positive definite, equal to the Fisher information $\cI(\theta_0)$, where there is some $c(\delta)>0$ such that $\inf_{\|\theta_0-\theta\|\geq\delta} \KL\l(P_{\theta_0}\|P_{\theta}\r)\geq c(\delta),$  $\delta>0$.

\end{itemize}
The sufficient assumptions from \cite{DoleraFavaroMainini2024} Propositions 1, 3, 4 and Theorem 1 lead to the parametric rate $n^{-\frac{1}{2}}$ for the $L^1$ of the $2$-Wasserstein distance $W_2$. However, note that for $L^2$, we are concerned with the expected squared distance, a subtly different object which, by Cauchy-Schwarz, bounds the former from above.
\end{remark}

Under our framework, let $\hat\theta_n$ be a sequence of estimators for $\theta_0$. Define the two terms, 
\begin{equation*}
    V_n\deq \sqrt{n}\l(\hat\theta_n - \theta_0\r), \qquad H_n(\theta)\deq \sqrt{n}\l(\theta - \hat\theta_n\r).
\end{equation*}
which represent the sampling fluctuation of the estimator and the posterior spread about the estimator respectively. Given that $\sqrt{n}(\theta - \theta_0) = H_n(\theta) + V_n$, we find
    \begin{align}
        nd^2\l(\Postn,\delta_{\theta_0}\r) &= \E_{P_{\theta_0}}\l[\int_{\Theta}
            \|\sqrt{n}\l(\theta-\theta_0\r)\|^2\,\Postn(\d \theta|X_{1:n})
        \r] = \E_{P_{\theta_0}}\l[\int_{\Theta}
            \|H_n(\theta)+V_n\|^2\,\Postn(\d \theta|X_{1:n})
        \r]\notag \\
        &\leq2\E_{P_{\theta_0}}\l[\int_{\Theta}
            \|H_n(\theta)\|^2\,\Postn(\d \theta|X_{1:n})\r] 
            +
            2\E_{P_{\theta_0}}\l[\|V_n\|^2\r].\label{eq:sup_bound_param}
    \end{align}
Thus we can make use of the implicit existence of such an estimator to conclude a parametric rate for the $L^2$ convergence, as detailed in the proposition below.
\begin{proposition}[Parametric consistency rate]\label{prop:bayes_parametric_rate}
Suppose that the following two uniform integrability conditions hold
    \begin{align*}
        \sup_{n\geq1} \E_{P_{\theta}}\l[\|V_n\|^{2}\r]<\infty,\qquad \sup_{n\geq1}\E_{P_{\theta}}\l[\int_{\Theta}\|H_n(\theta)\|^{2}\,\Postn(\d \theta|X_{1:n})\r]<\infty.
    \end{align*}
    Then the posterior converges to the truth at the parametric rate, $d(\Postn,\delta_{\theta_0}) = O\big(n^{-\frac{1}{2}}\big)$,
    where the implied constant only depends on the two uniform moment bounds above.
\end{proposition}
The pair of uniform integrability conditions ensures that the RHS of \eqref{eq:sup_bound_param} is bounded uniformly in $n$, allowing for the conclusion.
Before moving to the implications of the convergence for Bayesian consistency of the optimal transport machinery, we
collect a Bernstein-von Mises type result, relating the existing Bernstein-von Mises Theorem from \cite{vandervaart1998asymptotic} to the $\WL$ setting, replacing the total variation distance with the Wasserstein.
\begin{theorem}[$\WL$ Bernstein-von Mises]\label{thm:WL_BvM}
    Define $\tilde{\Pi}_n \deq T_\#\Postn,$ for $T(\theta)\deq\sqrt{n}(\theta-\hat\theta_n)$, $V_n\deq \sqrt{n}\l(\hat\theta_n - \theta_0\r)$. Suppose that in addition to the assumptions in the Bernstein-von Mises Theorem from \cite{vandervaart1998asymptotic}, there is some $\eta>0$ such that the following uniform moment bounds hold,
    \begin{align*}
        &\sup_{n\geq1}\E_{P_{\theta_0}}\l[\int_{\R^d}\|z\|^{2+\eta}\,\Tilde{\Pi}_n(dz|X_{1:n})\r]<\infty,\qquad \sup_{n\geq1}\E_{P_{\theta_0}}\l[\|V_n\|^{2+\eta}\r]<\infty.
    \end{align*}
    Then the following Wasserstein Bernstein-von Mises result holds
    $nd^2\l(\Pi_n,\delta_{\theta_0}\r)\to 2 \mathrm{tr}\l(\cI(\theta_0)^{-1}\r).$
    In particular, the parametric rate is achieved. 
\end{theorem}
A proof is given in Appendix \ref{App:thm:WL_BvM}, making use of the uniform integrability to pass the classical Bernstein-von Mises's convergence in total variation to the $L^2$ level.

Returning to the consequences for optimal transport, Proposition \ref{prop:bayes_parametric_rate} lets us access statistical consistency for geodesics and gradient flows in a similar manner to empirical measures. We begin with the former.
\begin{proposition}[Bayesian geodesic consistency]\label{prop:Bayesian_geodesics}
    Let $\mu,\nu \in \Pac$, and use $\eta_t$ to denote the displacement interpolation from Definition \ref{def:w2_geod} between the two. Let $\Pi_{n,0}$ and $\Pi_{n,1}$ be posterior laws conditioned on samples from $\mu$ and $\nu$ respectively. Define the Bayesian geodesic to be $\Pi_{n,t}\deq\l[(1-t)\pi^1 + t\pi^2\r]_\#\gamma_{\om,n}$ for some optimal coupling $\gamma_{\om,n}$. Under the conditions of Proposition \ref{prop:bayes_parametric_rate} $\Pi_{n,t}$ converges to $\eta_t$ uniformly in $t$, both pathwise and at the $L^2$ level: 
    $$a)\,
    \sup_{t\in[0,1]}W_2(\Pi_{n,t},\eta_t)\to 0\quad\text{ $\PP$-a.s.},\qquad b)\,\sup_{t\in[0,1]}d(\Pi_{n,t},\eta_t)\to 0.$$
    Moreover, if the Li-Nochetto conditions of Theorem \ref{thm:Li_Nochetto} are satisfied, then the convergence in b) is at the square root of the parametric rate $O\big(n^{-\frac{1}{4}}\big)$.
\end{proposition}

This follows from the same analysis for empirical measures in Appendices \ref{App:thm:geodesic_empirical_convergence_uniform} and \ref{App:prop:LN_FG}. In the gradient flows setting, if we initialise a posterior law over the parameter space, then the initial uncertainty is propagated through the flow.
\begin{proposition}[Bayesian propagation of chaos]\label{prop:Bayesian_propagations}
    Let $\cF$ be Wasserstein $C^1$ as in Definition \ref{def:Wasserstein_C1} with $\gradW \cF$ globally Lipschitz in $W_2$. Let $\mu_t$ denote the Wasserstein gradient flow of $\cF$ initialised at $\mu_0\in\Pac$ and let $\Pi_{n,t}$ denote the gradient flow initialised at the posterior law $\Postn(\cdot)$. Under the conditions of Proposition \ref{prop:bayes_parametric_rate}. for each fixed $T>0$:
    \begin{align}
        a)\, \sup_{t\in[0,T]}W_2(\Pi_{n,t},\mu_t)&\to 0\quad \text{for $\PP$-a.s.},\qquad
        b)\, \sup_{t\in[0,T]}d(\Pi_{n,t},\mu_t)\to 0.
    \end{align}
\end{proposition}
This is again a consequence of the Grönwall estimate in Proposition \ref{prop:WGF_lifted_stability}. The $L^2$ convergence in b) again follows the parametric rate $O\big(n^{-\frac{1}{2}}\big)$ in $n$, albeit deteriorating exponentially in time $T$.

\subsection{Wasserstein consistency}\label{subsec:5:consistency}

We now focus our attention on the non-parametric level, assuming a prior $\Prior\in\Pcpc$ over the space of probability distributions. Bayesian consistency is the phenomenon of a Bayesian posterior distribution concentrating around the true parameter value as the sample size $\ninf$. Formally, a sequence of posterior distributions is consistent at $\pi$ if for any neighbourhood $U$ of $\pi$ in $\cP(\R^d)$,
$\Pi_{\om,n}(U^c)\overset{\ninf}{\longrightarrow}0,$ $\PP$-a.s.
i.e. the mass that the posterior law assigns to the complement of any neighbourhood of the truth vanishes asymptotically. To consider \textit{Wasserstein} consistency, the neighbourhoods $U$ are open sets of the Wasserstein topology explored in the preceding sections: the condition can be stated in terms of Wasserstein balls $B_\varepsilon(\pi)\subset \Pc$.
\begin{definition}[Wasserstein posterior consistency]\label{def:w2_consistency}
    A sequence of random posterior laws $(\Pi_{\om,n})_{n \geq 1}$ is Wasserstein consistent at $\pi \in \Pc$ almost surely if, for every $\varepsilon > 0$, 
    \begin{equation*}
        \Postn\left( \left\{ \mu \in \Pc : W_2(\mu, \pi) > \varepsilon \right\} \right) = \Postn\left( B^c_\varepsilon(\pi)  \right) \overset{n\to\infty}{\longrightarrow} 0 \quad \text{$\PP$-a.s.}
    \end{equation*}
\end{definition}

Embedding the notion of Wasserstein consistency into our framework, we would also like to provide $L^2$ convergence statements i.e. we search for conditions such that $\dw (\Postn, \delta_\pi)\to0$ as $n \rightarrow \infty$ in light of the equivalence in Theorem \ref{thm:Vitali_consistency}, first targeting the consistency in the probability. Towards this, we discuss the following classical theorem of Schwartz in the dominated case, which requires the pair of definitions below.
\begin{definition}[KL support]\label{def:KL_support}
    Given a  prior $\Prior$ over $\cP(\cX)$, and a probability measure $\pi\in\cP(\cX)$, $\pi$ is said to be in the KL support of $\Pi$ if for all $\varepsilon>0$,
    \begin{equation}\label{eq:KL_Support}
        \Prior\l(\Big\{\mu\in\cP(\R^d):\KL\l(\pi\|\mu\r)\leq \varepsilon\big\}\r) > 0.
    \end{equation}
\end{definition}

Informally, the KL condition ensures that the prior distribution places positive mass on models arbitrarily close to $\pi$, where close is understood in an entropic sense. The other condition asks for a uniformly consistent sequence of test functions.
\begin{definition}[Testing conditions]\label{def:Bayes_testing}
    An open set $U\subset\cP(\cX)$ is said to satisfy the testing conditions for a population measure $\pi\in\cP(\cX)$ if $U$ is a neighbourhood of $\pi$ such that there is a sequence of test functions $(\Phi_n)_{n\geq1}$ satisfying
    \begin{equation*}
        \E_{\pi}[\Phi_n(X_{1:n})] \leq Ce^{-cn},\qquad \sup_{\mu\in U^c} \E_{\mu}[1-\Phi_n(X_{1:n})]\leq Ce^{-cn},
    \end{equation*}
    where $(X_n)_{n\geq1}$ is the i.i.d. sequence of draws from $\pi$ and $c,C>0$ are a pair of positive constants.
\end{definition}

The testing conditions informally assert that the hypothesis $H_0:\mu=\pi$ should be testable against complements of neighbourhoods of $\pi$. We may now formulate Schwartz's theorem.
\begin{theorem}[\cite{Schwartz1965OnBP}]\label{thm:Schwartz}
    Let $\cX$ be a Polish space, and let $\pi\in\cP(\cX)$ be the true population measure for which $X_i\overset{\text{i.i.d.}}{\sim}\pi$. Suppose that $\Prior$ is a prior over distribution over the space of possible models $\cP(\cX)$. If $\pi$ is in the $\KL$ support of $\Pi$ (Definition \ref{def:KL_support}) and every neighbourhood $U$ of $\pi$ satisfies the testing conditions (Definition \ref{def:Bayes_testing}), then the posterior is consistent at $\pi$, i.e. for every neighbourhood $U$ of $\pi$, 
        $\Postn\big( U^c \mid X_{1:n}\big) \to 0$, $\PP$-a.s.
    where $\PP$ is the joint law of $(X_n)_{n\geq1}$.
\end{theorem}
Endowing $\cP(\cX)$ with the weak topology, the tests $(\Phi_n)_{n\geq1}$ \textit{always} exist for any weak neighbourhood of $\pi$, which will prove useful.
\begin{corollary}[\cite{Schwartz1965OnBP}]\label{cor:schwartz}
    Let $\cX$ be a Polish space, and let $\pi$ be the population measure for random draws $X_i\overset{\text{i.i.d.}}{\sim}\pi$. Suppose $\Prior$ be a prior over $\cP(\cX)$. If $\pi$ is in the $\KL$ support of $\Prior$ as in Definition \ref{def:KL_support}, then the posterior is weakly consistent at $\pi$, i.e. for every weak neighbourhood $U$ of $\pi$, the posterior assigns vanishing mass to its complement, $\Prior\big(U^c| X_{1:n}\big) \to 0$ $\PP$-a.s.
\end{corollary}
Recalling the equivalence of Theorem \ref{thm:weak_wass}, combining the above corollary with the convergence of the posterior moments will deliver the Wasserstein consistency we seek. The following definition provides a sufficient condition on the moments for the latter.
\begin{definition}[Quadratic-tail tightness]\label{def:Posterior_QTT}
    Let $R>0$ and $\mu\in\Pc$. Let 
    \begin{equation*}
        T_R(\mu)\deq \int_{\{\|x\|>R\}}\|x\|^2\,\mu(\d x).
    \end{equation*}
    A posterior sequence is said to be quadratic-tail tight at $\pi$ if for every $\eta>0$,
    \begin{equation*}\label{eq:P_QTT}
        \lim_{R\to\infty} \limsup_{n\to\infty} \Postn\left( \left\{ \mu \in \Pc : T_R(\mu) > \eta \right\} \right) = 0 \quad \text{$\PP$-a.s.}
    \end{equation*}
\end{definition}
Note that $\lim_{R\to\infty}T_R(\pi)=0$ for any $\pi\in\Pc$, after applying dominated convergence to the integrand $\|x\|^2\1_{\{\|x\|>r\}}$. Note that on restriction to compact spaces, quadratic-tail tightness and uniform integrability are automatically satisfied.
The conditions for Schwartz's theorem in combination with the quadratic-tail tightness provide the following theorem.
\begin{theorem}[Wasserstein consistency]\label{thm:Posterior_consistency_wasserstein}
Let $\pi\in\Pc$, and let $\Pi$ be a prior on $\Pc$ such that a) $\pi$ is in the $\KL$ support of the prior, and b) the posterior sequence is quadratic-tail tight. Then the posterior is consistent at $\pi$, i.e. for all $\varepsilon>0$,
\begin{equation}\label{eq:consistency_in_prob}
    \Pi_n\Big(\l\{\mu\in\cP_2(\mathbb{R}^d) : W_2(\mu,\pi)>\varepsilon\r\}\Big)\to 0
    \quad \mathbb{P}\text{-a.s.}
\end{equation}
Moreover, if the uniform integrability from Theorem \ref{thm:Vitali_consistency_lifted} is satisfied then the posterior is consistent at the $\WWL$ level, i.e. $\dw (\Postn, \delta_\pi) \to 0$.
\end{theorem} 
The full proof is given in Appendix \ref{App:thm:Posterior_consistency_wasserstein}. As discussed, the Schwartz conditions guarantee the weak consistency of the posterior sequence, while the quadratic-tail tightness helps guarantee the convergence of the moments, thus providing both conditions required for the Wasserstein consistency.

If we want to encode uncertainty about the true non-parametric model, we need to work at the further superposed level of $\Pcpc$ valued random measures i.e., prior laws in $\cH\in\cP_2\l(\Pcpc\r)$. Data observations most commonly arrive as vector-valued in $\R^d$, rather than measures over $\R^d$. If we are to condition our posterior on vector-valued observations at the hierarchical level of non-parametric uncertainty, we need to provide a modification to Schwartz's theorem.

Suppose that data are hierarchically sampled as considered in Section \ref{subsec:4:hierarchical}, with a latent, unobserved sample of probability measures which we then use to sample grouped datapoints in $\R^d$.

For a non-parametric model $\P\in\Pcpc$, let $m_\P^{(j)}$ be the grouped density under $\P$ as below in \eqref{eq:Hierarchical_H}, so that the groups $Z_j = \l(X_{j,1},\dots,X_{j,n_j}\r)$ are independent with $Z_j\sim m_\P^{(j)}$. Let $\cH$ be a hierarchical prior measure over $\Pcpc$ concentrated on $\cM_{\mathrm{ac}}\deq \l\{\P\in\Pcpc:\P\l(\cP_{\mathrm{ac}}(\cK)\r)=1\r\}$ so that likelihoods are well defined; indeed, $m_{\P}^{(j)}\ll \mathrm{Leb}^{\otimes n_j}$. The hierarchical posterior law $\cH(\cdot)\in\cP_2\l(\Pcpc\r)$ is then given by
\begin{equation}\label{eq:Hierarchical_H}
    \cH_J\l(A\mid Z_{1:J}\r) = \dfrac{\int_A \prod_{j=1}^J \frac{dm_{\P}^{(j)}}{dm_{\P^*}^{(j)}}(Z_j)\,\cH(\d \P)}{\int \prod_{j=1}^J \frac{dm_{\P}^{(j)}}{dm_{\P^*}^{(j)}}(Z_j)\,\cH(\d \P)},\qquad m_\P^{(j)} \deq \int \mu^{\otimes n_j} \,\P(\d \mu)\quad \text{on $\cK^{n_j}$}.
\end{equation}
which is well defined $\P^*$-a.s., given that the densities $dm_{\P^*}^{(j)}>0$ a.s.

We now provide a modified, stronger condition on the KL hypothesis to adapt the Schwartz consistency in the hierarchical case. We cannot immediately apply Schwartz's theorem as stated, as it assumes that the data process $X_{j,i}$ is i.i.d., whereas in our case we have independence between each of the groups in $j$, but not identical distributions as the group samples are, by construction, of different sizes $n_j$. We provide a log-variance condition on the likelihoods to access Kolmogorov's strong law of large numbers, in place of Khintchine's, to conclude the almost sure convergence step.

\begin{definition}[Random measure law, KL condition]\label{def:stronger_KL_conditions}
    Let $g=\frac{\d \P^*}{\d \P}$ be the Radon-Nikodym derivative of the truth with respect to the approximator, so that $\KL(\P^*\|\P)= \int \log g \d \P^*$. Moreover let
    \begin{equation*}
        V(\P^*\|\P)\deq \int \l(\log g\r)^2 \d \P^*.
    \end{equation*}
    The condition KL(2) is given by the following assumption: For any $\varepsilon>0$, there is $M>0$ such that the set
    \begin{equation*}
        B_{\varepsilon,M}\deq \Big\{\P:\,\P^* \ll \P, \, \KL(\P^*\|\P)<\varepsilon,\, V(\P^*\|\P) < M \Big\}
    \end{equation*}
    has positive mass under the hierarchical prior, $\cH(B_{\varepsilon,M})>0$.
\end{definition}

\begin{theorem}[Hierarchical Schwartz consistency]\label{thm:Hierarchical_schwartz}
    Let $X_{j,i}$ be sampled according to the hierarchical scheme above, and assume the KL conditions of Definition \ref{def:stronger_KL_conditions}. Suppose the hierarchical prior $\cH$ is concentrated on the space $\cM_{\mathrm{ac}}\deq \l\{\P:\P\l(\cP_{\mathrm{ac}}(\cK)\r) = 1\r\}$ so that likelihoods are well defined. Then the hierarchical posterior $\cH_J(\cdot|Z_{1:J})$ with grouped observations $Z_j = \l(X_{j,1},\dots,X_{j,n_j}\r)$ is consistent at $\P^*$.
\end{theorem}

The proof in Appendix \ref{App:thm:Hierarchical_schwartz} follows the standard Schwartz's theorem argument, given the existence of the testing conditions in Proposition \ref{prop:hierarchical_tests}, where weak neighbourhoods over $\Pkpk$ are constructed using the empirical measure of the grouped observations which approximate the latent measures.

\subsection{Non-parametric posterior contraction rates}\label{subsec:5:pcr} In general, the non-parametric consistency can be arbitrarily slow. We would like to determine sufficient conditions for which the convergence $\dw (\Postn, \delta_\pi)\to0$ admits an explicit rate. We achieve this by adapting the strategy of Fournier-Guillin's Theorem \ref{thm:FG} to the Bayesian non-parametric setting using our framework.

\begin{definition}[$\kappa$-Concentrated Law]\label{def:kappa_concentration}
    Let $\Q \in \cP(\cP(\R^d))$ with intensity measure $m_Q(A)=\int\rho(A)\,\Q(d\rho)$. For $\kappa>0$, we say $\Q$ is {$\kappa$-concentrated} if for every Borel $A\subset\R^d$,
\begin{equation}\label{eq:kappa_concentration}
    \mathrm{Var}_{\rho\sim \Q}\big(\rho(A)\big)\leq\frac{m_Q(A)}{\kappa}.  
\end{equation}

\end{definition}

Two common examples are given below.
\subsubsection*{Empirical measure}
For $X_i\overset{\text{iid}}{\sim} \mu \in \cP(\R^d)$ with $\hmun = \frac1n\sum_{i=1}^n \delta_{X_i}$ and $\Q = \mathrm{Law}(\hmun)\in\cP(\cP(\R^d))$, $n\hmun(A)\sim\mathrm{Binomial}\big(n,\mu(A)\big)$ for which $\mathrm{Var}\big(\hmun(A)\big) \leq \frac{m_{\Q}(A)}{n}$. Hence $\Q$ is $\kappa$-concentrated with $\kappa=n$. This is the central structural requirement on the random measure in the proof of Fournier-Guillin, which we generalise to accommodate more general random measures.

\subsubsection*{Dirichlet process}
Consider a Dirichlet process $\Q = \mathrm{DP}(\alpha, G_0)$ with concentration parameter $\alpha > 0$ and base probability measure $G_0 \in \cP(\R^d)$. This is $\kappa$-concentrated with $\kappa = \alpha + 1$.

\begin{theorem}[Generalised Fournier-Guillin rate]\label{thm:Generalied_FG}
    Let $\Q\in\cP(\cP(\R^d))$ be $\kappa$-concentrated with intensity $m\deq m_Q$, and suppose $\cM_q(m)<\infty$ for some $q>p$. Then there is $C=C(p,d,q)$ such that
\begin{equation*}
\int W_p^p(\rho,m)\,\Q(d\rho)\leq C\,\cM_q(m)^{p/q}\,r_p(\kappa, d, q),
\end{equation*}
where $r_p(\kappa,d,q)$ is the Fournier-Guillin rate \eqref{eq:FG_rate}.
\end{theorem}

The proof follows exactly from Fournier-Guillin's Theorem 1 proof from \cite{fournier2015rate}, after their equation (4). The authors state:
    \textit{Since $N\mu_N(A)$ is $\mathrm{Binomial}(N,\mu(A))$-distributed, we have $$\E\l[\l|\mu_N(A)-\mu(A)\r|\r]\leq\min\l\{2\mu(A),\sqrt{\mu(A)/N}\r\}.$$}
From this point, the only distinguishing property of the empirical measure used in their proof is the variance bound, which is provided for more general random measure laws in \eqref{eq:kappa_concentration}. We achieve the same RHS bound by replacing the empirical measure with any $\kappa$-concentrated law. We state Theorem \ref{thm:Generalied_FG} for general $p\geq1$, but consider the case $p=2$ from here in line with the rest of this work.

For the empirical measure $\hmun$, the stochasticity is introduced in the data draws $X_i(\om)$, while for the Dirichlet process, it is introduced in the sample draw  $\rho \sim \mathrm{DP}(\alpha, G_0)$. In our non-parametric framework, the posterior $\Postn$ takes into account \textit{both} sources using a two-stage draw process, first the data stage $\om\sim\PP$, then the posterior stage $\rho\sim\Pi_{\om,n}$. 

For each $n$, we define the mixture law $\Q_n\deq\E_\om[\Pi_{\om,n}]$, which is the law of $\rho$ under the two-stage procedure. Indeed, $\P(\rho\in A)=\E_\om[\P(\rho\in A\mid\om)]=\E_\om[\Pi_{\om,n}(A)]$. We also define the posterior intensity $\alpha_{\om,n}\deq\int\rho\,\Pi_{\om,n}(d\rho)$. Under this notation, the intensity measure of $\Q_n$ is $\bar\alpha_n=\int\rho\,\Q_n(d\rho)=\E_\om[\alpha_{\om,n}]$. 

\begin{definition}[Vanishing fluctuation]\label{def:fluctuation_control}
A sequence $(\Pi_n)_{n \geq 1} \subset \WWL$ has $\kappa_n$-{vanishing fluctuation} if there is a deterministic sequence $\kappa_n \uparrow \infty$ such that the mixture law $\Q_n$ is $\kappa_n$-concentrated for every $n$. More explicitly, for every Borel $A$,
$$
\mathrm{Var}_{\Q_n}\big(\rho(A)\big)=\E_\om\,\E_{\rho\sim\Pi_{\om,n}}\Big[\big(\rho(A)-\bar\alpha_n(A)\big)^2\Big]\leq\frac{\bar\alpha_n(A)}{\kappa_n}.
$$
\end{definition} 

The following proposition provides sufficient conditions to verify this condition. 

\begin{proposition}[Total variance decomposition]\label{prop:total_variance}
For every Borel $A$,
$$
\mathrm{Var}_{\Q_n}\big(\rho(A)\big)=\E_\om\big[\mathrm{Var}_{\Pi_{\om,n}}(\rho(A))\big]+\mathrm{Var}_\om\big(\alpha_{\om,n}(A)\big).
$$
Suppose there are deterministic sequences $\kappa_n',\kappa_n''$ such that for every Borel $A$:
\begin{enumerate}[label=(\roman*)]
    \item $\mathrm{Var}_{\Pi_{\om,n}}(\rho(A))\leq\alpha_{\om,n}(A)/\kappa_n'$ $\PP$-a.s.
    \item $\mathrm{Var}_\om(\alpha_{\om,n}(A))\leq\bar\alpha_n(A)/\kappa_n''$.
\end{enumerate}
Then $(\Pi_n)$ has vanishing fluctuation with $\kappa_n=\big(1/\kappa_n'+1/\kappa_n''\big)^{-1}$, provided $\kappa_n$ increases to $\infty$.
\end{proposition}
The decomposition follows from the law of total variance. Taking $\E_\om$ in (i) and using $\E_\om[\alpha_{\om,n}]=\bar\alpha_n$ bounds the first term by $\bar\alpha_n(A)/\kappa_n'$, whence adding (ii) delivers the full claim. 
\subsubsection*{Example: Dirichlet process posterior}
Assume $X_i\overset{\text{iid}}{\sim} \pi \in \cP_2(\R^d)$ with empirical measure $\hmun$. Consider the Dirichlet process prior  $\mathrm{DP}(a,\alpha_0)$ with concentration parameter $a > 0$ and base prior measure $\alpha_0 \in \cP_2(\R^d)$. The Dirichlet process posterior $\Pi_n\in\WWL$ has vanishing fluctuation with $\kappa_n\geq \frac{a+n}{2}\uparrow \infty$, with $\kappa_n'=a+n+1$, $\kappa_n''=a+n$. A calculation for each term is given in Appendix \ref{App:Dirichlet_example}.\newline 

To derive a contraction rate for $\dw (\Postn, \delta_\pi)\to0$, we split the error into a fluctuation term (how far $\Pi_n$ spreads around its own intensity $\bar\alpha_n$) and a bias-type term (how far $\bar\alpha_n$ is from $\pi$). Fluctuation is controlled by vanishing fluctuation and bias is controlled by assuming some structure on $\alpha_{\om,n}$. Here, we assume that the posterior intensity  is a linear combination of the base prior measure and the empirical measure, as in the Dirichlet process case.

\begin{theorem}[Posterior contraction rates]\label{thm:pcr}
    Let $\pi\in\cP_q(\R^d)$ with $q>2$ in line with the conditions of Theorem \ref{thm:FG}  with $X_i\overset{i.i.d.}{\sim}\pi$. Assume  $(\Pi_n)_{n \geq 1} \subset \WWL$ has vanishing fluctuation, and that the posterior intensity is a linear combination of the base prior measure and empirical measure $\alpha_n=\lambda_n\,\alpha_0+(1-\lambda_n)\,\hmun$, with $\alpha_0\in\cP_q(\R^d)$ fixed and $\lambda_n\in[0,1]$ deterministic with $\lambda_n\to0$. Then there exists some $C=C(d,q,\cM_q(\pi),\cM_q(\alpha_0))$ independent of $n$ such that 
    $$\dw^2 (\Pi_n,\delta_\pi)\leq C\Big(\lambda_n+r(\kappa_n, d, q)\Big).$$
\end{theorem}
The proof in Appendix \ref{App:thm:pcr} splits the contraction terms into the fluctuation and bias-type terms as mentioned. The theorem shows that the posterior contraction rate is the maximum of $\lambda_n$, the rate at which the prior is forgotten, and $r(\kappa_n, d, q)$, the rate of vanishing fluctuation.

\subsubsection*{Example: Dirichlet process posterior}
We have verified above that the Dirichlet process posterior has vanishing fluctuation with $\kappa_n$ of order $O(n)$. The posterior intensity is given by $\alpha_n = \frac{a\alpha_0 + n\hmun}{a+n}$, which is the exact linear combination structure with $\lambda_n = \frac{a}{a+n} \to 0$. $\lambda$ is order $O(n^{-1})$ which is strictly lower than $r(n, d, q)$ for all dimensions $d$, meaning the Dirichlet process posterior contracts at the Fournier-Guillin $r(n, d, q)$ rate. We note that this is an improvement to the rate provided by Camerlenghi et al. \cite{Camerlenghietal2026}. One key distinction, which explains the improvement, is that Camerlenghi et al. partition the sample space at a single scale, whereas Fournier-Guillin use nested dyadic partitions which handle mass mismatch more efficiently and yield improved rate.

\begin{remark}
    Some non-parametric models have vanishing fluctuation, but a non-vanishing rate for the prior to be forgotten $\lambda_n\to\lambda_\infty>0$ (for example the Pitman-Yor process), hence the bias term does not collapse to 0. Using joint convexity, one shows that $W_2^2(\bar\alpha_n,\pi) \leq \dw^2(\Pi_n,\delta_\pi)$, hence such models do not access consistency through this setup.
\end{remark}

\subsection{Mixture models}\label{subsec:5:mixture}
The posteriors considered in Section \ref{subsec:5:pcr} are placed directly on the data-generating law. The Dirichlet process posterior, for instance, is supported on discrete measures, 
which is an unnatural description for continuous data. The standard modification is to place the non-parametric prior over a latent mixing distribution, which is then smoothed through a kernel. We show that the framework of the previous section extends to this setting. Our results differ in target and assumptions from those of Nguyen \cite{nguyen2013}, Ghosal and van der Vaart \cite{Ghosal_van_der_Vaart_2017} and Guha et al. \cite{guha2021posterior}. We use a different set of sufficient assumptions to bound the posterior in the data space directly, which allows for a more general and interpretable formulation.

Let $\Theta\subset\R^{d_\Theta}$ denote a latent space, $\Pi\in\mathcal{P}(\mathcal{P}(\Theta))$ be a prior over that latent space, and let $f(\cdot\mid\theta)$ be a Markov kernel from $\Theta$ to $\R^d$. The mixture model is described by the hierarchical scheme,
$$
\rho\sim\Pi,\qquad \theta_1,\dots,\theta_n\mid 
\rho\overset{\text{iid}}{\sim}\rho,\qquad X_i\mid\theta_i\overset{\text{indep.}}{\sim}f(\cdot\mid\theta_i).
$$

The data-generating distribution $\pi \in \cP(\R^d)$ and latent distributions $\rho\in\cP(\Theta)$ act over different spaces. We define the linear transfer map $\Phi:\cP(\Theta)\to\cP(\R^d)$ by $\Phi(\rho)=\int f(\cdot\mid\theta)\,\rho(\d \theta)$. We note that
$f(\cdot\mid\theta)=\delta_\theta$ gives $\Phi=\id$, which recovers the complete theory of Section \ref{subsec:5:pcr}. The mixture model provides a generalisation.

In the mixture setting, the data posterior $\mathrm{Law}(\rho\mid X_{1:n})$ is intractable in general - conditioning on the latent variables $\theta_{1:n}$ restores conjugacy. We enlarge the probability space letting $\Omega\deq\Theta^n\times(\R^d)^n$. In this setting, the data stage $\omega\sim\PP$ is itself two-step: First, $X_{1:n}\sim\pi^{\otimes n}$. Then, conditionally on $X_{1:n}=x$, the latents $\theta_{1:n}$ are drawn from their posterior law under the Bayesian model 
$\mathrm{Law}_{\mathrm{Bayes}}\big(\theta_{1:n}\mid X_{1:n}=x\big)$. This defines the law $\PP$ on $\Omega$. Using $\om=(\theta_{1:n},X_{1:n})$, we define conditional posteriors on both the latent and observation spaces,
$$
\begin{aligned}
\Pi^{\Theta}_{n}&\in L^2_{\mathrm{WoW}}(\Theta), && \Pi^{\Theta}_{\om,n}\deq\mathrm{Law}\big(\rho\mid\theta_{1:n},X_{1:n}(\om)\big)\in\cP(\cP(\Theta)),\\
\Postn&\in \WWL, && \Pi_{\om,n}\deq\Phi_\#\Pi^{\Theta}_{\om,n}\in\cP(\cP(\R^d)),
\end{aligned}
$$
together with the intensity of the conditional posterior, $\alpha^{\Theta}_{\om,n}\deq\int\rho\,\Pi^{\Theta}_{\om,n}(d\rho)$. Since $\rho$ is conditionally independent of $X_{1:n}$ given $\theta_{1:n}$, the conditional posterior $\Pi^{\Theta}_{\om,n}$ is now tractable. The proposition below shows that viewing $\Postn$ as a conditional posterior instead of a data posterior does not change the underlying $\dw$ distance under the enlargement. This follows from factorising the joint posterior and applying the tower property, as shown in Appendix \ref{App:prop:enlargement}.

\begin{proposition}[Enlargement preserves distance]\label{prop:enlargement}
Let $\Postn^{x}\coloneqq\Phi_\#\,\mathrm{Law}(\rho\mid X_{1:n}=x)$ be the data posterior, viewed as an element of $\WWL$. Then
$
\E_x\l[W_{W_2}^2\l(\Pi_n^{x},\delta_\pi\r)\r]=\dw^2(\Pi_n,\delta_\pi) .
$
\end{proposition}

We wish to apply tools from Section \ref{subsec:5:pcr} with $\theta_{1:n}$ in place of $X_{1:n}$, but note that $\theta_{1:n}$ are not i.i.d. under $\PP$. We provide a modification of Definition \ref{def:fluctuation_control}, assuming only condition (i) of Proposition \ref{prop:total_variance}, which is verifiable for $\theta_{1:n}$ non-i.i.d.

\begin{definition}[Path-wise vanishing fluctuation]\label{def: pathwise_vanishing_fluctuation}
    A sequence $(\Pi^\Theta_n)_{n \geq 1} \subset L^2_{\mathrm{WoW}}(\Theta)$ has path-wise $\kappa_n$-{vanishing fluctuation} if there is a deterministic sequence $\kappa_n \uparrow \infty$ such that $\Pi^\Theta_{\omega,n}$ is $\kappa_n$-concentrated $\PP$-a.s. for every n. More explicitly, for every Borel set $A \subset \Theta$, 
    $$\mathrm{Var}_{\rho \sim \Pi^\Theta_{\om,n}}\big(\rho(A)\big) = \int_{\cP(\Theta)} \big(\rho(A) - \alpha^\Theta_{\om,n}(A)\big)^2\,\Pi^\Theta_{\om,n}(d\rho) \leq \frac{\alpha^\Theta_{\om,n}(A)}{\kappa_n}.$$
\end{definition}

Finally, results from Section \ref{subsec:5:pcr} apply to posterior on the observation space, whereas the assumption holds for posterior on the (unobserved) latent space. For this we provide the following lemma to transfer statements from latent space to observation space.

\begin{lemma}[$W_2$-Lipschitz kernels]\label{lemma:wasserstein_lipschitz}
    Assume that the kernel $f$ is $W_2$-Lipschitz, meaning for some $L>0$, $W_2\l(f(\cdot\mid\theta),f(\cdot\mid\theta')\r)\leq L\|\theta-\theta'\|$ for all $\theta,\theta'$. 
Then, for any probability measures $\nu, \nu' \in \mathcal{P}_2(\Theta)$, $W_2\l(\Phi(\nu), \Phi(\nu')\r) \leq L \, W_2(\nu, \nu').$
\end{lemma}

The proof in Appendix \ref{App:lemma:wasserstein_lipschitz} follows a standard sub-optimal coupling argument, with $(\nu,\nu')$, and $\l(f(\cdot\mid\theta),f(\cdot\mid\theta')\r)$ optimally coupled, providing a sub-optimal coupling on $\l(\Phi(\nu),\Phi(\nu')\r)$. We now provide an explicit posterior contraction formulation for mixture models. We denote the latent empirical measure $\nu_n \deq \frac1n\sum_{i=1}^n\delta_{\theta_i}$, with intensity $\beta_n \deq \E_\om[\nu_n]\in\cP(\Theta)$. 

\begin{theorem}[Posterior contraction rates for mixtures]\label{thm:pcr_mixture}
Let $\pi\in\cP_2(\R^d)$ with $X_i\overset{i.i.d.}{\sim}\pi$. Assume $(\Pi_n^\Theta)_{n \geq 1} \subset L^2_{\mathrm{WoW}}(\Theta)$ has path-wise vanishing fluctuation and the kernel $f$ is {$W_2$-Lipschitz}. Assume $\sup_n\cM_q(\beta_n)<\infty$ for some $q>2$ in line with the conditions of Theorem \ref{thm:FG}. Assume the intensity of the conditional posterior on the latent space is given by the linear combination
$$\alpha^{\Theta}_{\om,n}=\lambda_n\,\alpha_0+(1-\lambda_n)\,\nu_n,\qquad \alpha_0\in\cP_q(\Theta) \text{ fixed},\qquad \lambda_n\in[0,1] \text{ deterministic with }\lambda_n\to0.$$

Then there is $C=C\big(d_\Theta,q,L, \sup_n\cM_q(\beta_n), \cM_q(\alpha_0), W_2(\Phi(\alpha_0),\pi)\big)$ independent of $n$ so that 
$$
\dw^2 ( \Pi_n,\delta_\pi)\leq C\l(\lambda_n+r(\kappa_n,d_\Theta,q)\r)+2 \E_\om [W_2^2\l(\Phi(\nu_n),\pi\r)].
$$
\end{theorem}

The proof in Appendix \ref{App:thm:pcr_mixture} follows a similar structure to the proof of Theorem \ref{thm:pcr}, with the key change that we no longer use the mixture law $\Q$, but instead appeal to Lemma \ref{lemma:wasserstein_lipschitz} before the application of Theorem \ref{thm:Generalied_FG}.

\begin{remark}
    The extra term $\E_\om [W_2^2\l(\Phi(\nu_n),\pi\r)]$ is a mixture bias, measuring how well the smoothed latent empirical measure $\Phi(\nu_n)$ approximates $\pi$ in expectation. Its behaviour depends on whether the model is well-specified or misspecified, and on the bandwidth of the kernel $f$, which may be fixed or shrinking with $n$.
\end{remark}

\section{Applications}\label{sec:6}
The optimal transport toolkit has found ubiquitous application in the field of generative modelling, providing a means of analysing measure-valued maps and data representations. The statistical framework we have been able to provide for the classical theory in Sections \ref{sec:4} and \ref{sec:5} enables statistical consistency results for those applications of optimal transport. We use this section to provide 3 applications focusing on different aspects, covering: consistency of Bonet et al.'s \cite{bonet2025flowing} WoW representation of labelled datasets in Section \ref{subsec:6:flowing_datasets}, an embedding of Chemseddine et al.'s \cite{chem} COT flow matching objective into our framework in Section \ref{subsec:6:COT}, and convergence for random self-attention flows as in \cite{rigollet2026meanfielddynamicstransformers} in Section \ref{subsec:6:TSA}.

\subsection{Flowing datasets}\label{subsec:6:flowing_datasets}

In their paper \cite{bonet2025flowing}, Bonet et al. present a representation of class hierarchical models for labelled datasets, demonstrating applications in domain adaptation, transfer learning and domain distillation. The choice of hierarchical representation is to represent each class by a conditional distribution of features, and the dataset itself as a mixture distribution supported on the classes, thus embedding the representation into the Wasserstein over Wasserstein space. The authors make use of the differential structure to design functional optimisation via WoW gradient flows for the given applications.

More explicitly, the authors \cite{bonet2025flowing} provide the representation of the labelled dataset with the classes as distributions $$\Q = \sum_{c=1}^C \pi_c\,\mu^{c,n}\,,\qquad {\mu^{c,n}} = \frac{1}{n}\sum_{i=1}^n \delta_{x_i^c},$$
where each $\mu^{c,n}$ is the empirical measure of the labelled samples $x_i^c$ from each class $c$, and the weights $\pi_c$ are uniform.
We could instead consider the case where the distribution of the classes (i.e. the weights $\pi_c$) is unknown i.e. there is uncertainty over what the classes are, but we assume that we still implicitly sample a group of $\{x_i^c\}_{i=1}^{n_c}\sim \mu_c$ where $\mu_c$ is the population measure of the class via the hierarchical representation.

This invites the theory of the hierarchical random measures in $\WWL$ from Sections \ref{subsec:4:hierarchical} and \ref{subsec:5:consistency}. Estimating the distribution over the possible class-conditionals $\mu_c$ is an estimate of the random measure law in $\Pcpc$. The given schematic places the representation directly into our latent sampling setup from the sections above. While Bonet et al. directly represent the labelled datasets via the empirical measures, given our setup, we treat the data coming in as a random process, sampled from  some underlying measure $\Q$,
\begin{align*}
    \mu_1,\dots,\mu_C \overset{i.i.d.}{\sim} \Q\qquad &\text{Unobserved, latent}\\
    x_{1}^c,\dots, x_{n_c}^c\overset{i.i.d.}{\sim}\mu_c\qquad &\text{Observed},
\end{align*}
under the notation of \cite{bonet2025flowing}, thus the hierarchical theory developed in Sections \ref{subsec:4:hierarchical} and \ref{subsec:5:consistency} immediately applies under this representation of the data, where we are uncertain of the distribution $\Q$. Bonet et al. \cite{bonet2025flowing} make the choice to represent the data distribution in the given, fixed, empirical form. In contrast, we treat their representation as a randomly sampled approximation to the underlying data process, and consider the quality of this representation in large sample.

\begin{proposition}[Consistency of the WoW representation]\label{prop:Flowing_empirical}
    Let $\Q\in\Pcpc$ be the true distribution over the classes. Then using the hierarchical sampling scheme, the representation from Bonet et al., which we denote $\hat \Q$, converges in $\WWL$ to the truth $\Q$. Moreover, if the data are assumed to live on a compact space, by Corollary \ref{cor:compact_wow_upgrade} the consistency is almost sure.
\end{proposition}
This is an immediate consequence of Theorem \ref{thm:Hierachical_WoW_LLN}, given that the setup matches. In addition, the hierarchical Bayesian consistency from Theorem \ref{thm:Hierarchical_schwartz} applies, finding consistency of the Bayesian posterior when we restrict to a compact base space $\cK\subset\R^d$. Suppose we have a prior over the class distribution $\Q\in\Pcpc$, denoted $\cH$, i.e. a hyperprior in the language of Theorem \ref{thm:Hierarchical_schwartz}. Then the posterior concentrates around the truth $\Q$ in large sample.

\begin{proposition}[Posterior consistency of the representation]\label{prop:flowing_bayes}
    Let $\Q\in\Pkpk$ be the true distribution over the classes. Suppose we have a hyperprior $\cH$ over the distribution of the classes satisfying the hypotheses of Theorem \ref{thm:Hierarchical_schwartz}. Then using the hierarchical sampling scheme, the hierarchical posterior $\cH(\cdot|X_{1:n})$ concentrates on arbitrary WoW neighbourhoods of the true distribution $\Q$ i.e. the true hierarchical representation of the class labels.
\end{proposition}
In summary, while \cite{bonet2025flowing} chooses a representation of a labelled dataset which assumes the class distribution is known a priori, treating their representation as an estimate of some true hierarchical sampling setup by accounting for uncertainty in that class distribution, we demonstrate that the quality of this representative estimate increases in large sample. We provide example experiments in Appendix \ref{App:Flowing_datasets_experiments} for the synthetic 3 ring dataset \cite{bonet2025flowing}, and a modified version which admits a density.

\subsection{Conditional optimal transport flow matching}\label{subsec:6:COT}

A natural application of our random optimal transport framework is provided by Conditional Optimal Transport (COT), explored by Chemseddine et al. in \cite{chem} where the authors consider the transportation between two conditional probability measures, given some context variable $x\in\cX$ for some sample space $\cX$. 

In \cite{chem}, the COT problem is set up in the following way: \\
Let $Y \in \cY$ be a source variable, $Z \in \cZ$ a target variable, and respective source and target conditional measures $\mu_{Y|X=x}$ and $\nu_{Z|X=x}$. For any fixed context $X=x$, the conditional transport plan $\pi_x(y, z)$ must have marginals $\mu_{Y|x}$ and $\nu_{Z|x}$. The conditional Kantorovich problem for a given fixed $x$ is:
$$W_c(\mu_{Y|x}, \nu_{Z|x}) = \inf_{\pi_x \in \Gamma_{\mu_{Y|x}, \nu_{Z|x}}} \int_{\cY \times \cZ} c(x, y, z)\, \pi_x(\d y, dz).$$
Note that the cost function $c(x, y, z)$ can also depend on $x$.

In practice, one does not solve with a single fixed $x$, but rather assumes $X$ follows some marginal distribution $P_X$. In this case, we find a collection of conditional transport plans $\{\pi_x\}_{x \in \cX}$ that minimises the expected conditional Wasserstein distance over all possible contexts:
$$\E_{X \sim P_X}\l[W_c(\mu_{Y|X}, \nu_{Z|X})\r] = \int_{\cX} \l( \inf_{\pi_x \in \Gamma_{\mu_{Y|x}, \nu_{Z|x}}} \int_{\cY \times \cZ} c(x, y, z) \, \pi_x(\d y, dz) \r) \,P_X(\d x).$$
To reformulate this into our framework, we note that the context variable $X$ induces a probability space, where $\Omega = \cX$, we take the Borel $\sigma$-algebra over $\cX$ and $\PP = P_X$. The conditional source and target distributions $\mu_{Y|X}$ and $\nu_{Z|X}$ are therefore random probability measures determined by the realisation of $X$. We define $\xi_{\text{source}}(\om) \deq \mu_{Y|X_\om}$ and $\xi_{\text{target}}(\om) \deq \nu_{Z|X_\om}$; both are elements of $\WL$, and the distance between them is:
$$d^2(\xi_{\text{source}}, \xi_{\text{target}}) = \E_{X \sim P_X} \l[ W_2^2(\mu_{Y|X}, \nu_{Z|X}) \r],$$
exactly the global objective function of COT under quadratic cost ($c(x,y,z) = \|y-z\|^2$). Within this formulation, conditional probability measures inherit all of the theory discussed in Section \ref{sec:3}, including the dynamic theory of geodesics and gradient flow paths. Moreover, under an appropriate statistical framework, the conditional probability measures also enjoy the statistical convergence results of Sections \ref{sec:4} and \ref{sec:5}.\newline

We apply this formulation to the flow matching problem defined in \cite{chem}, showing its equivalence with an exact orthogonal projection onto the true Riemannian tangent bundle of conditional probability distributions.
\begin{theorem}[Projection of COT flow matching onto $\WL$]\label{thm:cot_projection_wl}
    Let $(\cX, \cA, P_X)$ be a probability space corresponding to a conditioning variable $X$, and let ${M}_t : x \mapsto \mu_{t,x}$ be the flow of random conditional probability measures for times $t \in [0,1]$, such that $\xi_t \in \WL$ is a random probability measure for each fixed $t$. 
    
    Let $L^2(\xi_t)$ denote the Hilbert space of random vector fields equipped with the Wasserstein Dirichlet form inner product from \eqref{eq:lifted_ke_norm},
    $$\langle V, W \rangle_{\xi_t} = \E_{x \sim P_X} \l[ \int_{\R^d} \big\langle V(x, y), W(x, y) \big\rangle_{\R^d} \, \mu_{t,x}(\d y) \r].$$
    Assume $V^*(x, \cdot)$ is the target COT vector field (the displacement interpolation geodesic) and that the parametrised vector field is $V_\theta(x, \cdot) \in L^2(\xi_t)$. Then, $V_\theta$ admits a unique orthogonal decomposition into an irrotational component $V_\theta^\parallel$ and a divergence-free component $V_\theta^\perp$.
    
    Moreover, the CFM objective function $\cL^{\text{CFM}}_t(\theta)$ admits an orthogonal decomposition into two independent, non-negative loss components,
    $$\cL^{\text{CFM}}_t(\theta) = \E_{x \sim P_X, \, y_t \sim \mu_{t,x}} \l[ \big\| V_\theta(x, y_t) - V^*(x, y_t) \big\|^2 \r] = \l\| V_\theta^\parallel - V^* \r\|_{\xi_t}^2 + \l\| V_\theta^\perp \r\|_{\xi_t}^2.$$
\end{theorem}
The proof in Appendix \ref{App:thm:cot_projection_wl} makes use of Hilbert's projection theorem to determine a unique orthogonal decomposition. This result shows that COT flow matching provides an orthogonal projection onto $\cT_{\xi_t} \WL$. This also shows flow matching under this framework can be considered optimal in the sense that there is no divergence-free component of the learnt vector field.

Given a Lipschitz condition on the vector fields considered, we can provide a Grönwall stability result, an application of the idea behind Proposition \ref{prop:WGF_lifted_stability} - the calculation is shown in Appendix \ref{App:prop:GS_vector_cot}. This provides a correspondence between the convergence of the random flow in $\cT\WL$ and the random probability measures in $\WL$.

\begin{proposition}[Stability under vector field perturbation]\label{prop:GS_vector_cot}
Let $(\mu_t)_{t\in[0,1]}$ solve the continuity equation \eqref{eq:continuity} with field $v^*_t$, and let $v_t$ be a second, time-indexed vector field, $L$-Lipschitz in space uniformly in $t\in[0,1]$. Let $X_0\sim\mu_0$ and define the coupled flows 
$\dot X_t = v_t^*(X_t), \dot Y_t = v_t(Y_t),$ where $X_0=Y_0.$
Let $\hmu_t\deq\mathrm{Law}(Y_t)$. Then for every $t\in[0,1]$,
\begin{equation}\label{eq:gronwall_stability}
W_2^2(\mu_t,\hmu_t) \,\leq e^{(1+2L)t}\int_0^t \|v_s - v_s^*\|^2_{\mu_s}\,\d s.
\end{equation}
In particular, convergence of an approximate random vector field to the true random flow in Dirichlet norm \eqref{eq:lifted_ke_norm} delivers the convergence of the random probability measures in $\WL$.
\end{proposition}

\subsection{Self-attention flows}\label{subsec:6:TSA}

The recent study of transformers, and self-attention flows, as in \cite{rigollet2026meanfielddynamicstransformers}, \cite{glpr}, \cite{bruno2026a}, \cite{agazzi2026} among many other works, has exploded in interest. The reader is naturally directed to \cite{Vaswani} for a comprehensive review of transformer architecture. The works study self-attention as a dynamical system, defining a self-attention flow as a mean-field token interaction in continuous time.

Let $x_i(t)\in \bS^{d-1}$ denote the position of the $i$-th token at time $t$, and let $\beta>0$ be an inverse-temperature parameter.
The self-attention (SA) dynamics are given by
\begin{equation}\label{eq:SA}
    \dot x_i(t) 
    = \proj_{x_i(t)}\l(
        \frac{1}{Z_{\beta,i}(t)} 
        \sum_{j=1}^n e^{\beta \langle x_i(t),\, x_j(t)\rangle}\, x_j(t)
      \r),
      \qquad 
      Z_{\beta,i}(t) = \sum_{k=1}^n e^{\beta\langle x_i(t),\, x_k(t)\rangle},
\end{equation}
where $\proj_x y = y - \langle x,y\rangle x$ denotes the orthogonal projection onto $\bS^{d-1}$, enforcing a layer normalisation effect, restricting all tokens to the unit sphere.

As Rigollet remarks in \cite{rigollet2026meanfielddynamicstransformers}, to leverage the tools of optimal transport, it would be valuable if one could identify the dynamics in \eqref{eq:SA} with a Wasserstein gradient flow; however, this is not possible due to a symmetry issue, as in \cite{sinkformers}. Circumventing the issue, \cite{rigollet2026meanfielddynamicstransformers} provides the surrogate, un-normalised self-attention (USA) model, which loses the denominator $Z_{\beta,i(t)}$. The work demonstrates that the USA model retains many of the features of the standard (SA) model, while enjoying a more tractable analysis. One formulates the USA using the empirical distribution of the tokens $\mu_t = \frac{1}{n}\sum_{i=1}^n\delta_{x_i(t)}$, evolving according to the continuity equation with the family of vector fields given by
\begin{equation}\label{eq:continuity_2}
    v_t(x)=\proj_{(\cdot)}\int e^{\beta\langle x,y\rangle} y\,\mu_t(\d y).
\end{equation}
For the USA, this vector field is the Wasserstein gradient of the interaction energy projected back onto the unit sphere.
\begin{equation}
\label{eq:E}
    \cE^{\beta}(\mu)
    = \frac1{2\beta} \iint e^{\beta\langle x,y\rangle}\,\d  \mu(x)\d\mu(y),
\end{equation}

The fast-developing theory for the mean-field (USA) dynamics applies after fixing a token initialisation, and then letting the token distribution evolve through the (USA) gradient flow. Rather than fixing a token distribution, consider sampling tokens from some distribution over the embedding space - the perspective of the random empirical measure as a random measure from Section \ref{sec:4}. More generally, we can represent the token distribution using a random posterior draw $\Postn$ as in Section \ref{subsec:5:parametrics}. Our  framework then provides the general setting to explore statistical aspects of the (USA) paths induced by \eqref{eq:continuity_2}.

Following Definition \ref{def:random_functional}, we define the \textit{random} interaction energy and corresponding lifted functional$\cE^{\beta}:\Omega\times\wlsac\to\R$,
\begin{equation*}
    \cE^{\beta}(\om,\xi) = \cE^{\beta}(\xi_\om)
    = \frac1{2\beta} \iint e^{\beta\langle x,y\rangle}\,\xi_\om(\d x)\xi_\om(\d y), \qquad \mE^{\beta}(\xi) = \E_\om\l[\cE^{\beta}(\xi_\om)\r].
\end{equation*}
From Theorem \ref{thm:LiftedgradW}, this defines a random Wasserstein gradient flow, inducing USA paths
\begin{equation}\label{eq:lifted_interaction_gf}
    \partial_t \xi_{\om,t} =  \divergence\bigl(\xi_{\om,t} \,\gradWL \mE(\xi_t)(\om)\bigr),\quad \gradWL \mE\big(\xi_{\om,t}\big)(\cdot) = \int e^{\beta\langle \cdot, y\rangle}y\xi_{\om,t}(\d y).
\end{equation}
One projects this gradient flow back onto the tangent space of the unit sphere $\cT\bS^{d-1}$, considering
\begin{equation}\label{eq:lifted_interaction_proj_gf}
  \proj_{(\cdot)} \nabla\delta \cE^{\beta}(\xi_{\om,t})(\cdot) = \proj_{(\cdot)}\l(\int e^{\beta\langle \cdot, y\rangle}y\, \xi_{\om,t}(\d y)\r).
\end{equation}

Viewing the transformer as a token flow map, we quantify the stability of the evolution of the empirical gradient flow compared to the flow of the oracle token distribution. Showing the Lipschitz continuity of the interaction energy over the unit sphere permits the employment of Proposition \ref{prop:WGF_lifted_stability}, the propagation of chaos result, for the USA model, a result provided in Proposition 3.1 of \cite{agazzi2026} which makes use of the empirical measure representation of the token sample.

\begin{proposition}[Self-attention propagation of chaos]\label{prop:discrete_interaction_flow_stability}
    Let $\mu_t$ denote the gradient flow \eqref{eq:lifted_interaction_gf} induced by the vector field \eqref{eq:lifted_interaction_proj_gf} initialised at $\mu_0\in\Pc$. Let $\xi_{\om,t}$ denote the random gradient flow above initialised at either the empirical measure approximation $\hmuno$  under the appropriate assumptions from Theorem \ref{thm:empirical_convergence} or a parametric Bayesian posterior $\Postn$ under the appropriate assumptions from Proposition \ref{prop:Bayesian_propagations} respectively. Then the propagation of chaos from Proposition \ref{prop:WGF_lifted_stability} applies with $L=(\beta + 1)^2e^{2\beta} + 1 + 2(\beta + 2)e^\beta$.
    \begin{equation*}
        a)\, \sup_{t\in[0,T]}  W_2^2\big(\xi_{\om,t}, \mu_t\big) \leq e^{Lt} W_2^2\big(\xi_{\om,0}, \mu_0\big) \quad \text{$\PP$-a.s.},\qquad
        b)\,
        \sup_{t\in[0,T]}  d^2\big(\xi_t, \mu_t\big) \leq e^{Lt} d^2\big(\xi_t, \mu_0\big),
    \end{equation*}
    
\end{proposition}
This follows from Proposition \ref{prop:WGF_empirical}, after verifying the Lipschitz constants, targeting
\begin{align*}
    \big\|\gradW \cF(\mu)(x) - \gradW \cF(\nu)(x) \big\| \leq L_1 W_2(\mu, \nu),\qquad
    \big\|\gradW \cF(\mu)(x) - \gradW \cF(\mu)(y)\big\| \leq L_2 \l\|x - y\r\|,
\end{align*}
where the first is Lipschitz in measure $\Pc$, and the second is Lipschitz in space $\bS^{d-1}$. A full calculation is found in Appendix \ref{App:prop:discrete_interaction_flow_stability} where Proposition \ref{prop:WGF_lifted_stability} delivers $L=L_1^2 + 1 +  2L_2$.

\section{Discussion}
We introduce the space of $L^2$ Wasserstein random variables, providing a unified framework for a statistical analysis of optimal transport. We use the classical Wasserstein space theory of Section \ref{sec:2} to define the Wasserstein random variables in Section \ref{sec:3}, considering the optimal transport problems induced by their realisations, hence discussing a random optimal transport. Crucially, we demonstrate that the inherent geodesic geometry and formal Riemannian structure that $\WL$ induces over the space of random probability measures gives rise to random gradient flows with Wasserstein gradient flow sample paths thus $\WL$ can be considered the canonical setting within which to explore random functional dynamics in Wasserstein space.

The framework developed in Section \ref{sec:3} lays a solid foundation to discuss both a frequentist and Bayesian analysis of the optimal transport problem in Sections \ref{sec:4} and \ref{sec:5} respectively. We ensemble numerous new and existing consistency results of statistical optimal transport within our framework, providing sufficient conditions for Wasserstein consistency in the frequentist and Bayesian settings.

Given that the optimal transport machinery has extensive use in generative modelling, and our framework develops a statistical theory for optimal transport, we are able to provide a statistical analysis of various generative models in Section \ref{sec:6}, including Bonet et al.'s \cite{bonet2025flowing} hierarchical representation of labelled datasets, Chemseddine et al.'s \cite{chem} COT flow-matching objective, and Rigollet et al.'s \cite{rigollet2026meanfielddynamicstransformers}, \cite{glpr} transformer self-attention flows.\newline 

The results provide a natural, structured framework for optimal transport statistics and applications thereof.

\section*{Acknowledgements}
The authors would like to thank Philippe Rigollet for discussions at the inception of this project, which significantly advanced its scope.

\addtocontents{toc}{\protect\setcounter{tocdepth}{1}}

\subsection*{AI Statement}
We acknowledge the use of generative AI tools for assisting with the coding of the experiment in Section \ref{subsec:6:flowing_datasets} and its Appendix \ref{App:Flowing_datasets_experiments}. We have reviewed all AI-assisted code and figures, and take responsibility for the final content of the experiments.

We have \textbf{not} used AI in the writing of this paper, the development of its ideas and results, the construction of the proofs, or the verification of their correctness.

\bibliographystyle{alphaabbr}
\bibliography{Bibliography.bib}

\addtocontents{toc}{\protect\setcounter{tocdepth}{1}}

\hypertarget{app}{\appendix}
\setcounter{equation}{0}

\section{Deterministic optimal transport (Section \ref{sec:2})}

\subsection{Dynamics in Wasserstein over Wasserstein space}\label{App:WoW_dynamics}

The dynamics detailed in this section summarise the works of Bonet et al. \cite{bonet2025flowing} and Pinzi and Savaré \cite{pinzi2025nestedsuperpositionprinciplerandom}.

To consider functional dynamics at the Wasserstein over Wasserstein (WoW) level, we would like to define and identify a WoW gradient for appropriate functionals from $\Pcpc\to\R$. In preparation, we first define the space of cylinder functions, as Bonet et al. do in Definition 3.5 of \cite{bonet2025flowing}, providing a suitable base class of Wasserstein-differentiable functionals.
\begin{definition}[Cylinder functional]
    A functional $\cF:\Pc\to\R$ is cylindrical if there exists $k\geq0$, $F\in C_c^\infty(\R^k)$, potentials $V_1,\dots, V_k\in C_c^\infty(\R^d)$ such that for all $\mu\in\Pc$,
    \begin{equation}\label{eq:W2_cylinder}
        \cF(\mu) = F\l(
            \int V_1 \d \mu,\dots, \int V_k\d\mu
        \r).
    \end{equation}
    We denote this $\cF\in\mathrm{Cyl}\l(\Pc\r)$.
\end{definition}
Applying Proposition \ref{prop:w2_grad} and making use of the chain rule on \eqref{eq:W2_cylinder}, any $\cF\in\mathrm{Cyl}\l(\Pc\r)$ is Wasserstein-differentiable with
\begin{equation*}
    \gradW\cF(\mu) = \sum_{i=1}^k \frac{\partial}{\partial x_i} F\l(\int V_1\d\mu, \dots, \int V_k\d\mu\r)\grad V_i,\qquad \mu\in\Pc.
\end{equation*}
These cylinder functions are dense on appropriate spaces - in this manner, the closure of Wasserstein gradients of cylinder functionals over $\Pc$ determines the tangent space to $\Pcpc$ analogously to how the closure of gradients of compactly supported smooth functionals over $\R^d$ defines the tangent space to $\Pc$.
\begin{definition}[Tangent space to $\Pcpc$]\label{def:WoW_tangent}
    The tangent space to $\Pcpc$ at $\P$ is given by
    \begin{equation*}
        \cT_{\P}\Pcpc = \overline{\bigg\{
            \gradW \varphi \;\big\vert\; \varphi\in \mathrm{Cyl}\l(\Pc\r), \quad \|\varphi\|_\P <\infty
        \bigg\}}^{L^2(\P)}
    \end{equation*}
    where the space is endowed with the $L^2(\P)$ inner product,
    \begin{equation*}
        \langle V,\, W \rangle_\P \deq \int \big\langle V(\mu,\cdot), W(\mu,\cdot) \big\rangle_\mu\, \P(\d \mu) = \int_{\Pc} \int_{\R^d} \big\langle V(\mu,x), W(\mu,x) \big\rangle_{\R^d} \;\mu(\d x)\P(\d \mu).
    \end{equation*}
\end{definition}
\begin{remark}
    This inner product, a Wasserstein Dirichlet form, is noticeably similar to \eqref{eq:lifted_ke_norm} from Definition \ref{def:Lifted_tangent}. Indeed, the laws of those tangent vectors determine tangent vectors in this space in accordance with the diagram in Figure \ref{fig:superposition}.
\end{remark}

Under this definition, Pinzi and Savaré \cite{pinzi2025nestedsuperpositionprinciplerandom} are able to describe the equivalent of the continuity equation for measures in $\Pc$ at this superposed level $\Pcpc$, defining a superposed continuity equation which is satisfied in duality with these cylinder functions.
\begin{definition}[Continuity equation on random measure laws]\label{def:WoW_continuity}
    Let $(\P_t)_{t\in[0,T]}$ be a continuous curve of probability measures in $\Pcpc$, i.e. $(\P_t)_{t\in[0,T]}\in C_T\l(\Pcpc\r).$ Let $V_t$ be a tangent vector in the tangent space to $\P_t$, with $V_t(\mu,\cdot)\in L^2(\mu;\R^d)$ for all $\mu\in\Pc$. \newline
    
    The curve $(\P_t)$ is said to satisfy the continuity equation in the weak sense if
    \begin{equation}\label{eq:WoW_continuity}
        \frac{\partial}{\partial t} \P_t + \divergence_{\Pc} \l(\P_tV_t\r) = 0,\qquad\text{i.e.}\qquad \partial_t \int_{\Pc} \cF(\mu) \,\P_t(\d \mu) = \big\langle \gradW \cF(\mu)(\cdot), V_t(\mu, \cdot)\big\rangle_{\P_t}.
    \end{equation}
    for all $\cF\in\mathrm{Cyl}\l(\Pc\r)$,
\end{definition}
With this in hand, the authors of \cite{pinzi2025nestedsuperpositionprinciplerandom} are then able to replicate the Benamou-Brenier Theorem \ref{thm:benamou_brenier} over $\Pc$, now at the superposed level of $\Pcpc$ for random measure laws. As in the classical setting, the result expresses a dynamic formulation for the WoW optimal transport problem, demonstrating that it acts as a (superposed) kinetic energy minimisation.

\begin{theorem}[WoW Benamou-Brenier Formula]\label{thm:WoW_BB}
    For all $\P_0,\P_1\in\Pcpc$, the following dynamic representation holds
    \begin{equation*}
        W_{W_2}^2(\P_0,\P_1) = \min_{\P_t\in\mathrm{CE}} \l\{\int_0^1
            \int_{\Pc} \int_{\R^d} \l\|V_t(x,\mu)\r\|^2 \mu(\d x)\P_t(\d \mu) dt
        \r\}
    \end{equation*}
    subject to the superposed continuity equation constraints $\mathrm{CE}$.
    \begin{equation*}
        \mathrm{CE} = \l\{(\Q_t)_{t\in[0,1]}\in \mathrm{AC}^2\l(0,1;\Pcpc\r)
            \;\bigg\vert\; \Q_0 = \P_0,\, \Q_1 = \P_1, \, \partial_t\Q_t \divergence_{\Pc}(\Q_tV_t) = 0
        \r\}.
    \end{equation*}
\end{theorem}
The result provides a set of geodesic paths between random measure laws in WoW, though uniqueness is not guaranteed without stronger regularity restrictions (which would then permit a Brenier-type uniqueness as in \cite{beiglboeck}), as the optimal coupling determining the interpolation is not necessarily unique. As Pinzi and Savaré describe in \cite{pinzi2025nestedsuperpositionprinciplerandom}, the geodesics superpose straight-line paths on $\R^d$ to geodesics between measures on $\Pc$, then further superpose these geodesic paths to $\Pcpc$.
\begin{proposition}[WoW is a geodesic space]\label{prop:WoW_geodesic}
    Let $\P_0,\P_1\in\Pcpc$, and let $\gamma\in\Gamma_o(\P,\Q)$ be an optimal WoW coupling between $\P,\Q$. Then for the displacement interpolation on the coupling $\disp_t(x,y) \deq (1-t)x + ty,\quad t\in[0,1],$ the path $\P_t$ given by ${\disp_t}_\#\gamma$ is a constant-speed geodesic in $\l(\Pcpc,W_{W_2}\r)$ connecting $\P_0$ to $\P_1$.
\end{proposition}

\subsection{Wasserstein over Wasserstein gradient flows}\label{App:WoW_gradients}

With the necessary machinery available, we can begin to define the Wasserstein over Wasserstein gradient for suitable functionals $\bF:\Pcpc\to\R$. 
WoW gradient flows are defined for functionals $\bF:\Pcpc\to\R$ using sub- and super-differentials as in section 3.2 from \cite{bonet2025flowing}.
\begin{definition}[Sub-/Super-differentials]
    $\xi\in L^2\l(\P;\cT\Pc\r)$ belongs to the sub-differential $\partial^-\bF(\P)$ of $\bF$ at $\P$ if for all $\Q\in\Pcpc$, 
    \begin{align*}
        \bF(\Q)\geq  \bF(\P) + &\sup_{\varpi\in\Gamma_{\P,\Q}}\iint\big\langle \xi(\mu)(x),\, y-x \big\rangle \,\gamma_{\mu,\nu} (\d x,\d y) \,\varpi(\d \mu,\d \nu) + o\l(W_{W_2}(\P,\Q)\r).
    \end{align*}
    Similarly, $\xi\in L^2(\P,\cT\Pc)$ belongs to the super-differential $\partial^+\bF(\P)$ of $\bF$ at $\P$ if $-\xi\in\partial^-(-\bF)(\P)$
\end{definition}
The Wasserstein over Wasserstein gradient, where it exists, is defined when the intersection of sub- and super-differentials is non-empty.
\begin{definition}[WoW derivative]\label{def:WoW_derivative}
    $\bF$ is Wasserstein-differentiable at $\P\in\Pcpc$ if $\partial^+\bF(\P)\cap\partial^-\bF(\P)\neq\oldemptyset$. In this case, we say that $\xi\in\partial^+\bF(\P)\cap\partial^-\bF(\P)$ is a WoW gradient of $\bF$ at $\P$, and it satisfies for any $\Q\in\Pcpc$, $\varpi\in\Gamma_{\P,\Q}$,
    \begin{equation*}
        \bF(\Q) = \bF(\P) + \iint \big\langle \xi(\mu)(x),\, y-x \big\rangle \,\gamma_{\mu,\nu} (\d x,\d y) \,\varpi(\d \mu,\d \nu) + o\l(W_{W_2}(\P,\Q)\r).
    \end{equation*}
\end{definition}
Bonet et al. provide some examples, proving that for the classes of WoW functionals below, the WoW gradient coincides with the Wasserstein gradient first variation of the functional at the level of $\Pcpc$, just as we were able to similarly identify that the gradient for Wasserstein functionals over $\Pc$ is the Euclidean gradient of the first variation.
\begin{definition}[WoW via first variation]
    Let $\bF:\Pcpc\to\R$ be a functional. The first variation $\frac{\delta\bF}{\delta\P}(\P):\Pc\to\R$ satisfies
    \begin{equation*}
        \lim_{\varepsilon\to0}\frac{\bF(\P+\varepsilon\bX - \bF(\P))}{\varepsilon} = \int \frac{\delta\bF}{\delta\P}(\P)\d\chi,
    \end{equation*}
    for all signed random measure laws $\bX$ such that $\P + \varepsilon\bX\in\Pc$ for all sufficiently small $\varepsilon$. Then the WoW gradient of a functional $\bF$, where this first variation exists, can be identified with the Wasserstein gradient of its first variation, i.e.
    \begin{equation*}
        \gradW_{\mathrm{WoW}}\bF(\P)(\mu,x)=\gradW\frac{\delta\bF}{\delta\P}(\P)\,(\mu,x)=\nabla_x\l({\frac{\delta}{\delta\mu}}\l(\frac{\delta\bF}{\delta\P}(\P)\r)[\mu]\,\r)(x).
    \end{equation*}
\end{definition}
The examples that Bonet et al. focus on take the form of free energy functionals:
\begin{itemize}
    \item Potential energy: Given $\cF:\Pc\to\R$, define the potential $\bV(\P) = \int \cF(\mu)\P(\d \mu).$ Similarly to Wasserstein gradients, the WoW gradient is $
        \gradW_{\mathrm{WoW}} \bV(\P)(\mu) = \gradW\cF(\mu).$
    \item Interaction energy: Given a kernel $\cK:\Pc\times\Pc\to\R$, we define interaction energies as $\bK(\P) = \iint \cK(\mu,\nu)\P(\d \mu)\P(\d \nu).$ The WoW gradient is
    \begin{equation*}
        \gradW_{\mathrm{WoW}} \bK(\P)(\mu) = \int \gradW_1\cK(\mu,\nu) + \gradW_2\cK(\nu,\mu)\P(\d \nu).
    \end{equation*}
\end{itemize}

\section{Wasserstein random variables (Section \ref{sec:3})}

\subsection{Proof of Remark \ref{remark:sigma_algebras}}\label{App:remark:sigma_algebras}
The Borel $\sigma$-algebras induced by the Wasserstein and weak topologies over $\cP$ coincide exactly.
\begin{proof}
We show both inclusions.
$\cB_{\mathrm{Weak}}\subset\cB_{W_2}$ is immediate given the fact that the Wasserstein topology is finer than the weak. For the reverse inclusion we claim that it suffices to show that the moment map $\cM_2(\cdot)$ is weakly-measurable. Indeed, given the equivalence of Theorem \ref{thm:weak_wass}, we can characterise the Wasserstein Borel $\sigma$-algebra as 
$\cB_{W_2} = \sigma\l(\cB_{\mathrm{weak}}\,,\, \cM_2^{-1}\l(\cB(\R_+)\r)\r)$
and so showing that $\cM_2^{-1}\l(\cB(\R_+)\r)\subset \cB_{\mathrm{weak}}$ lets us conclude.

Returning to the claim, one can consider a sequence of approximating functions
    \begin{equation*}
        \cM_{2}^n(\mu) \deq \int\l(\|x\|^2\wedge n\r)\,\mu(\d x)
    \end{equation*}
    which are each continuous with respect to the weak topology. By monotone convergence $\cM_2^n(\mu)\to\cM_2(\mu)$, hence the moment map is weakly-measurable.
\end{proof}

\subsection{Proof of Proposition \ref{prop:lifted_geodesics} (Constant-speed geodesics)}\label{App:prop:lifted_geodesics}

\begin{proof}
    We first note that due to Theorem \ref{thm:Measurability}, $\xi_t$ is measurable for each fixed $t$. Indeed, $\nabla\varphi_{\om}$ is determined via an optimal coupling shown to have a measurable selection.
    
    We know from Proposition \ref{prop:Wasserstein_is_geodesic} that, pathwise, \eqref{eq:lifted_constant_speed} is the constant-speed geodesic in $\P_{2,ac}(\R^d)$ joining $\xi_{\om,0}$ to $\xi_{\om,1}$ and so $W_2\l(\xi_{\om,s}, \xi_{\om,t}\r) = |s-t|W_2(\xi_{\om,0},\xi_{\om,1}),$ 
    from which we calculate,
    \begin{align*}
        d(\xi_s,\xi_t) &= \E_\om\l[W_2^2\big(\xi_{\om,s}, \xi_{\om,t}\big)\r]^\frac{1}{2}
        =\E_\om\l[|s-t|^2\,W_2^2\big(\xi_{\om,0}, \xi_{\om,1}\big)\r]^{\frac{1}{2}} = |s-t|d(\xi_0,\xi_1)
    \end{align*}
    as desired. As a convex combination of a random probability measure and a linear pushforward, $\xi_t$ defines a random probability measure with sample paths in $\Pac$. Moreover, for each $t\in[0,1]$, the second energy of $\xi_t$ is bounded: $\cE_2[\xi_t] \leq 2\E\l[W_2^2(\xi_0, \xi_t)\r] + 2\cE_2[\xi_0] = 2t^2d^2(\xi_0,\xi_1) + 2\cE_2[\xi_0] <\infty.$
    thus $\xi_t$ is indeed in $\wlac$.
\end{proof}

\subsection{Proof of Theorem \ref{thm:lifted_benamou-brenier} (Random Benamou-Brenier)}\label{App:thm:lifted_benamou-brenier}
\begin{proof}
    Let $X_0\sim\mu_0$, $\dot X_t = v_t(X_t)$. The proof of the (deterministic) Benamou-Brenier theorem, as in \cite{chewi2025statistical} Theorem 5.7, shows that for all $(\mu_t,v_t)_{t\in[0,1]}$ subject to the continuity equation \eqref{eq:continuity}    \begin{equation*}
        W_2^2(\mu_0,\mu_1)\leq\E\l[\|X_0-X_1\|^2\r] =
        \E\l[\l\|\int_0^1\dot X_t dt\r\|^2\r] 
        \leq \int_0^1\E\l[\l\|\dot X_t\r\|^2\r]dt
        = \int_0^1\|v_t\|^2_{\mu_t}dt.
    \end{equation*}
    Pathwise for our random probability measures, this gives
    $$ W_2^2\big(\xi_{\om,0},\xi_{\om,1})\leq\int_0^1\big\|v_{\om,t}(\cdot)\big\|^2_{\xi_{\om,t}}dt,$$
     where each path ($\xi_{\om,t},v_{\om,t})$ is subject to the continuity equation. Taking expectations
    $$d^2(\xi_0,\xi_1) = \E_\om\bigg[W_2^2\big(\xi_{\om,0}, \xi_{\om,1}\big)\bigg]\leq \E\l[\int_0^1\big\|v_{\om,t}(\cdot)\big\|^2_{\xi_{\om,t}}dt\r]$$
    so that $d^2(\xi_0,\xi_1)$ indeed lower bounds the infimum in \eqref{eq:lifted_benamou_brenier}.
    We also know that pathwise, the constant-speed geodesics \eqref{eq:lifted_constant_speed} attain
$$W_2^2(\xi_{\om,0},\xi_{\om,1})=\int_0^1\big\|v_{\om,t}(\cdot)\big\|^2_{\xi_{\om,t}}dt,\qquad \text{$\PP$-a.s.}$$
$$
        v_{\om,t}(\cdot) = (\grad\varphi_\om - \id)\circ\l(\grad\varphi_{\om,t}\r)^{-1}(\cdot),\qquad \varphi_{\om,t}(x) \deq (1-t)\frac{\|x\|^2}{2} + t\varphi_{\om}(x),$$
    where $\varphi_\om$ is the Brenier map from $\xi_{\om,0}$ to $\xi_{\om,1}$ (measurably selected as a consequence of Theorem \ref{thm:Measurability}).
    So indeed when $\xi_t$ is given by \eqref{eq:lifted_constant_speed},
    $$d^2(\xi_0,\xi_1) = \E_\om\bigg[W_2^2(\xi_{\om,0},\xi_{\om,1})\bigg] = \E_\om\l[
        \int_0^1\|v_{\om,t}(\cdot)\|^2_{\xi_{\om,t}}dt
        \r].
    $$
    Thus the infimum on the RHS of \eqref{eq:lifted_benamou_brenier} is also attained.
\end{proof}

\subsection{Proof of Theorem \ref{thm:LiftedgradW} ($\WL$ gradient)}\label{App:thm:LiftedgradW}

\begin{proof}
    Let ${(\xi_t)}_{t\ge 0}$ be a curve of random probability measures in $\WL$ with corresponding tangent vectors ${(V_t)}_{t\ge 0}$.
    The fact that $V_t$ is the tangent vector at time $t$ means that each path $V_{\om,t}(\cdot)$ solves the continuity equation~\eqref{eq:continuity}.
    We calculate
    \begin{equation*}\label{eq:interchange}
        \partial_t \mF(\xi_t)
        = \partial_t \,\E_\om\bigg[\cF_\om(\xi_{\om,t})\bigg]
        = \E_\om\bigg[\partial_t \cF_\om(\xi_{\om,t})\bigg]
    \end{equation*}
    where we have interchanged the partial derivative and the expectation under suitable integrability conditions\footnote{These for example could be the existence of a suitable dominating function.} of the random velocity field $V_{\om,t}(\cdot)$ and the pathwise Wasserstein gradient $\grad \delta \cF(\xi_\om)(\cdot)$.
    We then apply the continuity equation \eqref{eq:continuity} (now inside the expectation)
    \begin{align*}
        \E_\om\bigg[\partial_t \cF_\om(\xi_{\om,t})\bigg] &= \E_\om\l[\int \delta \cF_\om(\xi_{\om,t})\partial_t \xi_{\om,t}\r] = \E_\om\l[-\int \delta \cF_\om(\xi_{\om,t}) \divergence\big(\xi_{\om,t} V_{\om,t}(\cdot)\big)\r]
    \end{align*}
    after which integration by parts (again inside the expectation) yields
    \begin{align*}
        \E_\om\l[-\int \delta \cF_\om(\xi_{\om,t}) \divergence\big(\xi_{\om,t} V_{\om,t}(\cdot)\big)\r] &= \E_\om\l[\int \langle \nabla \delta \cF_\om(\xi_{\om,t}), V_{\om,t}(\cdot)\rangle d \xi_{\om,t}\r],
    \end{align*}
    which we can identify as the inner product in the tangent space $\cT_{\xi_t}\wlac$ from Definition \ref{def:Lifted_tangent},
    \begin{align*}
       \E_\om\l[\int \big\langle \nabla \delta \cF_\om(\xi_{\om,t})\,,\, V_{\om,t}(\cdot)\big\rangle d \xi_{\om,t}\r] &= \E_{\om}\bigg[ \big\langle\nabla \delta \cF(\xi_t)\,,\, V_t\big\rangle_{\xi_{\om,t}} \bigg]= 
        \big\langle \nabla \delta \cF(\xi_t), V_t\big\rangle_{\xi_t}.
    \end{align*}
    Moreover, considering that the paths $\nabla \delta \cF_\om(\xi_{\om,t})$ are gradients of some function, from Definition \ref{def:Lifted_tangent} we indeed find that $\nabla \delta \cF_\om(\xi_t) \in \cT_{\xi_t} \wlac$. By identification, we formally conclude that the random functional $\nabla \delta \cF_\om(\xi_{\om,t})$ is indeed the $L^2$ over Wasserstein gradient of $\mF$ at $\xi_t$.
\end{proof}

\subsection{Proof of Proposition \ref{prop:WGF_lifted_stability} (Propagation of chaos)}\label{App:prop:WGF_lifted_stability}

We present a standard argument.
\begin{proof}
    Fix $\om\in\Omega$, and denote for convenience $\mu_t\deq \xi_{\om,t}, \nu_t\deq \eta_{\om,t}$. Let $X_0\sim \mu_0$ and let $Y_0$ be the image of $X_0$ under the Brenier map from $\mu_0$ to $\nu_0\in\Pac$ so that they are optimally coupled and $W_2^2(\mu_0,\nu_0) = \E\l[\|X_0-Y_0\|^2\r]$.
    
    For $t>0$, the gradient flow propagates $X_t,Y_t$,
    \begin{equation}\label{eq:derivative_wgf}
        \frac{d}{dt}\|X_t - Y_t\|^2 = 2 \l\langle X_t - Y_t \;,\; \dot{X_t} - \dot{Y}_t\r\rangle = -2 \bigg\langle X_t - Y_t\;,\; \gradW \cF(\mu_t)(X_t) - \gradW \cF(\nu_t)(Y_t)\bigg\rangle 
    \end{equation}
since the measures are subject to Wasserstein gradient flow generated by $\cF$. We decompose the velocity difference inside the inner product above as,
\begin{equation*}
    \underbrace{\gradW \cF(\mu_t)(X_t) - \gradW \cF(\nu_t)(X_t)}_{\text{Measure difference}} + \underbrace{\gradW \cF(\nu_t)(X_t) - \gradW \cF(\nu_t)(Y_t)}_{\text{Spatial difference}}
\end{equation*}
where we have added and subtracted the middle terms above. Given the Lipschitz assumptions, we find for some $L_1,L_2>0$,
\begin{align*}
    \l\|\gradW \cF(\mu_t)\l(X_t\r) - \gradW \cF(\nu_t)\l(X_t\r) \r\| &\leq L_1 W_2(\mu_t, \nu_t),\quad
    \l\|\gradW \cF(\nu_t)\l(X_t\r) - \gradW \cF(\nu_t)(Y_t)\r\| &\leq L_2 \l\|X_t - Y_t\r\|.
\end{align*}
Inserting these estimates into \eqref{eq:derivative_wgf} and applying Cauchy-Schwarz followed by Young's $2ab\leq a^2 + b^2$,
\begin{align*}
    \frac{d}{dt}\|X_t - Y_t\|^2 \leq 2L_1 W_2\big(\mu_t, \nu_t\big)\l\| X_t - Y_t \r\| + 2L_2 \|X_t - X_t \|^2\leq L_1^2 W_2^2\big(\mu_t, \nu_t\big) + (1+2L_2) \l\| X_t - Y_t \r\|^2.
\end{align*}
Let $f(t) \deq \E[\|X_t-Y_t\|^2]$. While the coupling $(X_t,Y_t)$ is optimal at $t=0$, it is not necessarily so for $t>0$ thus $W_2^2\big(\mu_t, \nu_t\big)\leq f(t)$.
Taking expectations in the derivative inequality above gives,
\begin{equation*}
    \dot{f}(t) \leq L_1^2 f(t) + (1+ 2L_2)f(t) = Lf(t)
\end{equation*}
with $L = L_1^2 + 1 + 2L_2$. Thus by Grönwall's lemma, 
\begin{equation*}
    W_2^2(\mu_t, \nu_t) \leq f(t) \leq e^{LT}f(0) = e^{LT} W_2^2(\xi_{\om,0}, \eta_{\om,0}),
\end{equation*}
which is uniform in $t$ given $t\leq T$, and we write
$\sup_{t\in[0,T]}  W_2^2\big(\mu_t, \nu_t\big) \leq e^{LT} W_2^2\big(\xi_{\om,0}, \eta_{\om,0}\big).$
a) follows with $\om$ in a set of probability 1 where the gradient is well-defined. Taking expectations (bounded using the fact that the $\WL$ distance between the initialisations is bounded), we conclude b).

\end{proof}

\subsection{Dynamics in $L^2$ over Wasserstein over Wasserstein space}\label{App:L2WoW}

The first piece of accessible random dynamic machinery are sampled geodesic paths in Wasserstein over Wasserstein space. Proposition \ref{prop:WoW_geodesic} provides the geodesic paths in WoW, which superpose the Wasserstein geodesics in $\Pc$ to the WoW space $\Pcpc$ (explored in more detail in \cite{pinzi2025nestedsuperpositionprinciplerandom}). In much the same manner as Proposition \ref{prop:lifted_geodesics}, geodesics in $\WWL$ are constructed from pathwise geodesics in $\Pcpc$, as below.
\begin{proposition}
    For $\zeta_0,\zeta_1\in\WWL$, let 
       $\zeta_{\om,t} \deq (\disp_t)_\# \gamma_{\zeta_0,\zeta_1}(\om)$
    where $\disp_t(x,y) = (1-t)x+ty$ and $\gamma_{\zeta_0,\zeta_1}(\om)$ is the (measurably selected) sample optimal coupling. Then ${\disp_t}_{\#}\zeta_{\om,0}$ is a constant-speed geodesic in $\WWL$, joining $\zeta_0$ to $\zeta_1$. Hence $\WWL$ is also a geodesic space.
\end{proposition}
The proof is analogous to the $\WL$ case, using the measurability and integrability of the ensemble of WoW geodesic paths. Recall in the context of the random Benamou-Brenier Theorem \ref{thm:lifted_benamou-brenier} over $\WL$, the random geodesic paths minimise the average KE from a random measure $\xi_0$ to $\xi_1$ in $\WL$. An analogous argument, which makes use of Theorem \ref{thm:WoW_BB} (\cite{pinzi2025nestedsuperpositionprinciplerandom} Theorem 1.3) in place of Theorem \ref{thm:benamou_brenier}, shows that the $\WWL$ distance between $\zeta_0,\zeta_1$ can be dynamically represented as a minimisation of the average (superposed) kinetic energy over paths connecting the points in $\WWL$, subject to the superposed continuity equation \eqref{eq:WoW_continuity}.
\begin{theorem}[Random WoW Benamou-Brenier]
    Let $\zeta_0, \zeta_1\in \WWL$. Then 
    \begin{equation*}\label{eq:lifted_wow_benamou_brenier}
        d_{\mathrm{WoW}}^2(\zeta_0,\zeta_1) = 
        \min\l\{
            \E_\om\l[\int_0^1
                \big\|\cV_{\om,t}(\cdot)\big\|^2_{\zeta_{\om,t}}
            dt \r]\,\bigg|\,(\zeta_t(\om, \cdot), \cV_t(\om, \cdot))_{t\in[0,1]}\text{ solves \eqref{eq:WoW_continuity} $\PP$-a.s.}
        \r\}
    \end{equation*}
\end{theorem}

This result provides us with the motivation for the appropriate definition of tangent vectors to $\WWL$, as $L^2$ random vectors endowed with the (lifted) KE norm, the same way as we defined for $\WL$ in Definition \ref{def:Lifted_tangent}, now using the WoW tangent space from Definition \ref{def:WoW_tangent} as the closure of Wasserstein gradients of cylinder functionals.
\begin{definition}[Random WoW tangent vectors]
    Let $\zeta\in\WWL$. We define the tangent space to $\WWL$ at $\zeta$ to be the $L^2$ space of measurable functions from $\Omega\times\cP(\R^d)\times\R^d\to\cP(\R^d)\times\R^d$ with $\PP$-a.s. samples in $\cT_{\zeta_\om}\Pcpc$, recalling
    \begin{align*}
        \cT_{\zeta_\om} \Pcpc
        &\deq \overline{\bigg\{
            \gradW \varphi \;\big\vert\; \varphi\in \mathrm{Cyl}\l(\Pc\r), \quad \|\varphi\|_{\zeta_\om} <\infty
        \bigg\}}^{L^2(\zeta_\om)}.
    \end{align*}
    
    \noindent More explicitly, let $\PP_\zeta\in\cP\l(\Omega\times\cP(\R^d)\times\R^d\r)$ be given by $\PP_\zeta(\d\om,\d \mu, \d x)\deq \PP(\d\om)\zeta(\om,\d \mu)\mu(\d x).$ Then the tangent space is given by 
    \begin{equation*}
        \cT_\zeta\WWL = {\bigg\{
            \cV\in L^2\l(\PP_\zeta;\R^d\r)\,\Big|\,\cV_\om(\cdot, \cdot)\in \cT_{\zeta_\om}\Pcpc, \,\,||\cV||_\zeta<\infty
        \bigg\}},
    \end{equation*}
    where we have endowed the tangent space $\cT_\zeta\WWL$ with the inner product
    \begin{align*}\label{eq:lifted_wow_ke_norm}
        \langle \cV, \cW \rangle_\zeta &\deq \E_\om\bigg[\big\langle \cV_\om(\cdot, \cdot), \cW_\om( \cdot, \cdot)\big\rangle_{\zeta_\om}\bigg]=\iiint_{\Omega\times\cP(\R^d)\times\R^d} \big\langle \cV_\om( \mu, x), \cW_\om( \mu, x)\big\rangle\,\PP_\zeta(\d\om, \d \mu, \d x).
    \end{align*}

\end{definition}

With this appropriate differential structure, we prepare the formulation of the random gradient flow on $\WWL$. Consider a collection of functionals $\bF_\om:\Pcpc\to\R$ for each $\om\in\Omega$ as in the previous Appendix \ref{App:WoW_dynamics}, now indexed over the sample space to give us a random functional.

\begin{definition}[Random WoW functionals]
    Let $\{\bF_\om\}_{\om\in\Omega}$ be a collection of measurable functionals over $\Pcpc$ i.e. for each $\om\in\Omega$, $\bF_{\om}:\Pcpc\to\R$ . Let $\bF:\Omega\times\WWL\to\R$ be given by $\bF(\om,\xi) = \bF_{\om}(\xi_\om).$ We call $\bF$ a random (WoW) functional. We define the corresponding lifted functional as the map $\mF:\WWL\to\R$, $\mF(\zeta) = \E[\bF(\zeta)] = \E_{\om}[\bF_\om(\zeta_\om)].$
\end{definition}

\subsection{Random Wasserstein over Wasserstein gradient flows}\label{App:Random_WoW_gfs}

While the formulation of the WoW derivative can be more general, as in Definition \ref{def:WoW_derivative} using the identification of the sub-/super-differentials, we focus on the case of free energy functionals $\bF_\om$ where it has been shown by Bonet et al. in \cite{bonet2025flowing} that we can identify the WoW gradient with the Wasserstein gradient of the first variation. Equivalently to the random $\WL$ gradient, we find that the dynamic structure endowed on $\WWL$ induces a random gradient with WoW gradient sample paths.

\begin{theorem}[Random WoW gradient]
The function $\gradW_{L^2_{\mathrm{WoW}}}\mF(\zeta):\Omega\times\Pc\times\R^d\to\Pc\times\R^d$ given by
        $\gradW_{{L^2}_{\mathrm{WoW}}}\mF\big(\zeta\big)(\om) = \gradW \frac{\delta}{\delta \bF}\bF\big(\zeta_\om\big) $
    is $L^2$ over WoW gradient of $\mF$ at $\zeta\in\WWL$ under suitable integrability conditions. Phrased differently, the lifted $\WWL$ gradient is a random gradient with WoW gradient sample paths.
\end{theorem}
This in turn defines a dynamic random gradient flow in continuous time in correspondence with the continuity equation, with WoW gradient flow sample paths.

\begin{definition}[Random WoW gradient flow]
    Let $\{\bF_\om\}_{\om\in\Omega}$ be a collection of functionals on $\Pcpc$ and let the lifted functional be given by the map $\mF:\WWL\to\R$, $\mF(\zeta) = \E[\bF(\zeta)] = \E_{\om}[\bF_\om(\zeta_\om)]$. ${(\zeta_t)}_{t\ge 0}$ is the {Wasserstein gradient flow} of $\mF$ if it solves the PDE below $\PP$-a.s.
    \begin{equation*}\label{eq:WoWGF_lifted}
        \partial_t \zeta_{\om,t} = \divergence_{\Pc}\bigg(\zeta_{\om,t} \gradW_{L_{\mathrm{WoW}}^2} \mF(\zeta_t)(\om)\bigg) = \divergence_{\Pc}\bigg(\zeta_{\om,t} \gradW_{\mathrm{WoW}} \cF_\om\big(\zeta_{\om,t}\big)\bigg),
    \end{equation*}
    in the weak sense of Definition \ref{def:WoW_continuity} i.e. $\zeta_t$ has sample paths which are Wasserstein gradient flows of $\cF_\om$, as in Definition \ref{def:w2_grad_flow}.
\end{definition}
Finally, we formulate a simple, Grönwall-based, propagation of chaos result in the same vein as Proposition \ref{prop:WGF_lifted_stability} - now at the superposed WoW level.
\begin{proposition}[WoW propagation of chaos]
    Let $\bF$ be a $C^1$ WoW functional with Lipschitz WoW gradient, and let $(\Q_t), (\Q_t')$ be WoW gradient flows started at the respective initialisations $\Q_0,\Q_0'\in \Pcpc$. Then for all $T>0$ there exists a Lipschitz constant $L>0$ such that
    \begin{equation*}
        \sup_{t\in[0,T]}W_{W_2}^2\l(\Q_t,\Q_t'\r)\leq e^{2LT}W_{W_2}^2\l(\Q_0,\Q_0'\r).
    \end{equation*}
\end{proposition}
\begin{proof}
    Let $\gamma_0 \in \Gamma_o(\Q_0,\Q_0')$ be an optimal coupling, and draw $(\mu_0,\nu_0)\sim\gamma_0$. Let $(\mu,\nu)\mapsto \gamma_{\mu,\nu}\in\Gamma_o(\mu,\nu)$ be a measurable selection and further draw the conditional $(X_0,Y_0)\sim\gamma_{\mu,\nu}$. Then
    \begin{equation*}
        \E\l[\|X_0-Y_0\|^2\r] = \int_{\Pc\times\Pc} W_2^2(\mu,\nu)\gamma_0(\d \mu,\d \nu) = W_{W_2}^2(\Q_0,\Q_0').
    \end{equation*}
    Let $\mu_t=S_t(\mu_0)$, $\nu_t=S_t'(\nu_0)$, and let $X_t,Y_t$ solve
    $$\dot X_t=-\gradW_{\mathrm{WoW}}{\bF}(\Q_t)(\mu_t,X_t),\qquad \dot Y_t=-\gradW_{\mathrm{WoW}}{\bF}(\Q_t')(\nu_t,Y_t).$$
    Now let $f(t)\deq\E\l[\|X_t-Y_t\|^2\r]$. For each fixed $(\mu_0,\nu_0)$, the pair $(X_t,Y_t)$ is a coupling of $(\mu_t,\nu_t)$, and $(\mu_t,\nu_t)\sim(S_t,S_t')_\#\gamma_0$ is a coupling of $(\Q_t,\Q_t')$. Hence by the sub-optimality of each coupling,
    $$W_{W_2}^2(\mu_t,\nu_t)\leq\E\l[\|X_t-Y_t\|^2\r],\qquad W_{W_2}^2\l(\Q_t,\Q_t'\r)\leq\E\l[W_2^2(\mu_t,\nu_t)\r]\leq f(t).$$

    Write the difference $\gradW_{\mathrm{WoW}}(\Q_t)(\mu_t,X_t)-\gradW_{\mathrm{WoW}}(\Q_t')(\nu_t,Y_t) =$
    \begin{align*}
        & \underbrace{\gradW_{\mathrm{WoW}}(\Q_t)(\mu_t,X_t)-\gradW_{\mathrm{WoW}}(\Q_t')(\mu_t,X_t)}_{\text{law-of-laws difference in $\Pcpc$}}+\underbrace{\gradW_{\mathrm{WoW}}(\Q_t')(\mu_t,X_t)-\gradW_{\mathrm{WoW}}(\Q_t')(\nu_t,X_t)}_{\text{measure difference in $\Pc$}}+\\
        &\underbrace{\gradW_{\mathrm{WoW}}(\Q_t')(\nu_t,X_t)-\gradW_{\mathrm{WoW}}(\Q_t')(\nu_t,Y_t)}_{\text{spatial difference in $\R^d$}}.
    \end{align*}
    
The three terms are bounded by $L_1W_{W_2}(\Q_t,\Q_t')$, $L_2W_2(\mu_t,\nu_t)$ and $L_3\|X_t-Y_t\|$ respectively for some $L_1,L_2,L_3>0$ given the Lipschitz assumptions on the WoW gradient. By Cauchy–Schwarz and Young’s inequality $2ab\le a^2+b^2$,
$$\frac{d}{dt}\|X_t-Y_t\|^2\leq L_1\bigg(|X_t-Y_t\|^2+W_{W_2}^2(\Q_t,\Q_t')\bigg)+L_2\bigg(\|X_t-Y_t\|^2+W_2^2(\mu_t,\nu_t)\bigg)+2L_3\|X_t-Y_t\|^2$$
Taking expectations and using the above:
$$f(t)\leq f(0)+\int_0^tL\,f(s)\,\d s,\qquad L=2(L_1+L_2+L_3).$$
By Grönwall's inequality, $f(t)\le e^{Lt}f(0)$. Combining with all the above gives the result
$W_{W_2}^2(\Q_t,\Q_t')\leq f(t)\leq e^{Lt}W_{W_2}^2(Q_0,Q_0').$
\end{proof}
The proposition delivers the desired propagation of chaos result as a corollary in an equivalent manner to Proposition \ref{prop:WGF_lifted_stability}, showing the almost sure result and lifting to $L^2$ given the integrability.

\begin{corollary}
    Let $\bF$ be WoW-Lipschitz and take $\bF_\om=\bF$ constant. Let $\zeta_0,\zeta_0'\in \WWL$, and let $\zeta_{\om,t},\zeta'_{\om,t}$ be the WoW gradient flows of $\bF$ from $\zeta_{\om,0},\zeta'_{\om,0}$, as in Definition \ref{def:WoW_derivative}. Then for each $T>0$ and $L=2(L_1+L_2+L_3)$:
\begin{align*}
    a) \sup_{t\in[0,T]}W_{W_2}^2\l(\zeta_{\om,t},\zeta_{\om,t}'\r)&\leq e^{2LT}\,W_{W_2}^2\l(\zeta_{\om,0},\zeta'_{\om,0}\r),\qquad \text{$\PP$-a.s.}\\
   b) \sup_{t\in[0,T]}d_{\mathrm{WoW}}\l(\zeta_t,\zeta_t'\r)&\leq e^{2LT}\,d_{\mathrm{WoW}}\l(\zeta_0,\zeta_0'\r)
\end{align*}

\end{corollary}

\section{Statistical analysis with empirical measures (Section \ref{sec:4})}

\subsection{Proof of Proposition \ref{prop:L2W_Uniform_empirical} (Uniform convergence on compact spaces)}\label{App:prop:L2W_Uniform_empirical}
\begin{proof}
    The Kantorovich-Rubinstein duality from Theorem 5.10 of \cite{villani2009optimal} provides a dual representation of the 1-Wasserstein distance between $\mu,\nu\in\cP_1(\cK)$,
    \begin{equation*}
        W_1(\mu,\nu) = \sup_{f\in \text{Lip}_1(\cK)}\l|\int f\d\mu - \int f\d\nu\r|,
    \end{equation*}
    where $\text{Lip}_1(\cK)$ is the set of $1$-Lipschitz functions over $\cK$. By the compactness of $\cK$, this set is uniformly bounded by its diameter $\cD$, and is equicontinuous hence totally bounded in $\big(C(\cK),\|\cdot\|\big)$ by Arzelá-Ascoli. This provides the existence of a $\delta$-net $f_1,\dots, f_{N(\delta)}$ for any $\delta>0$. Hence we can always approximate
    \begin{equation*}
        W_1(\hmun,\mu)\leq 2\delta + \max_{m=1,\dots,N(\eta)}\l|\int f_m \d \mu - \int f_m \d \nu\r|
    \end{equation*}
    For the (randomly chosen) empirical measure $\nu=\hmun$,
    \begin{align*}
        \E[W_1(\hmu,\mu)] &\leq \delta + \sup_{f\in \Lip_1(\cK)}\l|\frac{1}{n}\sum_{i=1}^n f_m(X_i) - \int f_m \d \mu\r|\\
        &=2\delta + \max_{m=1,\dots,N(\eta)}\E\l[\l|\frac{1}{n}\sum_{i=1}^n \l(f_m(X_i) - \int f_m \d \mu\r)\r|\r]\\
        &\leq 2\delta + \sum_{m=1}^{N(\eta)}\E\l[\l|\frac{1}{n}\sum_{i=1}^n \l(f_m(X_i) - \int f_m \d \mu\r)\r|\r] \leq 2\delta + N(\delta)\frac{\cD}{\sqrt{n}}
    \end{align*}
    where we have used Cauchy-Schwarz and the fact that the $\|f_m\|_\infty\leq \cD$ are all Lipschitz. First letting $\ninf$ followed by $\delta\to0$ gives $\E[W_1(\hmun,\mu)]\to 0$ uniformly in $\mu$. On the compact domain, we find for a $W_1$ optimal coupling $\gamma$,
    \begin{equation*}
        W_2^2(\hmun,\mu)\leq \int_\cK \|x-y\|^2\gamma(\d x,\d y)\leq \cD\int_\cK \|x-y\|\gamma(\d x,\d y) = \cD W_1(\hmun,\mu).
    \end{equation*}
    Thus we also find the uniform convergence over $\cP_2(\cK)$ for
    $\E[W_2^2(\hmun,\mu)]\to 0$.
\end{proof}

\subsection{Proof of Proposition \ref{prop:W2_concentration} (Sub-Gaussian concentration)}\label{App:prop:W2_concentration}

\begin{proof}
    Consider modifying each coordinate individually, mapping $x_i$ to $x_i'$ which is sub-optimal in general, hence
    \begin{equation*}
        \l|f(x) - f(x')\r|\leq W_2^2\l(\frac{1}{n}\sum_{i=1}^n\delta_{x_i}\,, \,\frac{1}{n}\sum_{i=1}^n\delta_{x_i'}\r) \leq \frac{1}{n}\sum_{i=1}^n\|x_i-x_i'\|^2 = \frac{1}{n}\|x-x'\|^2.
    \end{equation*}
    The result then follows from the reverse triangle inequality.
\end{proof}

\subsection{Proof of Theorem \ref{thm:geodesic_empirical_convergence_uniform} (Geodesic consistency)}\label{App:thm:geodesic_empirical_convergence_uniform}

We first detail the pointwise convergence, and show equicontinuity to conclude uniform convergence. Towards proving the convergence, we first collect a more general lemma.
\begin{lemma}\label{lemma:interpolant_stability}
    Let $\mu,\nu\in\Pac$. Suppose that $W_2(\mu_k,\mu)\to0$ and $W_2(\nu_k,\nu)\to0$. Denote $\eta_t^k$, $\eta_t$ the geodesic interpolations from $\mu_k$ to $\nu_k$ and $\mu$ to $\nu$ respectively. Then over $\cP_2\l(\R^d\times\R^d\r)$, the geodesic interpolant converges $W_2(\eta_t^k,\eta_t)\to0$ pointwise.
\end{lemma}
\begin{proof}
    The equivalence in Theorem \ref{thm:weak_wass} dictates that the convergence in Wasserstein metric immediately implies the weak convergence $\mu_k\wto \mu$, $\nu_k\wto \nu$ and convergence of the second moments, $\cM_2(\mu_k) \to \cM_2(\mu)$, $\cM_2(\nu_k) \to \cM_2(\nu)$.\newline

    We first claim that the sequence of optimal couplings\footnote{which can be measurably selected under non-uniqueness.} $(\gamma_k\in\Gamma_{\mu_k,\nu_k})$ is tight, due to the sequences of corresponding marginals being tight. Indeed from the convergence of the second moments, $C_\mu\deq\sup_k \cM_2(\mu_k), C_\nu\deq\sup_k \cM_2(\nu_k)$ are both finite. Markov's inequality gives
    \begin{align*}
        \gamma_k\l((\bar B_R\times \bar B_R)^c\r)\leq \gamma_k\l((\bar B_R)^c\times\R^d\r) + \gamma_k\l(\R^d\times (\bar B_R)^c\r)\leq \frac{C_\mu+ C_\nu}{R^2},
    \end{align*}
    which can be made arbitrarily small, uniformly in $k$. As $\bar B_R\times \bar B_R$ is compact in $\R^d\times\R^d$, the sequence $(\gamma_k)$ is tight and by Prokhorov's theorem there is a subsequence $(\gamma_{k_j})$ converging weakly to some $\gamma'\in \cP(\R^d\times\R^d)$. This immediately gives weak convergence of the corresponding interpolants using the continuous mapping $\Phi_t:\R^d\times \R^d\to\R^d$, $(x,y)\mapsto (1-t)x + ty$,
    \begin{equation*}
        \eta_t^{k_j} \deq (\Phi_t)_\#\gamma_{k_j} \wto (\Phi_t)\#\gamma'\eqqcolon \eta_t'
    \end{equation*}
    as desired. Thus we have shown weak convergence of $\eta_t^{k_j}\wto \eta_t$ along the subsequence $(k_j)$.
    \newline
    
    Before verifying the convergence of the second moment for the subsequence $\eta_t^{k_j}$, we first verify $\gamma'$ is both a valid coupling, with marginals $\mu,\nu$ respectively, and in fact optimal for ($\mu,\nu$). We begin with the former. Fix $f\in C_b(\R^d)$, then $(x,y)\mapsto f(x)$ is bounded continuous on $\R^d\times\R^d$. Given the weak convergence of the couplings $\gamma_{k_j}\wto\gamma'$, 
    \begin{equation*}
        \int f(x)\,\gamma'(\d x,\d y) = \lim_{j\to\infty}\int f(x)\,\gamma_{k_j}(\d x,\d y) = \lim_{j\to\infty} \int f(x) \,\mu_{k_j}(\d x) = \int f \d \mu
    \end{equation*}
    where the last equality is $\mu_{k_j}\wto \mu$. We conclude by Portmanteau's equivalence that the first marginal of $\gamma'$ is $\mu$ - similarly, the second marginal is $\nu$. Hence $\gamma'\in\Gamma_{\mu,\nu}$.

    Towards optimality: consider the truncated cost $f_R(x,y) \deq \min\l(\|x-y\|^2, R\r)$ which is bounded continuous for $R>0$. We apply Portmanteau's theorem to the weakly convergent sequence of marginals $\gamma_{k_j}$,
\begin{equation*}\label{eq:portmanteau_coupling_optimality}
        \int f_R \d \gamma' = \lim_{j\to\infty} \int f_R \d \gamma_{k_j} \leq \liminf_{j\to\infty} \int \|x-y\|^2 \,\gamma_{k_j}(\d x,\d y) = \liminf_{j\to\infty} W_2^2(\mu_{k_j}, \nu_{k_j}) = W_2^2(\mu,\nu),
    \end{equation*}
    where the last equality follows from the $W_2$ convergence of $\mu_{k_j},\nu_{k_j}$ to $\mu,\nu$ respectively. Taking $R\to\infty$ above using monotone convergence gives
    \begin{equation*}
        \int \|x-y\|^2 \gamma' (\d x,\d y) \leq W_2^2(\mu,\nu),
    \end{equation*}
    thus $\gamma'\in\Gamma_{\mu,\nu}$ is indeed optimal.

    We now show the convergence of the second moments. We calculate
    \begin{equation*}\label{eq:dominating_function}
        \int \l(\|x\|^2 + \|y\|^2\r)\gamma_{k_j}(\d x,\d y) = \cM_2(\mu_{k_j}) + \cM_2(\nu_{k_j}) \to \cM_2(\mu) + \cM_2(\nu) = \int \l(\|x\|^2 + \|y\|^2\r)\gamma'(\d x,\d y),
    \end{equation*}
    given the convergence of second moments for the marginals.
    The modulus is convex thus $g(x,y) = \|(1-t)x+ty\|^2$ satisfies $g(x,y) \leq \|x\|^2 + \|y\|^2$. The above shows that $\|x\|^2 + \|y\|^2$ provides an integrable domination for $g(x,y)$, establishing uniform integrability of $g$ under $(\gamma_{k_j})$. For $R>0$, we split the integral with a tail term
    \begin{equation*}
      \int g \d \gamma_{k_j} = \int_{B_R\times B_R} g \d \gamma_{k_j} + \int_{(B_R\times B_R)^c} g \d \gamma_{k_j}.
    \end{equation*}
    The second term vanishes uniformly in $j$ as $R\to\infty$ by the domination $g(x,y) \leq \|x\|^2 + \|y\|^2$ paired with the tightness of the sequence $(\gamma_{k_j})$ and the fact that the weak limit $\gamma'$ has finite second moment. For the first term, $g$ is bounded continuous on the domain, hence
    \begin{equation*}
        \int_{B_R\times B_R} g\d\gamma_{k_j} \to \int_{B_R\times B_R}g \d \gamma',
    \end{equation*}
    for each $R>0$, given $\gamma_{k_j}\wto\gamma'$. Thus taking $R\to\infty $, we indeed find the second moment convergence $\cM_2(\eta_t^{k_j})\to \cM_2(\eta_t)$ along the subsequence $(k_j)$. Combining the weak convergence with convergence of the second moment, we conclude by Theorem \ref{thm:weak_wass} that along the subsequence $(k_j)_{j\geq1}$, $W_2\l(\eta_t^{k_j},\eta_t\r)\to0$.\newline
    
    To conclude convergence of the full sequence $\eta_t^{k}\to\eta_t$, recall that as $\mu,\nu\in\Pac$ the optimal coupling is unique (induced by a Brenier map), hence any limit point of couplings $\gamma_{k_j}$ must be equal to $\gamma'$. By a standard result, if every subsequence has a further subsequence converging to the same limit, then the whole sequence converges. We conclude the result.
\end{proof}

Combining Lemma \ref{lemma:interpolant_stability} above with the $\WL$ Law of Large Numbers Theorem \ref{thm:empirical_convergence} provides the pointwise result.
\begin{proposition}\label{prop:geodesic_empirical_convergence}
    Let $\mu,\nu \in\Pac$. Let $\eta_t = \l[(1-t)\id + t\nv^{\mu\to\nu}\r]_\#\mu$ be the displacement interpolation geodesic from $\mu\to\nu$. Let $\hmu_n, \hnu_n$ be the empirical measures estimating $\mu,\nu$ respectively. Then for each $t\in[0,1]$ the lifted coupling interpolation $\eta_{\om,t}^n \deq \l[(1-t)\pi^1 + t\pi^2\r]_\#\gamma_{\hmuno,\hnuno}$ (where $\gamma_{\hmu_n,\hnu_n}$ is an optimal coupling) satisfies a) $W_2(\eta_{\om,t}^n,\eta_t)\to 0$ $\PP$-a.s. and b) $d(\eta_t^n,\eta_t)\to 0$.
\end{proposition}
\begin{proof}
    a) From Theorem \ref{thm:empirical_convergence}, the empirical measures $\hmuno,\hnuno$ both converge a.s. to $\mu,\nu$ in $W_2$ respectively thus Lemma \ref{lemma:interpolant_stability} applies a.s.\newline
    
    b) We establish a (uniform in $t$) bound on $W_2(\eta_{\om,t}^n, \eta_t)$ to use a DCT argument. Fix $t\in[0,1].$
    \begin{equation}\label{eq:Empirical_energy}
        W_2(\eta_{\om,t}^n, \eta_t) \leq W_2(\eta_{\om,t}^n, \delta_0) + W_2(\delta_0, \eta_t) = \cM_2\l(\eta_{\om,t}^n\r)^{\frac{1}{2}} + \cM_2\l(\eta_t\r)^{\frac{1}{2}}.
    \end{equation}
    We use the bound
    \begin{equation*}\label{eq:energy_bound}
        \cM_2(\eta_t) = \int \|(1-t)x +ty\|^2 \gamma(\d x,\d y) \leq \int \l(\|x\|^2 + \|y\|^2\r) \gamma(\d x,\d y) = \cM_2(\mu) + \cM_2(\nu),
    \end{equation*}
    where $\gamma\in\Gamma_{\mu,\nu}$ is the optimal coupling for $(\mu,\nu)$. Similarly for $\eta_{\om,t}^n$, $\cM_2(\eta_{\om,t}^n)\leq \cM_2(\hmuno) + \cM_2(\hnuno)$. Applying this to \eqref{eq:Empirical_energy}, with $(a+b)^2\leq 2a^2+2b^2$, gives the uniform in $t$ bounding variable $H_n$,
    \begin{equation*}\label{eq:double_energy_bound}
        W_2^2(\eta_{\om,t}^n, \eta_t) \leq 2\bigg(
            \cM_2(\hmuno) + \cM_2(\hnuno) + \cM_2(\mu) + \cM_2(\nu)
        \bigg)\eqqcolon H_n(\om)
    \end{equation*}
        Thus denoting $Z_n(\om) \deq \sup_{t\in[0,1]}W_2^2(\eta_{\om,t}^n, \eta_t)$, $Z_n\leq H_n$ $\PP$-a.s. We find $
        \E[H_n] = 4(\cM_2(\mu)+\cM_2(\nu)),$
        as the moments of the empirical measures are unbiased. So $H_n$ is integrable and constant (hence convergent) in $n$. Moreover, $\cM_2(\hmuno)\to \cM_2(\mu)$, $\cM_2(\hnuno)\to \cM_2(\nu)$ $\PP$-a.s. from the strong Law of Large Numbers hence 
        $H_n(\om)\to 4(\cM_2(\mu)+\cM_2(\nu))$ $\PP$-a.s.
        Generalised DCT applies: As $Z_n\to 0$ a.s. by part a), we conclude $\E[Z_n] = d^2(\eta_t^n,\eta_t)\to0$ as desired.
\end{proof}

As the dominating sequence $H_n$ is uniform in $t$, we can upgrade to uniform convergence in $t\in[0,1]$, and complete the proof of Theorem \ref{thm:geodesic_empirical_convergence_uniform}.

\begin{proof}
   a) Fix $\om$ in the $\PP$ full measure event where Theorem \ref{thm:empirical_convergence} holds and the empirical measures converge $W_2(\hmuno,\hnuno)\to W_2(\mu,\nu)$. The convergence immediately gives $C_\om\deq \sup_{n\geq1}W_2(\hmuno,\hnuno)<\infty.$ Moreover, by Lemma \ref{lemma:interpolant_stability} above, we already have pointwise in $t$ convergence of the interpolants.\newline
    
    We turn to equicontinuity. As $W_2$ satisfies the triangle inequality, we find
    \begin{equation*}
        \Big|
            W_2
            (\eta_{\om,t}^n, \eta_t) - W_2(\eta_{\om,s}^n, \eta_s)
        \Big| \leq W_2(\eta_{\om,t}^n, \eta_{\om,s}^n) + W_2(\eta_{\om,t},\eta_{\om,t}).
    \end{equation*}
    As the interpolating geodesics are constant-speed (Proposition \ref{prop:lifted_geodesics}),
    \begin{equation}\label{eq:equicontinuity}
        W_2(\eta_{\om,t}^n, \eta_{\om,s}^n) + W_2(\eta_t,\eta_s) \leq |t-s| \l( C_\om + W_2(\mu,\nu) \r).
    \end{equation}
    Then the family $(f_n^\om)_{n\geq1}$, $f_n^\om(t)\deq W_2(\eta_{\om,t}^n,\eta_t)$ is equicontinuous in $n$ on $[0,1]$, with common Lipschitz constant $K_\om\deq C_\om + W_2(\mu,\nu) < \infty$.\newline

    Fix $\varepsilon>0$ and choose $\delta = \frac{1}{N}$ for some integer $N$ s.t. $\frac{1}{N} < \varepsilon/(3K_\om)$. Let $\{t_1,\dots,t_N\}$ be a uniform $N$-partition of $[0,1]$, $t_j = j/N$. Given the pointwise convergence at each $t_j$, there is a corresponding $n_j(\om)$ s.t. for all $n\geq n_j(\om)$:
    \begin{equation}\label{eq:eps/3_no1}
        W_2(\eta_{\om,t_j}^n,\eta_{t_j}) < \frac{\varepsilon}{3}.
    \end{equation}
    Setting $n^*_{\om} \deq \max_{1\leq j\leq N} n_j(\om)$. Now fix $t\in[0,1]$, and choose the grid-point $t_j$ s.t. $|t-t_j|<\frac{1}{N}$. By the triangle inequality,
    \begin{equation*}
        W_2\l(\eta_{\om,t}^n, \eta_t\r)\leq W_2\l(\eta_{\om,t}^n, \eta_{\om,t_j}^n\r) + W_2\l(\eta_{\om,t_j}^n, \eta_{t_j}\r) + W_2\l(\eta_{t_j}, \eta_t\r).
    \end{equation*}
    From the Lipschitz bound \eqref{eq:equicontinuity}, the first and last term on the RHS are both bounded by $K_\om\delta< \varepsilon/3$ by construction. Using the pointwise convergence \eqref{eq:eps/3_no1}, the middle term is also bounded by $\varepsilon/3$ thus we find $\forall n\geq n^*_{\om}$, 
        $W_2\l(\eta_{\om,t}^n, \eta_t\r) < \varepsilon$
    which is uniform in $t$. As $\varepsilon>0$ was arbitrary, we conclude the uniform convergence in a). b) follows from an analogous argument to the previous proposition, defining a dominating sequence and concluding via generalised DCT.
\end{proof}

\subsection{Proof of Proposition \ref{prop:LN_FG} (Rates for convergence of empirical geodesics)}\label{App:prop:LN_FG}

\begin{proof}
    We proceed in two steps: we first show that the distance $d^2(\eta_{\om,t}^n, \eta_t)$ is (uniformly) bounded by $d^2(\gamma_n,\gamma)$ where $\gamma_{\om,n}$ couples the empirical measures for each $\om\in\Omega$ and $\gamma\in\Gamma_{\mu,\nu}$ is an optimal couplings, and the distance $d$ is over $L^2_W(\R^{d}\times\R^d)$. We then appeal to Corollary 3.9 of \cite{NochettoLi} which bounds this further by estimates $d^2(\hmu_n,\mu),d^2(\hnu_n,\nu)$ for which the Fournier-Guillin rates of Theorem \ref{thm:FG} apply.\newline

    Let $\disp_t:\R^d\times\R^d\to\R^d$ denote the interpolation $(x,y)\mapsto (1-t)x+ty$. From the convexity of the norm, we find for any $x,y,x',y'\in\R^d$,
    \begin{align}\label{eq:norm_convexity}
        \l\|\disp_t(x,y) - \disp_t(x',y')\r\|^2_{\R^d} &= \l\|(1-t)(x-x') + t(y-y')\r\|^2_{\R^d} \leq \l\|x-x'\r\|^2_{\R^d} + \l\|y-y'\r\|^2_{\R^d} \notag\\
        &= \l\|(x-x',y-y')\r\|^2_{\R^d\times\R^d}.
    \end{align}
    We can now use this to uniformly bound the Wasserstein distance between the empirical and population geodesics. Fix $\om\in\Omega$, and write $\gamma_{\om,n}$ for the optimal coupling between $(\hmuno, \hnuno)$. Viewing both $\gamma_{\om,n}$ and $\gamma\in\Gamma(\mu,\nu)$ (the optimal coupling for $(\mu,\nu)$) as probability measures over $\R^d\times\R^d$, we can consider coupling these too i.e. consider some $\Pi\in\Gamma_{\gamma_{\om,n},\gamma}$, where $\Pi$ (not yet optimal) is of the form $\Pi\big((\d x_n,\d y_n),(\d x,\d y)\big)$. 
    
    Recalling the definition of the displacement interpolation, since $\Pi$ couples the optimal couplings,
    \begin{equation*}
        W_2^2(\eta_{\om,t}^n,\eta_t) \leq \int \l\|\disp_t(x_n,y_n) - \disp_t(x,y)\r\|^2\Pi((\d x_n,\d y_n),(\d x,\d y)).
    \end{equation*}
    Using the inequality \eqref{eq:norm_convexity} above, we find
    \begin{equation*}
        W_2^2(\eta_{\om,t}^n,\eta_t) \leq\int \l\|(x_n-x,y_n-y)\r\|^2_{\R^d\times\R^d}\,\Pi(\d x_n,\d y_n,\d x,\d y) = W_2^2(\gamma_{\om,n},\gamma)
    \end{equation*}
    where we have assumed $\Pi\in\Gamma_{\gamma_{\om,n},\gamma}$ is optimal on the RHS, giving the Wasserstein distance between the couplings. Note that this bound is uniform in $t$ thus we can take a supremum on the LHS of the inequality above.\newline

    We can now employ theorem \ref{thm:Li_Nochetto} above. As we have already assumed $\gamma_{\om,n}$ is optimal, the bound \eqref{eq:Li_Nochetto_bound} reduces with $\tilde e_h = e_h$, and so the theorem gives
    \begin{align*}\label{eq:bound_nli}
        W_2^2(\gamma_{\om,n},\gamma) &\leq \l(2\lambda^{\frac{1}{2}}e_h^\frac{1}{2}\big(W_2(\mu,\nu) + e_h\big)^{\frac{1}{2}} + e_h\r)^2 = O(e_h)
    \end{align*}
    as $n\to\infty$ and $e_h\to0$. But $e_h = W_2(\hmuno,\mu)+W_2(\hnuno,\nu)$, which converges at the square root of the Fournier-Guillin rate of Theorem \ref{thm:FG} after applying Cauchy-Schwarz.
\end{proof}

\subsection{Proof of Theorem \ref{thm:WoWLLn} (Wasserstein over Wasserstein Law of Large Numbers)}\label{App:thm:WoWLLn}
\begin{proof}
    Varadarajan's theorem applies to the Polish space $\Pcpc$, which immediately delivers the almost sure weak convergence of the empirical measure to the population. The strong Law of Large Numbers applied to the i.i.d. samples $\cM_2(\mu_i)$ subject to $\E[\cM_2(\mu_i)] = \cE_2(\Q)$ delivers the convergence of the moments $\cE_2(\hQn) = \frac{1}{n}\sum_{i=1}^n \cM_2(\mu_i) \to \cE_2(\Q)$, $\PP$-a.s. Together, the equivalence of Theorem \ref{thm:weak_wass} gives $W_{W_2}(\hQn,\Q)\to0$ for $\PP$-a.s. $\om$.\newline

    A sequence of dominating functions lets us lift to $L^2$ convergence. Let $Z_n(\om)\deq W_2^2\l(\hQn,\Q\r)\geq0$ so that $Z_n\to0$ a.s. as above. Further, by the elementary inequality $(a+b)^2\leq 2a^2 + 2b^2$,
    $$Z_n(\om) = W_2^2(\hQn(\om), \Q) \leq \l(W_2\l(\hQn(\om),\delta_{{\delta_0}}\r) + W_2(\Q,\delta_{\delta_0})\r)^2\leq 2\cE_2\l(\hQn(\om)\r) + 2\cE_2(\Q)\eqqcolon G_n(\om).$$
    $G_n$ is clearly integrable: $\E[G_n] = 4\cE_2(\Q)$ is a constant, hence convergent, sequence. Generalised DCT applies to the sequence $Z_n$ with integrable domination $G_n$ and we conclude $\E[Z_n] = d_{{WoW}}^2\l(\hQn,
        \Q\r) \to 0 $.
\end{proof}

\subsection{Proof of Theorem \ref{thm:Hierachical_WoW_LLN} (Hierarchical Wasserstein over Wasserstein Law of Large Numbers)}\label{App:thm:Hierachical_WoW_LLN}
\begin{proof}
    From the standard Wasserstein Law of Large Numbers over $\R^d$ (Theorem \ref{thm:empirical_convergence}), the approximation $\hmun$ converges in $L^2$ to $\mu\in\Pc$, i.e.
    $d^2(\hmun,\mu) = \E\l[W_2^2(\hmun,\mu)\r]\to0$
    for each fixed $\mu$. We show that
    $\E_{\mu\sim\Q}[d^2(\hmun,\mu)]\to 0$ which follows by a domination argument. Indeed, $d^2\l(\hmun,\mu\r)\leq 2\E\l[\cM_2(\hmun)\r]+2\cM_2(\mu) = 4\cM_2(\mu)$, which is integrable in $\mu$ (as $\Q\in\cP_2(\Pc)$ is integrable). By dominated convergence, we conclude $\E_{\mu\sim\Q}[d^2(\hmun,\mu)]\to 0$.

    Consider the approximation of the empirical measure $\hQ_j$ with the latent empirical measure $\tilde\Q_j$.
    We can couple off each $\hmu_j$ with $\mu_j$ in a valid sub-optimal WoW coupling to find the inequality.
    \begin{equation}\label{eq:WoW_suboptimal_coupling}
        W_{W_2}^2 \l(\hQ_J(\om), \tilde\Q_J(\om)\r)\leq \frac{1}{J}\sum_{j=1}^J W_2^2\big(\hmu_{\om,j}),\mu_{\om,j}\big).
    \end{equation}
    Taking expectations on both sides and conditioning on the latent samples $\mu_j$, 
    \begin{align*}
        \E\l[W_{W_2}^2 (\hQ_J, \tilde\Q_J)\r]&\leq \frac{1}{J}\sum_{j=1}^J \E\l[W_2^2(\hmu_j,\mu_j)\r] = \frac{1}{J}\sum_{j=1}^J \E_{\mu_j\sim\Q}\l[\E_{X_{j,1:n_j}\sim{\mu_j}|\mu_j}\l[W_2^2(\hmu_j,\mu_j)\r]\r]\\
        &= \frac{1}{J}\sum_{j=1}^J \E_{\mu_j\sim\Q}\l[d^2(\hmu_j,\mu)\r]\to0
    \end{align*}
    given the above convergence of the expected $L^2_W$ distance, and the fact that the Cesàro average of a sequence converging to 0 does too.

    The Law of Large Numbers at the WoW level in Theorem \ref{thm:WoWLLn} itself provides $\E\l[W_{W_2}^2(\tilde \Q_J,\Q )\r]\to0$ i.e. the convergence of the latent empirical WoW measure to $\Q$ itself. Combining the estimates gives the conclusion.
        $\E\l[W_{W_2}^2(\Q,\hQ_J)\r]\leq 2\E\l[W_{W_2}^2 (\hQ_J, \tilde\Q_J)\r] + 2\E\l[W_{W_2}^2 (\tilde \Q_J, \Q)\r]\to0.$
    \end{proof}

\subsection{Proof of Corollary \ref{cor:compact_wow_upgrade} (Compact hierarchical convergence)}\label{App:cor:compact_wow_upgrade}
\begin{proof}
    The WoW Law of Large Numbers  in Theorem \ref{thm:WoWLLn} gives a.s. convergence of the latent empirical measure to the population $\Q$ i.e. $W_{W_2}(\tilde \Q_J(\om),\Q)\to0$ a.s. We show that the convergence of the hierarchical empirical measure to the latent $W_{W_2}(\tilde \Q_J(\om),\hQ_J)$ is also almost sure.\newline

    Let $\cK\subset \R^d$ have diameter $\cD$ and let
    $Z_j(\om)\deq W_2^2\big(\hmu_{\om,j},\mu_{\om,j}\big)\in \l[0,\cD^2\r]$. The $Z_j$ are all independent in $j$ from the i.i.d. sample assumptions, and from a standard variance bound for bounded random variables, $\Var(Z_j)\leq \frac{D^4}{4}$ is bounded uniformly in $j$.

    From the proof of Theorem \ref{thm:Hierachical_WoW_LLN} above the $\E_{\mu\sim\Q}[Z_j]\to0$ as $j\to\infty$ thus the Cesàro average converges to 0 too, $\frac{1}{J}\sum_{j=1}^{J}\E[Z_j] \to 0$.
    Given the variance bound, the sum $\sum_{j=1}^\infty \Var(Z_j)/j^2 \leq D^4\sum_{j=1}^\infty 1/j^{2}<\infty$. Thus by Kolmogorov's Strong Law of Large Numbers,
    $\frac{1}{J}\sum_{j=1}^J \l(Z_j - \E[Z_j]\r) \to 0 \quad a.s.$
    Adding the two sums gives $\frac{1}{J}\sum_{j=1}^J Z_j\to0$ a.s. The sub-optimal coupling which pairs off each $(\hmu_{\om,j},\mu_{\om,j})_{j=1}^J$ (one we used in \eqref{eq:WoW_suboptimal_coupling}) gives $$W_{W_2}^2\l(\hQ_J(\om), \tilde \Q_J(\om)\r)\leq \frac{1}{J}\sum_{j=1}^JW_2^2\big(\hmu_{\om,j},\mu_{\om,j}\big) = \frac{1}{J}\sum_{j=1}^JZ_j(\om)\to0\text{ a.s.}$$
    The triangle inequality delivers the almost sure result on the intersection of probability 1 events.
\end{proof}

\section{Bayesian analysis (Section \ref{sec:5})}

\subsection{Proof of Theorem \ref{thm:WL_BvM} ($\WL$ Bernstein-von Mises)}\label{App:thm:WL_BvM}

We first recall the classical Bernstein-von Mises theorem for completeness.
\begin{theorem}[\cite{vandervaart1998asymptotic} Chapter 10. Bernstein-von Mises Theorem]\label{thm:BvM}
    Suppose that the following conditions hold
    \begin{itemize}
        \item $\Theta\subset\R^d$ be open, the model $\{P_{\theta}:\theta\in\Theta\}$ is an identifiable, dominated model, and that $X_i$ are drawn i.i.d. from $P_{\theta_0}$ for some point $\theta_0$ in the interior of $\Theta$
        \item The model is locally asymptotically normal at $\theta_0$ with finite, nonsingular Fisher information $\cI(\theta_0)$
        \item The prior $\Prior$ has a density with respect to Lebesgue measure that is continuous and strictly positive at $\theta_0$
        \item $\forall\theta>0$, there exists a sequence of tests $\Phi_n$ s.t. $
            \E_{P_{\theta_0}^{\otimes n}}[\Phi_n] \to 0$, $ \sup_{\theta:\|\theta-\theta_0\|}\E_{P_{\theta_0}^{\otimes n}}[1-\Phi_n] \to 0.$
    \end{itemize}
    Then for any best-regular estimator sequence $\hat\theta_n$, the posterior $\Postn$ satisfies for $T(\theta) = \sqrt{n}\l(\theta-\hat\theta_n\r)$
    \begin{equation*}
        \l\|
            \l(T^{-1}\r)_\# \Postn\l(\cdot|X_{1:n}\r) - \cN_d\l(\cdot\,;0,\cI(\theta_0)^{-1}\r)
        \r\|_{\mathrm{TV}} \to 0
    \end{equation*}  
    in probability with respect to $\P_{\theta_0}^{\otimes n}$, where best-regular refers to an estimator that is asymptotically linear and efficient.
\end{theorem}
With this in hand, we provide the proof of Theorem \ref{thm:WL_BvM}.
\begin{proof}
    Consider the two moments
    \begin{equation*}
        m_{1,n} \deq \int_{\R^d} z \, \tilde{\Pi}_n(dz|X_{1:n}),\qquad m_{2,n} \deq \cM_2\l(\tilde{\Pi}_n\r) = \int_{\R^d} \|z\|^2 \, \tilde{\Pi}_n(dz|X_{1:n}).
    \end{equation*}
    We claim that $m_{1,n}\to0$ and $m_{2,n}\to \mathrm{tr}\l(\cI(\theta_0)^{-1}\r)$ in probability with respect to $P_{\theta_0}^{\otimes n}$. Indeed, we can use the supremum bounds to show that the tails are uniformly negligible in probability.
    \begin{align*}
        \int_{\|z\|>R}\|z\|\,\tilde{\Pi}_n(dz|X_{1:n})&\leq R^{-(1+\eta)}\int \|z\|^{2+\eta}\,\tilde{\Pi}_n(dz|X_{1:n})\\
        \int_{\|z\|>R}\|z\|^2 \,\tilde{\Pi}_n(dz|X_{1:n})&\leq R^{-\eta}\int \|z\|^{2+\eta}\,\tilde{\Pi}_n(dz|X_{1:n})
    \end{align*}
    both of which converge to $0$ uniformly in $n$ as $R\to\infty$. As convergence in total variation implies weak convergence, and the standard Bernstein-von Mises result 
    provides convergence of the posterior to a Gaussian $\cN\l(0,\cI(\theta_0)^{-1}\r)$,
    for $f\in C_b(\R^d)$,
    \begin{align*}
        \int f(z) \,\tilde{\Pi}_n(dz|X_{1:n}) \to \int f(z) \,\cN\l(dz;\, 0, \cI\l(\theta_0\r)^{-1}\r).
    \end{align*}
    Hence first letting $\ninf$
    \begin{align*}
        \int_{\R^d} z\1_{\{\|z\|\leq R\}} \tilde{\Pi}_n(dz|X_{1:n}) &\overset{\ninf}{\longrightarrow}\int_{\R^d} z\1_{\{\|z\|\leq R\}}\, \cN\l(dz;\, 0, \cI\l(\theta_0\r)^{-1}\r) = 0,\\
        \int_{\R^d} \|z\|^2\1_{\{\|z\|\leq R\}} \tilde{\Pi}_n(dz|X_{1:n}) &\overset{\ninf}{\longrightarrow} \int_{\R^d} \|z\|^2\1_{\{\|z\|\leq R\}}\, \cN\l(dz;\, 0, \cI\l(\theta_0\r)^{-1}\r),
    \end{align*}
    and taking then $R\to\infty$, the second moment term converges to the variance of the given Gaussian, i.e. the trace of the inverse Fisher information $\mathrm{tr}\l(\cI(\theta_0)^{-1}\r)$ as desired. Combining this with the fact that the tail estimates converge to $0$, we conclude the claim,
    $$m_{1,n}\to0, \qquad m_{2,n}\to \mathrm{tr}\l(\cI(\theta_0)^{-1}\r), \qquad \text{in $P_{\theta_0}^{\otimes n}$ probability as $\ninf$}.$$
    Now, let $z = \sqrt{n}\l(\theta - \theta_0\r)$ so that
    \begin{align*}
        A_n\l(\theta_0\r) &\deq \int_{\Theta}\l\|\sqrt{n}(\theta-\theta_0)\r\|^2\Postn(\d \theta|X_{1:n})= \int_{\R^d} \| z + V_n \|^2\tilde{\Pi}_n(dz) = m_{2,n} + 2V_n^Tm_{1,n} + \|V_n\|^2.
    \end{align*}
    Given the above claim, the first term converges weakly to $\mathrm{tr}\l(\cI(\theta_0)^{-1}\r)$, the second to constant $0$, and the third to $\|Z\|^2$, for $Z\sim\cN\l(0,\l(\cI(\theta_0)^{-1}\r)\r)$, so that together $A_n \wto \mathrm{tr}\l(\cI(\theta_0)^{-1}\r) + \|Z\|^2.$
    Finally, given that $A_n\leq 2m_{2,n} + 2\|V_n\|^2$, with Jensen's
    \begin{equation*}
        m_{2,n}^{1+\frac{\eta}{2}} = \l(\int \|z\|^2 \,\tilde{\Pi}_n(dz)\r)^{1+\frac{\eta}{2}} \leq \int \|z\|^{2+\eta}\,\tilde{\Pi}_n(dz),
    \end{equation*}
    the assumed uniform integrability gives uniform integrability of $A_n$, 
        $\sup_{n}\E_{P_{\theta_0}}\l[A_n^{1+\frac{\eta}{2}}\r] < \infty.$
    Thus we can pass the weak convergence to convergence of the expectation to conclude.
    \begin{equation*}
        \lim_{\ninf} nd^2(\Postn,\delta_{\theta_0}) = \lim_{\ninf}\E_{P_{\theta_0}}[A_n] = \mathrm{tr}\l(\cI(\theta_0)^{-1}\r) + \E_{P_{\theta_0}}[\|Z\|^2] = 2\mathrm{tr}\l(\cI(\theta_0)^{-1}\r),
    \end{equation*}
\end{proof}

\subsection{Proof of Theorem \ref{thm:Posterior_consistency_wasserstein} (Wasserstein consistency)}\label{App:thm:Posterior_consistency_wasserstein}
Before providing the proof, the following inequality proves useful.
\begin{lemma}\label{lemma:projection_inequality}
  For every $\mu,\nu\in\Pc$, $R>0$,
  \begin{equation*}\label{eq:QTT_inequality}
      W_2(\mu,\nu)\leq T_R(\mu)^{\frac{1}{2}} + W_2\l({p_R}_{\#}\mu,{p_R}_{\#}\nu\r) + T_R(\nu)^{\frac{1}{2}},
  \end{equation*}
  where $p_R(x) = x\1_{\{\|x\|\leq R\}} + R\frac{x}{\|x\|}\1_{\{\|x\|> R\}}$ is the orthogonal projection onto $\bar B(0,R)\subset\R^d$.
\end{lemma}
This follows from the $W_2$ triangle inequality and using $T_R$ itself as a transport map between the measures $\mu,\nu$ and their respective projections. We can now complete the proof of Theorem \ref{thm:Posterior_consistency_wasserstein}.
\begin{proof}
    Let $R$ be sufficiently large so that $T_R(\pi)^{\frac{1}{2}}<\varepsilon_1$ for some $\varepsilon_1>0$ to be determined. Since $p_R$ is bounded continuous, the map $\mu\mapsto W_2({p_R}_{\#}\mu,{p_R}_{\#}\nu)$ is weakly continuous. Indeed, the pushforward is continuous in the weak topology and the projected measures are supported on the compact ball $\bar B_0(R)$, where $W_2$ and weak convergence coincide, as in Corollary \ref{cor:compact_equivalence}. Thus
    \begin{equation*}
        U_{R,\varepsilon_2}\deq \l\{\mu\in\Pc:W_2({p_R}_{\#}\mu,{p_R}_{\#}\pi)<\varepsilon_2\r\},
    \end{equation*}
    is a weak neighbourhood of $\pi$, where $\varepsilon_2>0$ is to be determined. By Lemma \ref{lemma:projection_inequality}, any $\mu\in U_{R,\varepsilon_2}$ with $T_R(\mu)\leq {\varepsilon_3}^2$ for some $\varepsilon_3>0$ satisfies
    $W_2(\mu,\pi)\leq {\varepsilon_1} + {\varepsilon_2} + \varepsilon_3 = \varepsilon,$
    after letting $\varepsilon_1=\varepsilon_2=\varepsilon_3=\frac{\varepsilon}{3}$.
    Thus, $$\{\mu:W_2(\mu,\pi)>\varepsilon\} \subset U^c_{R,\frac{\varepsilon}{3}}\,\cup \,\l\{\mu:T_R(\mu)>\frac{\varepsilon}{3}\r\},$$
    so that a union bound gives
    \begin{equation*}
        \Postn\bigg(\big\{\mu:W_2(\mu,\pi)>\varepsilon\big\}\bigg) \leq \Postn\l(U^c_{R,\frac{\varepsilon}{3}}\r) + \Postn\bigg(\l\{\mu:T_R(\mu)>\frac{\varepsilon}{3}\r\}\bigg).
    \end{equation*}
    Taking $\limsup_{\ninf}$, the first term vanishes from the corollary to Schwartz's theorem applied to the weak neighbourhood $U_{R,\frac{\varepsilon}{3}}$. Letting then $R\to\infty$, the second term vanishes with the quadratic tail bound. This concludes the Wasserstein consistency assertion \eqref{eq:consistency_in_prob}.
\end{proof}

\subsection{Proof of Theorem \ref{thm:Hierarchical_schwartz} (Hierarchical Schwartz consistency)}\label{App:thm:Hierarchical_schwartz}

Targeting consistency within the hierarchical regime, we seek to provide an approximation to the latent measures using the datapoints, and then use something akin to Schwartz's Theorem \ref{thm:Schwartz} to conclude. The first step is to find uniform exponential tests which we can use to bound the probabilities of Type I and Type II errors respectively. The tests we construct in equation \eqref{eq:hierarchical_tests} below  are based not on the latent measures $\mu_1,\dots,\mu_J\sim \P^*$, but their hierarchical approximations $\hmu_j$ which are formed from the observable $X_{j,i}$.

\begin{proposition}[Uniform exponential tests for the hierarchical Bayes]\label{prop:hierarchical_tests}
    Let $\cK\subset\R^d$ be compact, and let $P^*\in\Pkpk$ be a population measure for the hierarchical sample as above. Fix a basic weak neighbourhood $U\ni\P^*$ for some $\delta>0$:
\begin{equation*}
    U = \bigcap_{i=1}^k\l\{\P:\l|\int f_i \d \P - \int f_i \d \P^*\r|<\delta\r\}
\end{equation*}
for some $k\in\bN$,
where $f_i\in\Lip_1(\cP_2(\cK), W_2)$ subject to $\|f_i\|_\infty\leq1$ - these provide a base set for the weak topology on the compact Polish space $\l(\cP_2(\cK), W_2\r)$. Let
$$A_i \deq \l\{\P:\l|\int f_i \d \P - \int f_i \d \P^*\r|\geq\delta\r\},\quad \text{so that}\quad U^c = \bigcup_{i}A_i.$$ Let 
\begin{equation}\label{eq:hierarchical_tests}
    \Phi_{J,i}\deq \1\l\{\l|\frac{1}{J}\sum_{i=1}^J f_i(\hmu_j) - \int f_i\d \P^*\r|>\frac{\delta}{2}\r\},
\end{equation}
where $\hmu_j = \frac{1}{n_j}\sum_{i=1}^{n_j}\delta_{X_{j,i}}$. Let
$\Phi_J\deq \max_{i}\Phi_{J,i}$. Then the $\Phi_J$ provide a set of uniform exponential tests as required for Schwartz's theorem when $J\geq J^*$ for some sufficiently large $J^*$.
\begin{equation}
    \E_{P^*}[\Phi_J]\leq 2k\exp\l(-\frac{J\delta^2}{32}\r),\qquad \sup_{\P\in A_i}\E_{\P}[1-\Phi_{J,i}]\leq 2\exp\l(-\frac{J\delta^2}{32}\r)
\end{equation}
\end{proposition}
The proof follows after splitting the test into an approximation term and a bias-type term, and then using Hoeffding's concentration inequality to control the probability of either error.

\begin{proof}
Considering the tests $\Phi_{J}$, we split the following distance
\begin{equation*}
    T_J^\P \deq \frac{1}{J}\sum_{i=1}^J f_i(\hmu_j) - \int f_i\d \P^* = \frac{1}{J}\sum_{i=1}^J Z_j - \int f\d \P^*  \coloneqq S_J^\P + B_J^{\P},
\end{equation*}
\begin{equation*}
    S_J^\P \deq \frac{1}{J}\sum_{j=1}^J\l(Z_j - \E_{\P}[Z_j]\r),\qquad B_J^{\P} \deq \frac{1}{J}\sum_{j=1}^J\E_\P[Z_j] - \int f\d \P^*.
\end{equation*}
To eventually access Hoeffding's inequality on $S_J^\P$, we split into an empirical approximator, and a bias-type term i.e.
\begin{equation*}\label{eq:empirical_approximator_split}
    \underbrace{\frac{1}{J}\sum_{i=1}^J f_i(\hmu_j) - \frac{1}{J}\sum_{i=1}^J \E[f_i(\hmu_j)]}_{\text{Empirical approximator}} + \underbrace{\frac{1}{J}\sum_{i=1}^J \E[f_i(\hmu_j)]- \int f_i\d \P}_{\text{bias-type term}}.
\end{equation*}

Fix $i$, and write $f=f_i$ for notational convenience. Let $Z_j \deq f(\hmu_j)\in[-1,1]$ be random variables, independent across $j$ under the hierarchical sampling scheme for any given $\P\in\Pkpk$. As $f$ is 1-Lipschitz over $W_2$,
\begin{equation*}
    \l|\E_{\P}\l[Z_j\r] - \int f\d \P\r| = 
    \l|\E_{\mu\sim\P}\l[
        \E_{\mu^{\otimes n_j}}\l[
            f(\hmu_j) - f(\mu)
        \r]
    \r]\r| 
    \leq 
    \E_{\mu\sim\P}\l[
        \E\l[
            W_2(\hmu_j,\mu)
        \r]
    \r] \leq  b(n_j),
\end{equation*}
given the uniform bound $b(n_j)\deq \sup_{\mu\in\cP_2(\cK)}d(\hmu_{n_j},\mu)$ on the $L^2$ distance over compact $\cK$. As $b(n_j)\to0$, we can take $J^*(\delta)$ sufficiently large so that the Cesàro average $\frac{1}{J}\sum_{j=1}^{J} b(n_j) < \frac{\delta}{4}$ for all $J\geq J_*$, hence for these $J$,
\begin{equation*}
    \l|B_J^\P\r| = 
    \l|
        \frac{1}{J}\sum_{j=1}^J \E_{\P}[Z_j] - \int f \d \P
    \r| \leq \frac{\delta}{4},\qquad \text{for all $\P\in \Pkpk$.}
\end{equation*}

We now bound $S_J^\P$, the quality of the empirical approximation. As $\E[S_J]=0$, and $|S_J|\leq2$, we can apply Hoeffding's inequality,
\begin{equation*}
    \P\l(
        \l|S_J^\P\r| \geq \frac{\delta}{4}
    \r)
    \leq 2 \exp\l(
            -\frac{2(J\delta/4)^2}{4J}
        \r) = 2 e^{-\frac{J\delta^2}{32}}.
\end{equation*}
We can use this to bound the probability of both Type I and Type II errors. Beginning with the Type I, consider the event $\{\Phi_{J,i} = 1\}$ on which $ \l\{\l|T_J^{\P^*}\r|\geq \frac{\delta}{2}\r\}$. Given that $J\geq J^*$ is sufficiently large, $\l|B_J^{\P^*}\r|<\frac{\delta}{4}$.  Together, these imply $\l|S_J^\P\r| \geq \frac{\delta}{4}$, i.e. for $J\geq J^*$, $\l\{\l|T_J^{\P^*}\r|\geq \frac{\delta}{2}\r\} \subset \l\{\l|S_J^\P\r| \geq \frac{\delta}{4}\r\}$ thus we find the Type I error bound
\begin{align*}
    \E_{\P^*}[\Phi_{J,i}]&\leq 2\exp\l(-\frac{J\delta^2}{32}\r),\qquad
    \E_{\P^*}[\Phi_{J}]\leq \sum_{i=1}^k \E_{\P^*}[\Phi_{J,i}] \leq 2k\exp\l(-\frac{J\delta^2}{32}\r).
\end{align*}
For the Type II bound, let $\P\in A_i$.
Then \begin{equation*}
    \l|T_J^\P\r|\geq \l|\int f_i\d \P - \int f_i \d \P^*\r| - \l|B_J^\P\r| - \l|S_J^\P\r| \geq \delta - \frac{\delta}{4} - \l|S_J^\P\r|.
\end{equation*}
On $\{\Phi_{J,i} = 0\}$, $\l|T_J^\P\r|> \frac{\delta}{2}$ whence $\l|S_J^\P\r|<\frac{\delta}{4}$. Thus
\begin{equation*}
    \E_{\P}[1-\Phi_{J,i}] = \P\l(\Phi_{J,i} = 0\r) \leq \P\l(\l|S_J^\P\r|>\frac{\delta}{4}\r).
\end{equation*}
The same Hoeffding-type bound applies as before, leaving
\begin{equation*}
    \sup_{\P\in A_i}\E_{\P}[1-\Phi_{J,i}] \leq 2\exp\l(-\frac{J\delta^2}{32}\r).
\end{equation*}
Thus we conclude the existence of the testing conditions.
\end{proof}

With this in hand, we also collect a useful pair of lemmas concerning the likelihoods below.
\begin{lemma}\label{lemma:likelihood_consequences}
    Let $\P^*\ll \P$ with $g = \frac{\d \P^*}{\d \P}$. Then the likelihood $m_{\P^*}^{(j)}\ll m_{\P}^{(j)}$ with
    \begin{equation*}
        r_j(z) \deq \frac{dm_{\P^*}^{(j)}}{dm_{\P}^{(j)}}(z) = \E_{\P}\l[g(\mu) \mid Z_j = z\r], \quad (\mu,Z_j)\sim \P\otimes \mu^{\otimes n_j}, \quad\text{and}\quad \int \varphi(r_j) dm_{\P}^{(j)} \leq \int \varphi(g)\d \P.
    \end{equation*}
    for every convex $\varphi:[0,\infty)\to\R$ bounded below.
    In particular, 
    \begin{itemize}
        \item $\varphi(x)=x\log(x)$ gives $\KL\l(m_{\P^*}^{(j)}\| m_{\P}^{(j)}\r)\leq \KL(\P^*\|\P)$.
        \item Let $\psi(x)=x(\log x)^2$, and define 
        \begin{equation*}
            \tilde\psi(x) = \begin{cases} 
            \psi(x) & \text{on } [e^{-1}, \infty) \\ 
            2e^{-1} - x & \text{on } [0, e^{-1}). 
            \end{cases}
        \end{equation*} 
        Then $\tilde\psi$ is convex and $\psi\le\tilde\psi\le\psi+2/e$. Hence
        \begin{equation*}
        \E_*\big[(\log r_j)^2\big]=\int\psi(r_j)\,dm_P^{(j)}\le\int\tilde\psi(g)\,dP\le V(P_*\|P)+\tfrac2e.
        \end{equation*}
        \item With the log-likelihood ratio $l_j(\P)\deq \log \frac{dm_{\P}^{(j)}}{dm_{\P^*}^{(j)}}(z) = -\log(r_j(Z_j))$, we have $$\E_{\P^*}[l_j(\P)] = -\KL(m_{\P^*}^{(j)}\|m_{\P^{(j)}})\in \l[-\KL(\P^*\|\P),0\r],\qquad \E_{\P^*}\l[|l_j(\P)|\r]\leq \KL(\P^*\|\P).$$
    \end{itemize}
    
\end{lemma}

Under these conditions, we can conclude the Schwartz consistency by the argument outlined above. First, a preparatory lemma bounding the rate of the likelihood ratio decay to $0$, at a strictly slower exponential rate than the testing conditions.
\begin{lemma}[Likelihood ratio for Schwartz testing conditions]\label{lemma:schwartz_likelihood}
    Fix $\delta>0$ and $M$ with $h\deq \cH(B_{\delta,M})>0$, and let 
    \begin{equation*}
        D_J \deq \int 
            \prod_{j=1}^J \frac{dm_{\P}^{(j)}}{dm_{\P^*}^{(j)}}(Z_j)
        \cH(\d \P).
    \end{equation*}
    Then $\P^*$-a.s. there is $J_{\om,1}$ with $D_J\geq he^{-2\delta J}$ for all $J\geq J_{\om,1}$ sufficiently large.
\end{lemma}
\begin{proof}
    Write $B = B_{\delta,M}$, $\Pi_B = \cH(\cdot\cap B)/h$ for notational convenience. We can bound the term $D_J$ as follows using Bayes' rule and then the convexity of the exponential function to use Jensen's inequality,
    \begin{align*}
        D_j &\geq \int_B \exp\l(\sum_{j=1}^Jl_j(\P)\r)\cH(\d \P) \geq h\int \exp\l(\sum_{j=1}^Jl_j(\P)\r)\Pi_B(\d \P) \notag \\
        &\geq h\exp\l(\int_B \sum_{j=1}^Jl_j(\P)\Pi_B(\d \P)\r) = \cH(B)\exp\l(\sum_{i=1}^J\xi_j\r) , \qquad \xi_j \deq \int l_j(\P)\Pi_B(\d \P).\label{eq:L_j_bound}
    \end{align*}
    Given that the $\xi_j$ are functions of the independent samples $Z_j$, the $\xi_j$ are independent. Employing Fubini's theorem and Lemma \ref{lemma:likelihood_consequences} in combination gives
    \begin{align*}
        \E_{\P^*}\l[|\xi_j|\r] \leq \delta + 2,\quad \E_{\P^*}\l[\xi_j\r] \geq -\delta,\qquad
        \Var\xi_j \leq \E_{\P^*}\l[\xi_j^2\r] \leq \int\E_{\P^*}\l[l_j^2\r]\d \Pi_b \leq M + \frac{2}{e}.
    \end{align*}
    Hence $\sum_{j=1}^\infty \frac{\Var \xi_j}{j^2}<\infty$, and Kolmogorov's SLLN applies.
    \begin{equation*}
        \frac{1}{J}\sum_{j=1}^J \l(\xi_j - \E[\xi_j]\r) \to 0 \quad \text{a.s.} \implies \liminf_{J\to\infty} \frac{1}{J}\sum_{j=1}^J \xi_j \geq -\delta,
    \end{equation*}
    whence there is (for $\om$ in a full-probability set of $\P$) a sufficiently large $J_{\om,1}$ such that $\sum_{j=1}^J \xi_j \geq -2\delta J$ for all $J\geq J_{\om,1}$.
\end{proof}

Admitting the lemma, we may now complete the proof of Theorem \ref{thm:Hierarchical_schwartz}.

\begin{proof}
    
Fix a basic weak neighbourhood $U$ of $\P^*$ as in Proposition \ref{prop:hierarchical_tests}, i.e.
\begin{equation*}
    U = \bigcap_{i=1}^k\l\{\P:\l|\int f_i \d \P - \int f_i \d \P^*\r|<\delta\r\},
\end{equation*}
so that 
\begin{equation*}
    U^c = \bigcup_{i}A_i,\qquad A_i = \l\{\P:\l|\int f_i \d \P - \int f_i \d \P^*\r|\geq\delta\r\}.
\end{equation*}
Let $J>J_{\om,1}$ for $\om$ in a set of full $\P^*$ measure, as in Lemma \ref{lemma:schwartz_likelihood}. We can employ the testing conditions from Proposition \ref{prop:hierarchical_tests}, just as in the classical Schwartz's theorem.
\begin{align*}
    \cH_J(A_i|Z_{1:j}) &= \dfrac{\displaystyle{\int_{A_i} \prod_{j=1}^J \frac{dm_{\P}^{(j)}}{dm_{\P^*}^{(j)}}(Z_j)\cH(\d \P)}}{D_j} = \,\Phi_{J,i}\,+\,\frac{\displaystyle{\int_{A_i} (1-\Phi_{J,i})\prod_{j=1}^J \frac{dm_{\P}^{(j)}}{dm_{\P^*}^{(j)}}(Z_j)\cH(\d \P)}}{D_j}\\
    &\leq \Phi_{J,i}+\frac{\displaystyle{\int_{A_i} (1-\Phi_{J,i})\prod_{j=1}^J \frac{dm_{\P}^{(j)}}{dm_{\P^*}^{(j)}}(Z_j)\cH(\d \P)}}{h e^{-2\varepsilon J}}
\end{align*}
The first term is the Type I testing condition, which we know to have an exponential bound after taking expectations with respect to $\P^*$. Taking the $\P^*$ expectation of the second term numerator,
\begin{align*}
    \E_{\P^*}\l[\int_{A_i} (1-\Phi_{J,i})\prod_{j=1}^J \frac{dm_{\P}^{(j)}}{dm_{\P^*}^{(j)}}(Z_j)\,\cH(\d \P)\r] &= \int_{A_i} \E_{\P^*}\l[(1-\Phi_{J,i})\prod_{j=1}^J \frac{dm_{\P}^{(j)}}{dm_{\P^*}^{(j)}}(Z_j)\r]\cH(\d \P)\\
    &= \int_{A_i} \E_{\P}\l[1-\Phi_{J,i}\r]\cH(\d \P)\leq 2\cH(A_i) e^{-\frac{J\delta^2}{32}}
\end{align*}
where the first line uses Fubini's theorem, the second the change of measure identity, and the third the uniform bound for the Type II testing condition. Combining the above estimates, we find that for sufficiently large $J>J_{\om,1}$,
\begin{equation*}
    \E_{\P^*}\l[\cH(A_i|Z_{1:J})\r] \leq 2e^{-\frac{J\delta^2}{32}} + \frac{2\cH(A_i)e^{-\frac{J\delta^2}{128}}}{h e^{-2J\varepsilon}}, \leq 2e^{-\frac{J\delta^2}{64}} + \frac{2\cH(A_i)e^{-\frac{J\delta^2}{128}}}{h}
\end{equation*}
where we have taken $\varepsilon=\frac{\delta^2}{256}<\frac{\delta^2}{128}$, so that the term on the RHS becomes summable in $J$. Recall that $U^c = \bigcup_{i}A_i$ so that by a union bound,
\begin{equation*}
    \E_{\P^*}\l[\cH(U^c|Z_{1:J})\r] \leq \sum_{i=1}^k 2e^{-\frac{J\delta^2}{64}} + \frac{2\cH(A_i)e^{-\frac{J\delta^2}{128}}}{h }\leq 2ke^{-\frac{J\delta^2}{64}} + \frac{2e^{-\frac{J\delta^2}{128}}}{h }
\end{equation*}
which is still summable in $J$. So we find
\begin{equation*}
    \sum_{J=1}^\infty \E_{\P^*}\l[\cH(U^c|Z_{1:J})\r] < \infty \implies \E_{\P^*}\l[\sum_{J=1}^\infty \cH\l(U^c|Z_{1:J}\r)\r] < \infty. 
\end{equation*}
by Fubini's theorem. So the term inside the expectation is finite almost surely and we conclude the strong consistency result by the first Borel-Cantelli lemma,
    $\cH(U^c|Z_{1:J})\to0,$ {$\P^*$-a.s.}
\end{proof}

\subsection{Dirichlet process posterior example (Proposition \ref{prop:total_variance})}\label{App:Dirichlet_example}
By conjugacy, the posterior $\Postn\in\WWL$ is also a Dirichlet process as below, with posterior mean $\alpha_n$, 
\begin{equation*}
    \Pi_n:\om \longmapsto \Pi_{\om,n} = \mathrm{DP}\l(a+n,\;\alpha_{\om,n}\r),\qquad \alpha_n = \frac{a\alpha_0 + n\hmun}{a+n}
\end{equation*}
We verify conditions (i) and (ii) from Proposition \ref{prop:total_variance}.
\begin{enumerate}[label=(\roman*)]
    \item (Posterior spread) $\kappa_n'=a+n+1$: Under $\Pi_{\om,n}=\mathrm{DP}\l(a+n,\alpha_{\om,n}\r)$, the marginal law $\rho(A)\sim \mathrm{Beta}\l((a+n)\alpha_{\om,n}(A),\,(a+n)(1-\alpha_{\om,n}(A))\r)$. Its variance satisfies$$\mathrm{Var}_{\Pi_{\om,n}}\l(\rho(A)\r) = \frac{\alpha_{\om,n}(A)\l(1-\alpha_{\om,n}(A)\r)}{a+n+1} \leq \frac{\alpha_{\om,n}(A)}{a+n+1}.$$ 
    \item (Data fluctuation) $\kappa_n''=a+n$: Since $n\hmuno(A)\sim\mathrm{Bin}(n,\pi(A))$ has variance $n\pi(A)(1-\pi(A))$,$$\mathrm{Var}_\om\l(\alpha_{\om,n}(A)\r) = \left(\frac{n}{a+n}\right)^2\mathrm{Var}_\om\l(\hmuno(A)\r) = \frac{n\,\pi(A)\l(1-\pi(A)\r)}{(a+n)^2} \leq \frac{n\,\pi(A)}{(a+n)^2} \leq \frac{\bar\alpha_n(A)}{a+n},$$
where the final inequality uses $\bar\alpha_n = \E_\om[\alpha_n(\omega)] =\frac{a\alpha_0+n\pi}{a+n}$.
\end{enumerate}

\subsection{Proof of Theorem \ref{thm:pcr} (Posterior contraction rates)}\label{App:thm:pcr}
\begin{proof}
As $\E_\om[\hmuno] = \pi$, $\bar\alpha_n=\lambda_n\alpha_0+(1-\lambda_n)\pi$, and hence $\cM_q(\bar\alpha_n) < \max\{\cM_q(\alpha_0),\cM_q(\pi)\}\eqcolon K$. This implies $\cM_2(\bar\alpha_n)<\infty$, so that indeed $\bar\alpha_n\in\cP_2(\R^d)$.
Since the only coupling with a Dirac mass is the product coupling,
$$\dw (\Pi_n,\delta_\pi)^2 = \E_\om\l[\int W_2^2(\rho,\pi)\,\Pi_{\omega,n}(\d \rho)\r].$$
Since $\rho\mapsto W_2^2(\rho,\pi)$ is non-negative, Tonelli's theorem applies to $\omega\mapsto\Pi_{\omega,n}$,
$$\E_\om\l[\int W_2^2(\rho,\pi)\,\Pi_{\om,n}(\d \rho)\r]=\int W_2^2(\rho,\pi)\,\Q_n(\d \rho).$$
The triangle inequality gives the pathwise bound $W_2(\rho,\pi)\leq W_2(\rho,\bar\alpha_n)+W_2(\bar\alpha_n,\pi)$. Applying Minkowski's inequality in combination with $(x+y)^2\le2(x^2+y^2)$,
$$\dw (\Pi_n,\delta_\pi)^2\leq 2\int W_2^2(\rho,\bar\alpha_n)\,\Q_n(\d \rho)+2W_2^2(\bar\alpha_n,\pi).$$
$\Pi_n$ is assumed to have vanishing fluctuation, whence $\Q_n$ is $\kappa_n$-concentrated with intensity $m_{\Q_n}=\bar\alpha_n$. Theorem \ref{thm:Generalied_FG} gives the bound on the first (fluctuation) term, 
$$\int W_2^2(\rho,\bar\alpha_n)\,\Q_n(\d \rho)\leq C(d,q)\cM_q(\bar\alpha_n)^{2/q}\,r(\kappa_n, d,q)\leq C(d,q)K^{2/q}r(\kappa_n,d,q).$$
For the second (bias-type) term, write $\pi=\lambda_n\pi+(1-\lambda_n)\pi$. Joint convexity of $W_2^2$ then gives
$$W_2^2(\bar\alpha_n,\pi)\leq\lambda_nW_2^2(\alpha_0,\pi)+(1-\lambda_n)W_2^2(\pi,\pi)=\lambda_nW_2^2(\alpha_0,\pi).$$
Taking $C=2\max\l\{W_2^2(\alpha_0,\pi),\ C(d,q)\,K^{2/q}\r\}$ gives the above claim.

\end{proof}

\subsection{Proof of Proposition \ref{prop:enlargement} (Enlargement preserves distance)}\label{App:prop:enlargement}

\begin{proof}
    Let $p(\cdot | x) \deq \mathrm{Law}_{\mathrm{Bayes}}\big(\theta_{1:n}\mid X_{1:n}=x\big)$. 
    Disintegrating the joint posterior of $(\rho,\theta_{1:n})$ given $X_{1:n}=x$,
    $$\mathrm{Law}(\rho,\theta_{1:n}\mid X_{1:n}=x)(\d\rho,\d\theta)= \Pi^{\Theta}_{\om,n}(\d\rho)\;p(\d \theta\mid x),\qquad \mathrm{Law}(\rho\mid X_{1:n}=x)=\int \Pi^{\Theta}_{\om,n}\;p(\d \theta\mid x)$$
    after marginalising out $\theta$.
    Applying $\Phi_\#$ gives
    $$\Pi_n^x=\Phi_\#\,\mathrm{Law}(\rho\mid X_{1:n}=x)=\int\Phi_\#\Pi^\Theta_{\om,n}\,p(\d \theta\mid x)=\int\Pi_{\om,n}\,p(\d \theta\mid x).$$
    Finally, 
    \begin{align*}
        W_{W_2}^2(\Pi_n^{x},\delta_\pi) &= \int W_2^2(\rho,\pi)\,\Pi_n^{x}(\d\rho) = \int \l( \int W_2^2(\rho,\pi) \Pi_{\om,n}(\d \rho)\r) p(\d \theta\mid x) \\&= \E_\om [W_{W_2}^2(\Pi_{\om,n},\delta_\pi) | X_{1:n} = x].
    \end{align*}
    Taking $\E_x$ on both sides and applying the tower property yields the result. 
\end{proof}

\subsection{Proof of Lemma \ref{lemma:wasserstein_lipschitz} ($W_2$-Lipschitz kernels)}\label{App:lemma:wasserstein_lipschitz}
\begin{proof}
    Let $\gamma$ be an optimal coupling of $\nu$ and $\nu'$. For each $(\theta, \theta')$ let $\varrho_{\theta, \theta'}$ be an optimal coupling of $(f(\cdot\mid\theta),f(\cdot\mid\theta'))$ (chosen measurably by Theorem \ref{thm:Measurability}). Define $\bar \gamma \deq \int \varrho_{\theta, \theta'} \gamma(\d \theta, \d \theta')$, which we claim is a valid coupling for $(\Phi(\nu), \Phi(\nu'))$. Indeed, for Borel $A \subset \R^d$, 
    $$\bar\gamma(A\times\R^d)=\int\varrho_{\theta,\theta'}(A\times\R^d)\,\gamma(\d \theta,\d \theta')=\int f(A\mid \theta)\,\gamma(\d \theta,\d \theta'),$$
    where the second equality uses that $\varrho_{\theta,\theta'}$ has first marginal $f(\cdot\mid\theta)$. The integrand is now dependent on $\theta$, so we can integrate out $\theta'$ using the fact that the first marginal of $\gamma$ is $\nu$.
    $$\int f(A\mid\theta)\,\gamma(\d \theta,\d \theta')=\int f(A\mid\theta)\,\nu(\d \theta)=\Phi(\nu)(A).$$
    Likewise $\bar\gamma(\R^d\times B)=\Phi(\nu')(B)$, so $\bar\gamma$ couples $(\Phi(\nu),\Phi(\nu'))$.
    By Tonelli and the Lipschitz assumption, 
    $$
    \int\|y-y'\|^2\,\bar\gamma(\d y,\d y')=\int W_2^2\l(f(\cdot\mid\theta),f(\cdot\mid\theta')\r)\,\gamma(\d \theta,\d \theta')\leq L^2\int\|\theta-\theta'\|^2\,\gamma(\d \theta,\d \theta')=L^2W_2^2(\nu,\nu').
    $$
    The claim follows from $W_2^2(\Phi(\nu),\Phi(\nu')) \leq \int\|y-y'\|^2\,\bar\gamma(\d y,\d y')$.
\end{proof}

\subsection{Proof of Theorem \ref{thm:pcr_mixture} (Posterior contraction rates for mixtures)}\label{App:thm:pcr_mixture}

\begin{proof}
As $\beta_n = \E_\om[\nu_n]$, we have $\E_\om[\alpha^{\Theta}_{\om,n}] = \lambda_n\alpha_0 + (1-\lambda_n)\beta_n$. By Tonelli's theorem,
$$
\E_\om\big[\cM_q(\alpha^{\Theta}_{\om,n})\big] = \lambda_n\cM_q(\alpha_0) + (1-\lambda_n)\cM_q(\beta_n) \leq \max\big\{\cM_q(\alpha_0),\, \sup_n\cM_q(\beta_n)\big\} \eqcolon K.
$$
Since $\nu_n$ is finitely supported, $\cM_q(\alpha^{\Theta}_{\om,n}) < \infty$, which implies $\cM_2(\alpha^{\Theta}_{\om,n}) < \infty$. By linearity of $\Phi$, the conditional observation intensity satisfies
$$
\alpha_{\om,n} \deq \Phi(\alpha^{\Theta}_{\om,n}) = \lambda_n\Phi(\alpha_0) + (1-\lambda_n)\Phi(\nu_n).
$$
Using $\alpha_0 \in \cP_2(\Theta)$, Lemma \ref{lemma:wasserstein_lipschitz} delivers $W_2(\alpha_{\om,n},\Phi(\alpha_0)) \leq L\,W_2(\alpha^{\Theta}_{\om,n},\alpha_0) < \infty$, thus $\alpha_{\om,n} \in \Pc$.

Using Proposition \ref{prop:enlargement} and the pushforward $\Pi_{\om,n} = \Phi_\#\Pi^{\Theta}_{\om,n}$, we write
$$
\dw(\Pi_n,\delta_\pi)^2 = \E_\om\int W_2^2\l(\Phi(\rho),\pi\r)\,\Pi^{\Theta}_{\om,n}(\d \rho).
$$
Exploiting the same chain of inequalities from Appendix \ref{App:thm:pcr} conditionally on $\om$ (splitting at $\alpha_{\om,n}$) yields
$$
\dw(\Pi_n,\delta_\pi)^2 \le 2\,\E_\om\int W_2^2\l(\Phi(\rho),\Phi(\alpha^{\Theta}_{\om,n})\r)\,\Pi^{\Theta}_{\om,n}(\d \rho) + 2\,\E_\om\big[W_2^2(\alpha_{\om,n},\pi)\big].
$$
Since $\Pi^{\Theta}_{\om,n}$ is $\kappa_n$-concentrated $\PP$-a.s. with intensity $\alpha^{\Theta}_{\om,n}$, applying Lemma \ref{lemma:wasserstein_lipschitz} and Theorem \ref{thm:Generalied_FG} conditionally on $\om$, followed by Jensen's inequality (as $t \mapsto t^{2/q}$ is concave for $q>2$), bounds the first (fluctuation) term:
\begin{align*}
    \E_\om\int W_2^2\l(\Phi(\rho),\Phi(\alpha^{\Theta}_{\om,n})\r)\,\Pi^{\Theta}_{\om,n}(\d \rho) &\leq L^2\,\E_\om\int W_2^2\l(\rho,\alpha^{\Theta}_{\om,n}\r)\,\Pi^{\Theta}_{\om,n}(\d \rho)\\
    &\leq L^2\,C(d_\Theta, q)\,r(\kappa_n, d_\Theta, q)\,\E_\om\big[\cM_q(\alpha^{\Theta}_{\om,n})^{2/q}\big]\\
    &\leq L^2\,C(d_\Theta, q)\,K^{2/q}\,r(\kappa_n, d_\Theta, q).
\end{align*}
For the second (bias) term, writing $\pi = \lambda_n\pi + (1-\lambda_n)\pi$ and using joint convexity of $W_2^2$ pointwise in $\om$ gives
$$
W_2^2(\alpha_{\om,n},\pi) \leq \lambda_n W_2^2\l(\Phi(\alpha_0),\pi\r) + (1-\lambda_n)W_2^2\l(\Phi(\nu_n),\pi\r) \leq \lambda_n W_2^2\l(\Phi(\alpha_0),\pi\r) + W_2^2\l(\Phi(\nu_n),\pi\r).
$$
Taking $\E_\om$ and setting $C = \max\l\{2L^2 C(d_\Theta, q)K^{2/q},\, 2W_2^2(\Phi(\alpha_0),\pi)\r\}$ ultimately gives the desired claim.
\end{proof}

\section{Applications (Section \ref{sec:6})}\label{App:Random_generative_models}

\subsection{Experiments for Flowing datasets (Section \ref{subsec:6:flowing_datasets}) }\label{App:Flowing_datasets_experiments}

Code for the experiments can be found at \url{https://github.com/RohanMShenoy/L2-over-Wasserstein}

\subsubsection{3 ring dataset: Empirical representation}

Consider the 3 ring dataset used in \cite{bonet2025flowing}, where the hierarchical sampling is a finite mixture of 3 Diracs, $\Q = \frac{1}{3}\sum_{c=1}^3 \delta_{\nu^{c,n}}$, where each latent ring class is chosen with uniform probability, each $\nu^c$ corresponds to a uniform measure within the individual rings, and $\nu^{c,n}$ is an empirical measure of $n$ points sampled from each ring. A schematic is shown in Figure \ref{fig:ring_sampling}.

\begin{figure}[h!]
    \centering
    \includegraphics[width=0.9\linewidth]{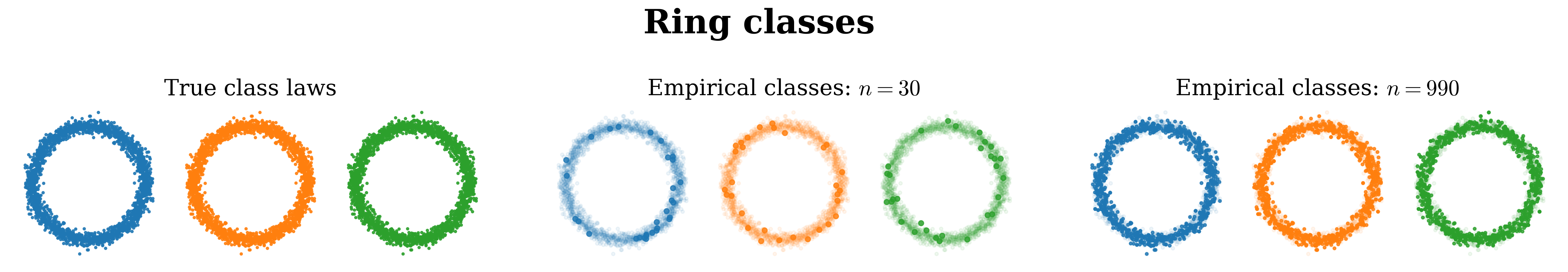}
    \caption{Hierarchical sampling from the 3 ring dataset. }
    \label{fig:ring_sampling}
\end{figure}

Consider the case where the true distribution over the ring classes (i.e. the uniform distribution) is known, and one approximates each class conditional ring measure $\nu^c$ with an empirical measure $\nu^{c,n}$. Figure \ref{fig:ring_fixed_c} illustrates the convergence in WoW distance of the empirical representation.

\begin{figure}[h!]
    \centering
    \includegraphics[width=0.9\linewidth]{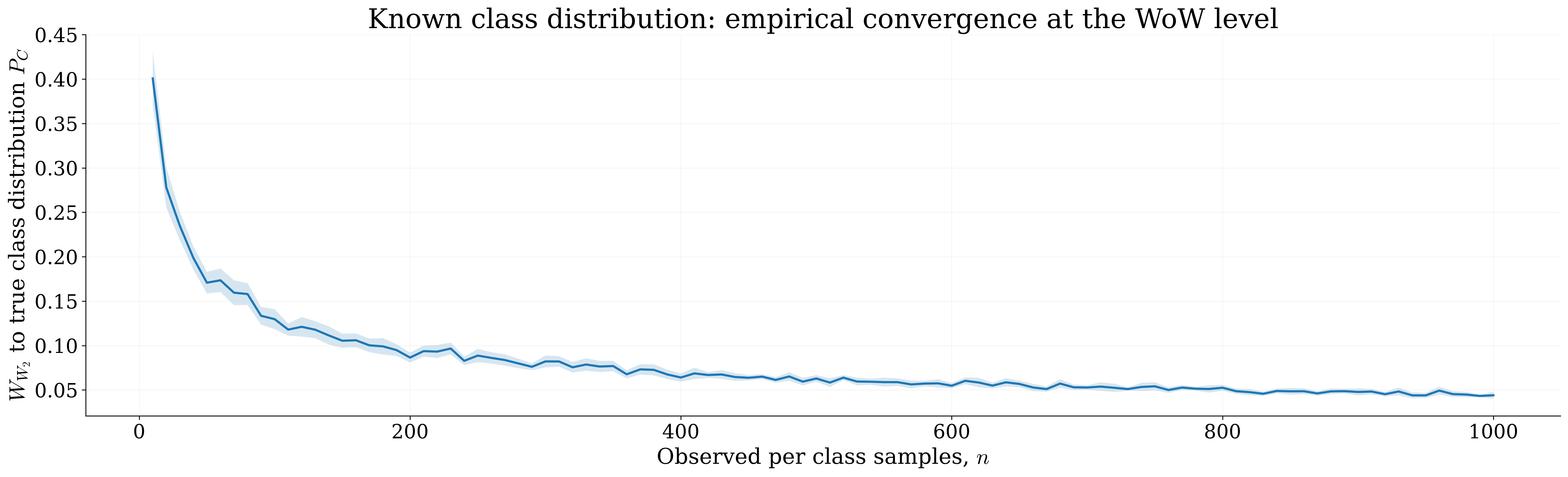}
    \caption{Illustration of the convergence of the empirical representation of the labelled 3 ring dataset. $W_{W_2}\l(P_c,P\r)$ loss with $95\%$ confidence interval from 16 samples.}
    \label{fig:ring_fixed_c}
\end{figure}

However, as mentioned, the framework we have developed allows one to say more, considering the case where the class distribution is \textit{unknown}. In the setting above, consider also sampling each ring class uniformly, before then sampling from the class conditional ring distributions. This is in accordance with the hierarchical sampling scheme in Section \ref{subsec:6:flowing_datasets}. Then, if we sample an increasing number of points $n$ from each sampled class conditional distribution $\nu_{c}$, we observe convergence in accordance with Proposition \ref{prop:Flowing_empirical}, as shown in Figure \ref{fig:ring_growing_c}.

\begin{figure}[h!]
    \centering
    \includegraphics[width=0.9\linewidth]{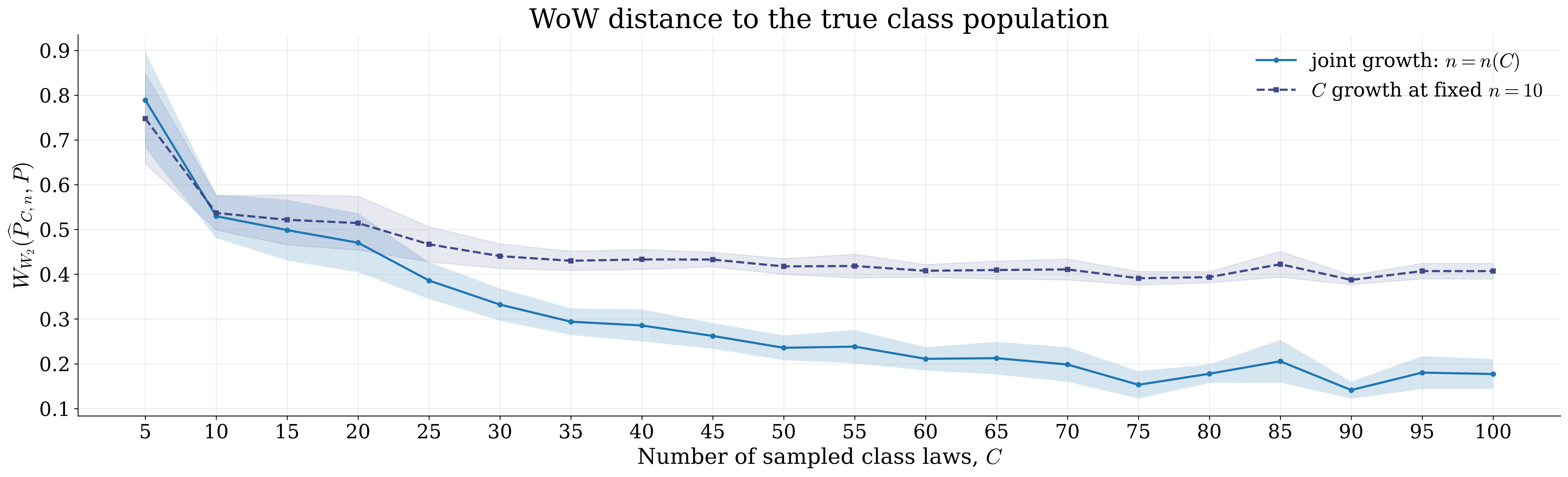}
    \caption{Illustration of the convergence in Proposition \ref{prop:Flowing_empirical} of the empirical representation of the labelled 3 ring dataset where the classes are unknown and sampled randomly. $W_{W_2}\l(P_{c,n},P\r)$ loss with $95\%$ confidence interval from 16 samples.}
    \label{fig:ring_growing_c}
\end{figure}

\subsubsection{Synthetic ring dataset: Bayesian hyperprior}

Returning to the synthetic ring dataset setup, we modify the true distribution to allow for a continuous density over which we can apply the Bayes update. We consider a ring class conditional distribution sampled uniformly from $[-L,L]$ with radius also sampled between $[R-r,R+r]$ using a truncated normal distribution, and uniform distribution over the angle.

Consider placing a hyperprior $\cH$ over the unknown WoW-level population measure $\Q\in\Pkpk$, as in the setting of Proposition \ref{prop:flowing_bayes}, where $\cK$ is a compact rectangle containing the origin and the support of the rings.  We consider both a parametric and non-parametric hyperprior over the hierarchical distribution.

\begin{itemize}
    \item Parametric: Consider an exponential tilt induced by a parameter $\theta$ with density
\begin{equation*}
    g_\theta(a)\propto\exp\{\theta(a/L)^2\},\qquad -L\leq a\leq L,
\end{equation*}
where a normal prior is placed over $\theta$, the truth corresponding to a continuous uniform distribution over $[-L,L]$, where $\theta^* = 0$.
\item Non-parametric: Consider a Dirichlet process hyperprior. Let 
\begin{equation*}
    w\sim\mathrm{Dirichlet}(\alpha H_1,\ldots,\alpha H_K),
\qquad P_w=\sum_{j=1}^K w_j\delta_{\mu_{a_j}}.
\end{equation*}

This is a finite-support approximation to a Dirichlet-process hyperprior on the family of ring laws. The base weights $H_j\propto\exp\{-a_j^2/(2s^2)\}$ favour centres located around $0$, while the true centre law is the uniform. We introduce a latent centre index $z_c$ for each class. Given the class likelihoods $\ell_{cj}=\log p(D_c\mid a_j)$, a Gibbs sampler alternates

\begin{equation*}
    z_c\mid w,D_c\sim\mathrm{Categorical}
\l(\frac{w_j e^{\ell_{cj}}}{\sum_k w_ke^{\ell_{ck}}}\r),
\quad
w\mid z\sim\mathrm{Dirichlet}(\alpha H_j+N_j).
\end{equation*}
\end{itemize}

The content of Proposition \ref{prop:flowing_bayes} states that the hyperprior eventually concentrates its mass on any WoW ball around the truth $\Q$, a phenomenon which is illustrated in Figure \ref{fig:bayes_mass} for both hyperpriors.

\begin{figure}[h!]
    \centering
    \includegraphics[width=0.9\linewidth]{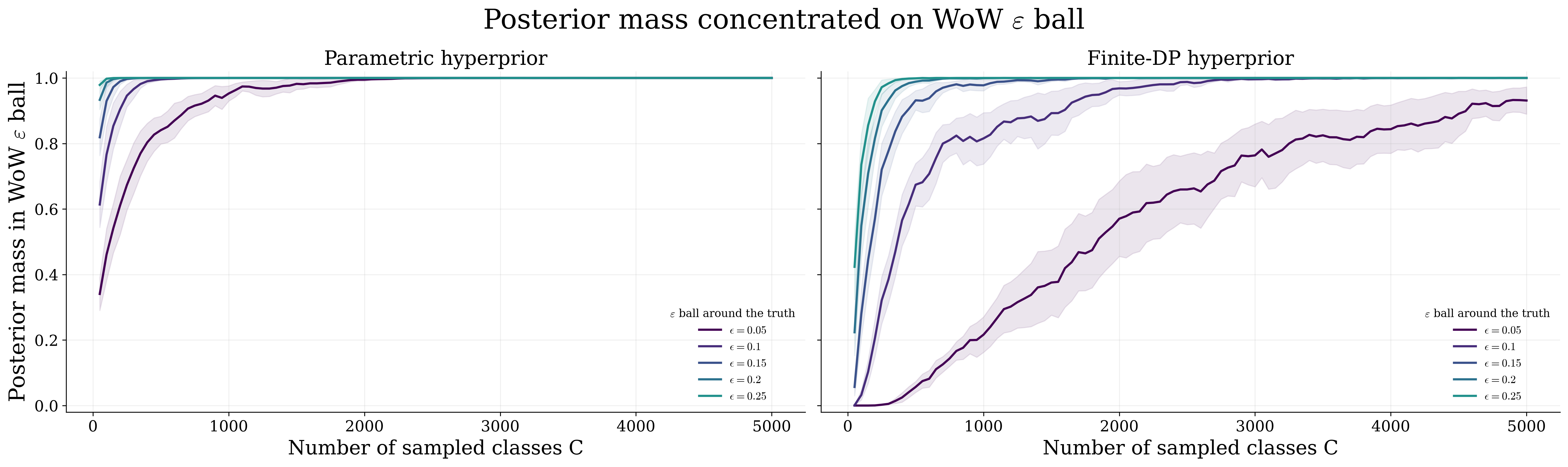}
    \caption{Illustration of the posterior consistency of the hierarchical WoW representation with the modified synthetic ring dataset. The probabilistic concentration about the truth $\Q$, $\cH_n\l(B_{\varepsilon}(\Q)\r)$ (where $B_{\varepsilon}(\Q)\subset \Pkpk$) is shown with $95\%$ confidence interval from 16 samples. The parametric hyper-posterior concentrates its mass on shrinking WoW $\varepsilon$ balls at a very fast rate. The DP, being non-parametric and far more flexible, takes significantly more time, but does eventually concentrate its mass around the truth too.}
    \label{fig:bayes_mass}
\end{figure}

\subsection{Proof of Theorem \ref{thm:cot_projection_wl} (Projection of COT flow matching onto $\WL$)}\label{App:thm:cot_projection_wl}
\begin{proof}
$L^2(\xi_t)$ is a Hilbert space equipped with the inner product, and $\cT_{\xi_t} \WL$ from Definition \ref{def:Lifted_tangent} is a closed subspace of $L^2(\xi_t)$. By the Hilbert projection theorem, any vector field $V_\theta \in L^2(\xi_t)$ admits a unique orthogonal decomposition: $V_\theta = V_\theta^\parallel + V_\theta^\perp$ where $V_\theta^\parallel \in \cT_{\xi_t} \WL$ is the projection onto the tangent space (the irrotational component), and $V_\theta^\perp \in (\cT_{\xi_t} \WL)^\perp$ its orthogonal complement. 

By the definition of the orthogonal complement in $L^2(\xi_t)$, $V_\theta^\perp$ satisfies $\langle V_\theta^\perp, \nabla \varphi \rangle_{\xi_t} = 0$ for all valid test functions $\varphi$. Integrating by parts yields $\nabla_y \cdot (V_\theta^\perp \mu_{t,x}) = 0$ in the weak sense, so that $V_\theta^\perp$ is the divergence-free component with respect to the measure $\mu_{t,x}$. By the LOTUS, we find
$$\cL_t(\theta) = \E_{x \sim P_X} \l[ \int_{\R^d} \l\| V_\theta(x, y) - V^*(x, y) \r\|^2 \, \mu_{t,x}(\d y) \r] = \| V_\theta - V^* \|_{\xi_t}^2.$$
Finally, since $V^* \in \cT_{\xi_t} \WL$ (the displacement interpolation provides a gradient flow in $\WL$), we use orthogonality to conclude,
$$\cL_t(\theta) = \l\| V_\theta^\parallel + V_\theta^\perp - V^* \r\|_{\xi_t}^2 = \l\| \l(V_\theta^\parallel - V^*\r) + V_\theta^\perp \r\|_{\xi_t}^2 = \l\| V_\theta^\parallel - V^* \r\|_{\xi_t}^2 + \l\| V_\theta^\perp \r\|_{\xi_t}^2.$$
\end{proof}

\subsection{Proof of Proposition \ref{prop:GS_vector_cot} (Stability under vector field perturbation)}\label{App:prop:GS_vector_cot}
\begin{proof}
Set $f(t)\deq\E\|X_t-Y_t\|^2$; $f(0)=0$. Then
\begin{align*}
\dot f(t) &= 2\,\E\l[\big\langle X_t-Y_t,\, v_t^*(X_t)-v_t(Y_t)\big\rangle\r] \\&= 2\,\E\l[\big\langle X_t-Y_t,\, v_t^*(X_t)-v_t(X_t)\big\rangle\r] + 2\,\E\l[\big\langle X_t-Y_t,\, v_t(X_t)-v_t(Y_t)\big\rangle\r].
\end{align*}
The second term is bounded by $2L\,f(t)$ by the Lipschitz assumption. For the first, Cauchy–Schwarz and Young's inequality provide
\begin{equation*}
2\,\E\l[\langle X_t-Y_t, v_t^*(X_t)-v_t(X_t)\rangle\r] \leq f(t) + \E\l\|v_t^*(X_t)-v_t(X_t)\r\|^2 = f(t) + \l\|v_t^*-v_t\r\|^2_{\mu_t},
\end{equation*}
using $X_t\sim\mu_t$. Hence $\dot f(t) \leq (1+2L)f(t) + \l\|v_t-v_t^*\r\|^2_{\mu_t}$, and Grönwall's inequality with $f(0)=0$ gives the bound.
\begin{equation*}
f(t) \le \int_0^t e^{(1+2L)(t-s)}\|v_s-v_s^*\|^2_{\mu_s}\,ds \le e^{(1+2L)t}\int_0^t\|v_s-v_s^*\|^2_{\mu_s}\,ds.
\end{equation*}
Since $(X_t,Y_t)$ is a (generally suboptimal) coupling of $(\mu_t,\hat\mu_t)$, $W_2^2(\mu_t,\hat\mu_t)\le f(t)$, giving the bound \eqref{eq:gronwall_stability}. If $v_t$ provides an approximation to $v_t^*$, the condition $\|v_t-v_t^*\|_{\mu_t}\to0$ uniformly in $t$ is exactly the convergence of the KE norm from the definition of the Wasserstein tangent space \eqref{def:w2_tangent}, hence convergence of the deterministic flows. Considering these vector fields as sample realisations of some random vector field (in the sense of Theorem \ref{thm:cot_projection_wl}), the convergence can be lifted to $\WL$, and we conclude.
\end{proof}

\subsection{Proof of Proposition \ref{prop:discrete_interaction_flow_stability} (Self-attention propagation of chaos)}\label{App:prop:discrete_interaction_flow_stability}
\begin{proof}
    For ease of notation, let $h_\mu(x) = \int e^{\beta\langle x,y\rangle} y\,\mu(\d y)$ with $v_\mu(x) = \proj_x(h_\mu(x))$. Let $\mu,\nu\in\Pc$, and let $\gamma\in\Gamma_{\mu,\nu}$ be an optimal coupling so that
\begin{equation*}
    h_{\mu}(x) - h_{\nu}(x) = \iint (e^{\beta\langle x,y\rangle}y - e^{\beta\langle x,y'\rangle}y')\gamma(\d y,\d y').
\end{equation*}
Considering the integrand, the triangle inequality provides
\begin{align*}
    \l\|e^{\beta\langle x,y\rangle}y - e^{\beta\langle x,y'\rangle}y'\r\|\leq \l|e^{\beta\langle x,y\rangle} - e^{\beta\langle x,y'\rangle}\r|\|y\| + e^{\beta\langle x,y'\rangle}\|y-y'\| \leq (\beta+1)e^\beta \|y-y'\|
\end{align*}
where we have used the mean-value theorem to bound the first difference. Using the bounds on the flow,
\begin{equation*}
    \|v_\mu(x) - v_\nu(x)\|\leq \|h_\mu(x) - h_\nu(x)\| \leq (\beta+1)e^\beta \int \|y-y'\|\gamma(\d y,\d y')\leq (\beta+1)e^\beta W_2(\mu,\nu),
\end{equation*}
where the last inequality follows from Cauchy-Schwarz. Thus the first Lipschitz constant $L_1 = (\beta + 1)e^\beta$.\newline

Working towards the other Lipschitz constant, we first decompose
\begin{equation*}
    v_\mu(x) - v_\nu(x') = \proj_x(h_\mu(x)-h_\nu(x')) + (\proj_x - \proj_{x'})h_{\mu}(x').
\end{equation*}
Again from the mean-value theorem, the first term $h_\mu(x)-h_\nu(x')\leq \beta e^\beta\|x-x'\|$, and as the projection retains the inequality here,
\begin{equation*}
    \l\|\proj_x(h_\mu(x)-h_\nu(x'))\r\| \leq \beta e^\beta\l\|x-x'\r\|.
\end{equation*}
For the second term, as $\|h_\mu(x')\|\leq e^\beta$, using the inequality on the projection
\begin{equation*}
    \l\|(\proj_x - \proj_{x'})h_{\mu}(x')\r\|\leq 2e^\beta \l\|x-x'\r\|.
\end{equation*}
Thus, combining the two, we find $L_2 = (\beta+2)e^\beta$. Combining the constants as in the proof of Proposition \ref{prop:WGF_lifted_stability} (Appendix \ref{App:prop:WGF_lifted_stability}), we find $L = L_1^2 + 1 + 2L_2$, so that we may conclude $L = (\beta + 1)^2e^{2\beta} + 1 + 2(\beta + 2)e^\beta.$
\end{proof}\label{Appendix}

\end{document}